\documentclass[11pt,a4paper]{article}
\usepackage[margin=1in]{geometry}

\usepackage{titling}

\usepackage{amsfonts,amsmath,amsthm} 
\usepackage{mathrsfs} 
\usepackage{stmaryrd} 
\usepackage{mathtools}

\usepackage{natbib}
\usepackage{url}

\usepackage{tikz}

\usepackage{color}

\allowdisplaybreaks[2]

\newtheorem{theorem}{Theorem}
\newtheorem{proposition}[theorem]{Proposition}
\newtheorem{lemma}[theorem]{Lemma}
\newtheorem{corollary}[theorem]{Corollary}

\theoremstyle{remark}
\newtheorem{remark}{Remark}

\theoremstyle{definition}
\newtheorem{definition}{Definition}

\newcommand\EatDot[1]{}

\newcommand{\bX}{X}
\newcommand{\bY}{Y}
\newcommand{\bZ}{Z}

\newcommand{\be}{e}
\newcommand{\bsf}{f}

\newcommand{\bu}{u}
\newcommand{\bv}{v}

\newcommand{\bx}{x}
\newcommand{\by}{y}
\newcommand{\bz}{z}

\newcommand{\bepsilon}{\epsilon}
\newcommand{\bvarepsilon}{\varepsilon}
\newcommand{\bzeta}{\zeta}

\newcommand{\btheta}{\theta}

\newcommand{\bxi}{\xi}

\newcommand{\bfA}{A}
\newcommand{\bfB}{B}
\newcommand{\bfC}{C}

\newcommand{\bfI}{I}
\newcommand{\bfJ}{J}

\newcommand{\bfM}{M}

\newcommand{\bfO}{O}
\newcommand{\bfP}{P}

\newcommand{\bfZ}{Z}

\newcommand{\bSigma}{\Sigma}

\newcommand{\bbS}{\mathbb{S}}

\newcommand{\calL}{\mathcal{L}}

\newcommand{\N}{\mathbb{N}}

\newcommand{\R}{\mathbb{R}}
\newcommand{\Z}{\mathbb{Z}}

\newcommand{\E}{\mathbb{E}}
\renewcommand{\Pr}{\mathbb{P}}
\newcommand{\Cov}{\textnormal{Cov}}

\newcommand{\Gauss}{\mathcal{N}}

\newcommand{\tr}{\textnormal{Tr}}
\newcommand{\diag}{\textnormal{diag}}

\newcommand{\diff}{\textnormal{d}}

\renewcommand{\epsilon}{\varepsilon}

\newcommand{\zero}{\mathbf{0}}
\DeclareMathOperator*{\argmin}{arg\,min}

\newcommand{\Ent}{\textnormal{Ent}}

\newcommand{\er}{\mathbf{r}}

\usepackage{newtxtext,newtxmath}

\newtheorem{assumption}{Assumption}
\theoremstyle{definition}
\newtheorem{example}{Example}

\newcommand{\mF}{\mathcal{F}}

\newcommand{\mS}{\mathcal{S}}

\newcommand{\Ep}{\mathbb{E}}
\renewcommand{\Pr}{\mathbb{P}}

\newcommand{\mE}{\mathcal{E}}

\renewcommand{\hat}{\widehat}
\renewcommand{\tilde}{\widetilde}

\newcommand{\Surv}{S}

\DeclarePairedDelimiter{\abs}{\lvert}{\rvert}
\DeclarePairedDelimiter{\bigabs}{\Bigg\lvert}{\Bigg\rvert}

\newcommand{\rc}{\color{black}}

\usepackage{bibunits}

\usepackage{url}
\usepackage[justification=centering]{caption}

\title{Corrigendum to ``Dimension-free Bounds for Sums of Dependent Matrices and Operators with Heavy-Tailed Distributions''}
\author{Shogo Nakakita\textsuperscript{1}, Pierre Alquier\textsuperscript{2}, and Masaaki Imaizumi\textsuperscript{1,3}\vspace{2ex}\\
{\it \textsuperscript{1}The University of Tokyo, \textsuperscript{2}ESSEC Business School,}\\
{\it \textsuperscript{3}RIKEN Center for Advanced Intelligence Project}}

\begin{document}


\date{}
\maketitle
\begin{bibunit}[apalike]

\begin{abstract}
We correct Theorem 4 of \citet{nakakita2024dimension} by introducing a log-Sobolev inequality in place of the the boundedness condition.
We show that the examples discussed in \citet{nakakita2024dimension} can be recovered under new conditions.
The original, uncorrected version---which includes the aforementioned error---is appended after this corrigendum for transparency and comparison.
\end{abstract}

\section{Introduction}
\citet{nakakita2024dimension} considered dimension-free concentration bounds on random matrices, such as sample covariance matrices, under dependence.
The proof of Theorem 4, however, contains an error.
In the proof,
we defined $u,v\in\Sigma^{1/2}\mathbb{S}^{p-1}$ on page 1148.
It holds
that, with probability $1-e^{-t}$, for any \underline{$u,v\in\Sigma^{1/2}\bbS^{p-1}$} simultaneously,
\begin{align*}
        u^{\top}\left(\frac{1}{n}\sum_{i=1}^{n}M_{i}-\Sigma\right)v\le 4\sqrt{2} \|\Sigma\|(\kappa^{2}+\Gamma_{n})\sqrt{\frac{4\er(\Sigma) + t}{n}}.
    \end{align*}
Using this result, \citet{nakakita2024dimension} incorrectly concluded that with probability $1-e^{-t}$,
    \begin{align*}
        \left\|\frac{1}{n}\sum_{i=1}^{n}M_{i}-\Sigma\right\|\le 4\sqrt{2} \|\Sigma\|(\kappa^{2}+\Gamma_{n})\sqrt{\frac{4\er(\Sigma) + t}{n}}.
    \end{align*}
    Since $u,v$ are not elements of $\bbS^{p-1}$ but rather of $\bSigma^{1/2}\bbS^{p-1}$, the dimension-free bound which we truly obtain is that with probability $1-e^{-t}$,
    \begin{align*}
        \left\|\Sigma^{1/2}\left(\frac{1}{n}\sum_{i=1}^{n}M_{i}-\Sigma\right)\Sigma^{1/2}\right\|\le 4\sqrt{2} \|\Sigma\|(\kappa^{2}+\Gamma_{n})\sqrt{\frac{4\er(\Sigma) + t}{n}}.
    \end{align*}
Hence, the applicability of the derived bound becomes limited; for instance, benign overfitting discussed in \citet{nakakita2024dimension} may not hold true.

In this corrigendum, we correct Theorem 4 of \cite{nakakita2024dimension} by showing the concentration of sample covariance matrices under a log-Sobolev inequality instead of boundedness and mixing (Theorem \ref{thm:concentration}); we also present a sufficient condition for the concentration of general random matrices, inspired by Theorem \ref{thm:concentration} and its proof, as a remark (Remark \ref{remark:general}).
By this correction, we can substantially recover arguments of \citet{nakakita2024dimension} as they are, except for the geometric structures of observation noises and driving noises.
Notably, we can show the concentration of sample covariance matrices of causal Bernoulli shift models by essentially replacing a condition of boundedness for noise with a log-Sobolev inequality condition (Figure \ref{fig:diagram}).

\begin{figure}
    \centering
    \begin{tikzpicture}
        \draw[rounded corners] (0, 0) rectangle (7, 2);
        \draw[rounded corners] (8, 0) rectangle (15, 2);

        \node[fill=white] at (2,2) {\citet{nakakita2024dimension}};
        \node[fill=white] at (9.7,2) {This corrigendum};

        \node[text width=6cm] at (3.25,0.9) {
        $\bX_{i}=f(\bzeta_{i},\bzeta_{i-1},\bzeta_{i-2},\ldots)$,\\
        where $f$ is Lipschitz, $\bzeta_{i}$ are i.i.d., and \\
        $\bzeta_{i}$ satisfy \underline{boundedness}};
        
        \node[text width=6cm] at (11.25,0.9) {
        $\bX_{i}=f(\bzeta_{i},\bzeta_{i-1},\bzeta_{i-2},\ldots)$,\\
        where $f$ is Lipschitz, $\bzeta_{i}$ are i.i.d., and \\
        $\bzeta_{i}$ satisfy \underline{a log-Sobolev inequality}};

        \draw[->, ultra thick, dotted] (7.25, 1) -- (7.75, 1);
    \end{tikzpicture}
    \caption{A schematic diagram of the correction. The main example of random matrices in \citet{nakakita2024dimension} is the sample covariance matrix of Lipschitz functions with bounded independent inputs.
    We correct the boundedness assumption by replacing it with log-Sobolev inequalities.}
    \label{fig:diagram}
\end{figure}

\subsection{Notation}
We introduce the notation used in this corrigendum, which is also used in \citet{nakakita2024dimension}.
We set the Euclidean sphere $\bbS^{d-1}:=\{\bx\in\R^{d};\sum_{i=1}^{d}\bx_{i}^{2}=1\}$.
For any matrix $\bfM\in\R^{d_{1}\times d_{2}}$, we let $\|\bfM\|=\sup_{\bu\in\bbS^{d_{1}-1},\bv\in\bbS^{d_{2}-1}}\bu^{\top}\bfM\bv$, which is the operator norm of $\bfM$ as a map from $\R^{d_{2}}$ to $\R^{d_{1}}$, with both spaces equipped with the $\ell^{2}$-norm.
For any positive semi-definite matrix $\bfM$, we define the effective rank 
\begin{equation*}
    \er(\bfM)=\begin{cases}
        \tr(\bfM)/\|\bfM\| & \text{ if }\|\bfM\|>0,\\
        0 & \text{ otherwise.}
    \end{cases}
\end{equation*}
For $p\times p$ positive semi-definite $\bfA$, let $\bfA^{\dagger}$ denote the Moore--Penrose inverse of $\bfA$.

\section{Correction to the main theorem under a log-Sobolev inequality}

We present a corrected result for the concentration of the empirical mean of dependent and heavy-tailed random matrices.

As preparation, we present a key concept that is important for the correction.
A probability distribution $P$ on $\R^{d}$ is said to satisfy a \textit{log-Sobolev inequality} with a constant $C>0$ if for any locally Lipschitz $f:\R^{d}\to\R$, it holds that
\begin{equation*}
    \Ent_{P}(f^{2})\le 2C\int \left\|\nabla f\right\|^{2}\diff P,
\end{equation*}
where $\Ent_{P}(g):=\int (g\log g)\diff P-\int g\diff P\log(\int g\diff P)$ for any non-negative $g$, and $\nabla$ is the gradient operator on $\R^{d}$.

We present the corrected statement as follows.

\begin{theorem}[dimension-free concentration of sample covariance matrices]\label{thm:concentration}
    Let $(\bX_{i})_{i=1}^{n}$ be a sequence of $p$-dimensional random vectors such that $\E[\bX_{i}\bX_{i}^{\top}]=\bSigma$ for all $i=1,\ldots,n$, and suppose that the distribution of the $\R^{pn}$-valued random vector $((\bSigma^{\dagger})^{1/2}\bX_{1},\ldots,(\bSigma^{\dagger})^{1/2}\bX_{n})$ satisfies a log-Sobolev inequality for constant $K$.
    Then, for any $t\ge0$ and positive integer $n$ satisfying $n\ge 2K(9\er(\bSigma)+4t)$, with probability at least $1-\exp(-t)$,
    \begin{equation*}
        \left\|\frac{1}{n}\sum_{i=1}^{n}\bX_{i}\bX_{i}^{\top}-\bSigma\right\|\le  12\sqrt{K}\|\bSigma\|\sqrt{\frac{9\er(\bSigma)+4t}{n}}.
    \end{equation*}
\end{theorem}

A key modification in this theorem is the assumption that the $pn$-dimensional random vector $((\bSigma^{\dagger})^{1/2}\bX_{1},\ldots,(\bSigma^{\dagger})^{1/2}\bX_{n})$ satisfies a log-Sobolev inequality.
While \citet{nakakita2024dimension} separately deal with the tail behaviour and mixing property of random vectors, the log-Sobolev constant simultaneously captures both tail behavior and dependence.
Hence, we can regard a log-Sobolev inequality as a unification of those assumptions in \citet{nakakita2024dimension}.

\begin{remark}
    The setting of \citet{zhivotovskiy2024dimension} corresponds to the case where (i) the independence of $\{\bX_{i}\}$ and (ii) 
    the subgaussianity of $\bv^{\top}(\bSigma^{\dagger})^{1/2}\bX_{i}$
    for all $\bv\in\bbS^{p-1}$ (i.e., arbitrary linear maps from $\R^{p}$ to $\R$) hold. 
    Alternatively, we allow the dependence of $\{\bX_{i}\}$ but assume log-Sobolev inequalities inducing the subgaussianity of $f((\bSigma^{\dagger})^{1/2}\bX_{1},\ldots,(\bSigma^{\dagger})^{1/2}\bX_{n})$ for arbitrary $1$-Lipschitz maps $f:\R^{np}\to\R$, which is stronger than (ii) above.
\end{remark}

In Section \ref{corr:sec:example}, we see examples of $\bX_{i}$ and observe that $K$ reflects the magnitude of the dependence among $\{\bX_{i}\}$.

\subsection{Proof of Theorem \ref{thm:concentration}}
We prepare an exponential integrability result inspired by the celebrated Herbst's argument \citep{ledoux1999concentration,bakry2014analysis}.
\begin{lemma}[exponential integrability of quadratic forms]\label{lemma:expmoment}
    Suppose that $\bfA$ is a $d\times d$ positive semi-definite matrix and $\bZ$ is a $d$-dimensional random vector whose distribution $\nu$ satisfies a log-Sobolev inequality with constant $K$.
    Then, for any $\tau\in(-\infty,1/(2\|\bfA\|K))$,
    \begin{equation*}
        \E[\exp(\tau\bZ^{\top}\bfA\bZ)]\le \exp\left(\tau\E[\bZ^{\top}\bfA\bZ]+\frac{2\|\bfA\|K\tau^{2}}{1-2\|\bfA\|K\tau}\E[\bZ^{\top}\bfA\bZ]\right).
    \end{equation*}
\end{lemma}

\begin{proof}
We divide the proof into the cases $\tau > 0$ and $\tau < 0$ (the statement with $\tau=0$ is obvious).

(Step 1: $\tau>0$)
Define $f(\bz)=\bz^{\top}\bfA\bz$ and $M_{+}(s)=\int e^{sf}\diff\nu$ for $s\ge0$ (the moment-generating function of $f(\bZ)$).
We have
\begin{equation*}
    \Ent( e^{sf})=\E[ e^{sf}sf]-\E[e^{sf}]\log\E[e^{sf}]=sM_{+}'(s)-M_{+}(s)\log M_{+}(s),
\end{equation*}
and the positive semi-definiteness of $\bfA$ yields
\begin{align*}
    \E\left[\|\nabla e^{sf/2}\|^{2}\right]&=\frac{s^{2}}{4}\int e^{sf(\bz)}\|\nabla f(\bz)\|^{2}\nu(\diff\bz)= s^{2}\int e^{sf(\bz)}\|\bfA\bz\|^{2}\nu(\diff\bz)\\
    &\le \|\bfA\|s^{2}\int e^{sf(\bz)}f(\bz)\nu(\diff\bz)= \|\bfA\|s^{2}M_{+}'(s).
\end{align*}
Since a log-Sobolev inequality holds true, we have
\begin{equation*}
    sM_{+}'(s)\le M_{+}(s)\log M_{+} (s)  +2\|\bfA\|Ks^{2}M_{+}'(s)
\end{equation*}
and thus 
\begin{equation*}
    s(1-2\|\bfA\|Ks)M_{+}'(s)\le M_{+}(s)\log M_{+}(s).
\end{equation*}
Let us define $F_{+}(s)=(1/s)\log M_{+}(s)$ (the cumulant-generating function of $f(\bZ)$ scaled by $1/s$; let $F_{+}(0):=\E[f(\bZ)]$, and then $F_{+}(s)$ is continuous since $\lim_{s\downarrow0}F_{+}(s)=\lim_{s\downarrow0}(1/s)(\log M_{+}(s)-\log M_{+}(0))=M_{+}'(0)/M_{+}(0)$) and $b:=2\|\bfA\|K$.
Then
\begin{align*}
    F_{+}'(s)&=-\frac{\log M_{+}(s)}{s^{2}}+\frac{1}{s}\frac{M_{+}'(s)}{M_{+}(s)}\\
    &\le -\frac{\log M_{+}(s)}{s^{2}}+\frac{\log M_{+}(s)}{s^{2}(1-bs)}=\left(\frac{1}{s(1-bs)}-\frac{1}{s}\right)F_{+}(s)=\left(\frac{b}{1-bs}\right)F_{+}(s).
\end{align*}
Then Gr\"{o}nwall's inequality yields
\begin{align*}
    F_{+}(s)&\le F_{+}(0)\exp\left(\int_{0}^{s}\frac{b}{1-bu}\diff u\right)=\E[f(\bZ)]\exp\left(\int_{0}^{bs}\frac{\diff v}{1-v}\right)=\E[f(\bZ)]\exp\left(\int_{1-bs}^{1}\frac{\diff v'}{v'}\right)\\
    &=\frac{\E[f(\bZ)]}{1-bs}.
\end{align*}
Therefore, we have
\begin{equation*}
    (1/s)\log M_{+}(s)\le \E[f(\bZ)]\frac{1}{1-bs},
\end{equation*}
and scaling with $s$ yields
\begin{equation*}
    \log M_{+}(s)\le s\E[f(\bZ)]\frac{1}{1-bs},
\end{equation*}
and the monotonicity of $\exp(\cdot)$ along with $1/(1-bs)=1+bs/(1-bs)$ leads to
\begin{equation*}
    M_{+}(s)\le \exp\left(s\E[f(\bZ)]\frac{1}{1-bs}\right)=\exp\left(s\E[f(\bZ)]-\E[f(\bZ)]\frac{bs^{2}}{1-bs}\right).
\end{equation*}
Hence, we have
\begin{equation*}
    \E[\exp(sf(\bZ))]\le \exp\left(s\E[f(\bZ)]+\frac{bs^{2}}{1-bs}\E[f(\bZ)]\right).
\end{equation*}

(Step 2: $\tau<0$) Let us define $g(\bz)=-\bz^{\top}\bfA\bz$ and $M_{-}(s)=\int e^{sg}\diff\nu(=\int e^{-sf}\diff\nu)$ for $s\ge0$.
Then, the positive semi-definiteness of $\bfA$ yields
\begin{align*}
    \E\left[\|\nabla e^{sg/2}\|^{2}\right]&=\frac{s^{2}}{4}\int e^{sg(\bz)}\|\nabla g(\bz)\|^{2}\nu(\diff\bz)= s^{2}\int e^{sg(\bz)}\|\bfA\bz\|^{2}\nu(\diff\bz)\\
    &\le \|\bfA\|s^{2}\int e^{sg(\bz)}|g(\bz)|\nu(\diff\bz)=\|\bfA\|s^{2}\int e^{sg(\bz)}(-g(\bz))\nu(\diff\bz)\\
    &=-\|\bfA\|s^{2}\int e^{sg(\bz)}g(\bz)\nu(\diff\bz)= -\|\bfA\|s^{2}M_{-}'(s).
\end{align*}
Therefore, a log-Sobolev inequality with constant $K$ leads to
\begin{equation*}
    sM_{-}'(s)\le M_{-}(s)\log M_{-}(s) -2\|\bfA\|Ks^{2}M_{-}'(s).
\end{equation*}
Define $F_{-}(s)=(1/s)\log M_{-}(s)$ (with $F_{-}(0)=\E[g(\bZ)]$) and $b:=2\|\bfA\|K$.
Then,
\begin{align*}
    F_{-}'(s)&=-\frac{1}{s^{2}}\log M_{-}(s)+\frac{M_{-}'(s)}{sM_{-}(s)}\\
    &\le-\frac{1}{s^{2}}\log M_{-}(s)+\frac{1}{s^{2}(1+bs)}\log M_{-}(s)=F_{-}(s)\left(\frac{1}{s(1+bs)}-\frac{1}{s}\right)=-F_{-}(s)\left(\frac{b}{1+bs}\right).
\end{align*}
Gr\"{o}nwall's inequality yields
\begin{align*}
    F_{-}(s)&\le F_{-}(0)\exp\left(-\int_{0}^{s}\frac{b}{1+bt}\diff t\right)=F_{-}(0)\exp\left(-\int_{1}^{1+bs}\frac{1}{u}\diff u\right)=F_{-}(0)\exp(-\log(1+bs))\\
    &=F_{-}(0)\frac{1}{1+bs}=\E[g(\bZ)]\frac{1}{1+bs},
\end{align*}
where the last identity is by $F_{-}(0)=\E[g]$.
Hence, we have
\begin{equation*}
    (1/s)\log M_{-}(s)\le \E[g(\bZ)]\frac{1}{1+bs},
\end{equation*}
and multiplying both sides by $s$ yields
\begin{equation*}
    \log M_{-}(s)\le s\E[g(\bZ)]\frac{1}{1+bs}.
\end{equation*}
The monotonicity of $\exp(\cdot)$ yields
\begin{equation*}
    M_{-}(s)\le \exp\left(s\E[g(\bZ)]\frac{1}{1+bs}\right)=\exp\left(s\E[g(\bZ)]-\E[g(\bZ)]\frac{bs^{2}}{1+bs}\right).
\end{equation*}
Then, we obtain
\begin{equation*}
    \E\left[\exp\left(sg(\bZ)\right)\right]\le \exp\left(s\E[g(\bZ)]-\E[g(\bZ)]\frac{bs^{2}}{1+bs}\right).
\end{equation*}
As we set $\tau=-s$, we derive the statement.
\end{proof}

We also present a basic result in linear algebra.
\begin{lemma}\label{lem:norm}
    For a $d\times d$ symmetric matrix $\bfA$, set a $(dn)\times(dn)$ matrix $\bfA_{n}=\diag\{\bfA,\ldots,\bfA\}$. Then $\|\bfA_{n}\|= \|\bfA\|$.
\end{lemma}

\begin{proof}
    Since $\bfA_{n}=\bfI_{n}\otimes \bfA$ ($\otimes$ denotes the Kronecker product), the eigenvalues of $\bfA_{n}$ coincide with those of  $\bfA$ (we ignore multiplicity), and thus we obtain the conclusion.
\end{proof}

We exhibit the proof of Theorem \ref{thm:concentration}; the strategy is adapted from \citet{zhivotovskiy2024dimension}.
\begin{proof}[Proof of Theorem \ref{thm:concentration}]
    We assume $\bSigma$ is invertible; otherwise, we can give the proof by considering subspaces.

    (Step 1: prior and posterior)
    Let $\beta,r>0$ be positive constants to be fixed later. 
    Let $\mu$ denote a $p$-dimensional Gaussian measure with zero mean and covariance $\beta^{-1}\bSigma$.
    Set $\bu\in\bSigma^{1/2}\bbS^{p-1}$ and define $f_{\bu}$ as a probability density function with respect to the $p$-dimensional Lebesgue measure such that
    \begin{align*}
        f_{\bu}(\bx)=\frac{\exp\left(-\frac{\beta}{2}(\bx-\bu)^{\top}\bSigma^{-1}(\bx-\bu)\right)\mathbf{1}\{\|\bx-\bu\|\le r\}}{C_{\beta,r}\sqrt{2\pi}^{p}\sqrt{\det(\beta^{-1}\bSigma)}},
    \end{align*}
    where $C_{\beta,r}>0$ is the normalizing constant; this is a truncated Gaussian distribution whose truncation is given as $\mathbf{1}\{\|\bx-\bu\|\le r\}$.
    We have the following lower bound on $C_{\beta,r}$:
    \begin{equation*}
        C_{\beta,r}=1-\Pr\left(\|\bzeta_{\bu}-\bu\|>r\right)\ge 1-\E\left[\|\bzeta_{\bu}-\bu\|^{2}\right]r^{-2}=1-\frac{\tr(\bSigma)}{\beta r^{2}},
    \end{equation*}
    where $\bzeta_{u}\sim\Gauss(\bu,\beta^{-1}\bSigma)$.
    Let $\btheta_{\bu}$ be a random vectors with density $f_{\bu}$ for $\bu\in\bbS^{p-1}$.
    Since $f_{\bu}(\bx)$ is symmetric about $\bu$, it satisfies $\E[\btheta_{\bu}]=\bu$.
    Moreover, for any $p\times p$ matrix $\bfA$,
    \begin{align*}
        \E[\btheta_{\bu}^{\top}\bfA\btheta_{\bu}]&=\bu^{\top}\bfA\bu+C_{\beta,r}^{-1}\E[(\bzeta_{\bu}-\bu)^{\top}\bfA(\bzeta_{\bu}-\bu)\mathbf{1}\{\|\bzeta_{\bu}-\bu\|\le r\}]\\
        &= \bu^{\top}\bfA\bu+C_{\beta,r}^{-1}\E[\bzeta^{\top}\bfA\bzeta\mathbf{1}\{\|\bzeta\|\le r\}],
    \end{align*}
    where $\zeta\sim \Gauss(\zero,\beta^{-1}\bSigma)$.
    For convenience, let us define
    \begin{align*}
        \bfB_{\beta,r}:=C_{\beta,r}^{-1}\E[\bzeta\bzeta^{\top}\mathbf{1}\{\|\bzeta\|\le r\}].
    \end{align*}
    Here, $\bfB_{\beta,r}$ denotes the covariance of $\btheta_{\bu}$, which is independent of $\bu$.
    We have the following representation:
    \begin{align}
        \E[\btheta_{\bu}^{\top}\bfA\btheta_{\bu}]&= \bu^{\top}\bfA\bu+\tr(\bfA\bfB_{\beta,r})\label{eq:posteriormean:1}.
    \end{align}
    
    Let $\rho_{\bu}$ be a probability measure on $\R^{p}$ induced by $\btheta_{\bu}$; its density is given as $f_{\bu}(\bx)$.
    Let $g$ denote the density of $\mu$, which is set as $\Gauss(\zero,\beta^{-1}\bSigma)$ above.
    The Kullback--Leibler divergence between $\rho_{\bu}$ and $\mu$ is given by
    \begin{align*}
        \textnormal{KL}(\rho_{\bu}\|\mu)&=\int \log\left(\frac{f_{\bu}(\bx)}{g(\bx)}\right)f_{\bu}(\bx)\diff \bx\\
        &=\E_{\rho_{\bu}}\log\left(\frac{1}{C_{\beta,r}}\exp\left(\frac{-\beta(\btheta_{\bu}-\bu)^{\top}\bSigma^{-1}(\btheta_{\bu}-\bu)+\beta\btheta_{\bu}^{\top}\bSigma^{-1}\btheta_{\bu}}{2}\right)\right)\\
        &=\log\left(\frac{1}{C_{\beta,r}}\right)+\E_{\rho_{\bu}}\left[\frac{\beta \bu^{\top}\bSigma^{-1}\btheta_{\bu}+\beta \btheta_{\bu}^{\top}\bSigma^{-1}\bu-\beta\bu^{\top}\bSigma^{-1}\bu}{2}\right]\\
        &=\log\left(\frac{1}{C_{\beta,r}}\right)+\frac{\beta\bu^{\top}\bSigma^{-1}\bu}{2}=\log\left(\frac{1}{C_{\beta,r}}\right)+\frac{\beta}{2}.
    \end{align*}
    where the last line is by $\bu\in\bSigma^{1/2}\bbS^{p-1}$.
    We fix
    \begin{equation*}
        r=\sqrt{2\beta^{-1}\tr(\bSigma)};
    \end{equation*}
    then, we obtain $C_{\beta,r}\ge 1/2$, and thus
    \begin{equation}
        \textnormal{KL}(\rho_{\bu}\|\mu)\le \log2+\frac{\beta}{2}.\label{eq:KL}
    \end{equation}

    (Step 2: PAC-Bayes bound)
    We use the following PAC-Bayes bound (e.g., see \citealp{catoni2004statistical,catoni2017dimension}) with $\mu=\Gauss(\zero,\beta^{-1}\bSigma)$ and $\bX=(\bX_{1},\ldots,\bX_{n})$: for arbitrary $h:\R^{np}\times\R^{p}\to\R$ such that $h$ is bounded above, with probability at least $1-\exp(-t)$, for all probability measures $\rho\ll\mu$ on $\R^{p}$ simultaneously,
    \begin{equation}\label{eq:pacbayes}
        \E_{\btheta\sim \rho}[h(\bX,\btheta)]\le \E_{\btheta\sim \rho}[\log\E_{\bX}(\exp(h(\bX,\btheta)))]+\textnormal{KL}(\rho\|\mu)+t.
    \end{equation}
    
    We let $h_{b}(\bX,\btheta)=\min\{b,\lambda\sum_{i=1}^{n}\btheta^{\top}\bSigma^{-1/2}\bX_{i}\bX_{i}^{\top}\bSigma^{-1/2}\btheta\}$ for $b\in\N$; $b$ is a truncation parameter introduced to ensure that $h_b$ is bounded, and we take the limit $b\to\infty$ later.
    Using (i) the trivial bound $\exp(\min\{b,x\})\le \exp(x)$ for any $x\in\R$ and (ii) Lemma \ref{lemma:expmoment} with 
    \begin{equation*}
        \bfA=\diag\{\btheta\btheta^{\top},\ldots,\btheta\btheta^{\top}\},\ \bZ=(\bSigma^{-1/2}\bX_{1},\ldots,\bSigma^{-1/2}\bX_{n})
    \end{equation*}
    (note $\|\bfA\|\le \|\btheta\|^{2}$ by Lemma \ref{lem:norm} and $\E[\bZ^{\top}\bfA\bZ]\le \|\btheta\|^{2}n$), we derive that for any $b\in\N$ and $\btheta\in\R^{p}$, and $\lambda\in[0,1/(4\|\btheta\|^{2}K)]$,
    \begin{align}
        \E_{\bX}[\exp (h_{b}(\bX,\btheta))]&\le \E_{\bX}\left[\exp\left(\lambda\btheta^{\top} \left(\sum_{i=1}^{n}\bSigma^{-1/2}\bX_{i}\bX_{i}^{\top}\bSigma^{-1/2}\right)\btheta\right)\right]\notag\\
        &\le \exp\left(\lambda n\btheta^{\top}\bSigma^{-1/2}\bSigma\bSigma^{-1/2}\btheta+4K\lambda^{2}\|\btheta\|^{4}n\right).\label{eq:expmoment:1}
    \end{align}
    We have for all $\bu\in\bSigma^{1/2}\bbS^{p-1}$, for $\rho_{\bu}$-almost all $\btheta$,
    \begin{equation*}
        \|\btheta\|^{4}\le (\sqrt{\|\bSigma\|}+r)^{4}=(\sqrt{\|\bSigma\|}+\sqrt{2\beta^{-1}\tr(\bSigma)})^{4}.
    \end{equation*}
    As \citet{zhivotovskiy2024dimension}, we choose $\beta=2\er(\bSigma)$; then $\|\btheta\|^{4}\le 16\|\bSigma\|^{2}$.
    
    Hence, Eqs.~\eqref{eq:posteriormean:1}, \eqref{eq:KL}, \eqref{eq:pacbayes}, and \eqref{eq:expmoment:1} yield that for any $\lambda\in[0,1/(16\|\bSigma\|K)]$ with probability $1-e^{-t}$, for all $\bu\in\bSigma^{1/2}\bbS^{p-1}$ simultaneously,
    \begin{align*}
        \frac{1}{n}\E_{\rho_{\bu}}[h_{b}(\bX,\btheta_{\bu})]&\le \frac{\lambda}{n}\sum_{i=1}^{n}\bu^{\top}\bSigma^{-1/2}\left(\bSigma\right)\bSigma^{-1/2}\bu+\frac{\lambda}{n}\sum_{i=1}^{n}\tr(\bfB_{\beta,r}\bSigma^{-1/2}\left(\bSigma\right)\bSigma^{-1/2})\}\\
        &\quad+ 64K\lambda^{2}\|\bSigma\|^{2}+\frac{\log2+\er(\bSigma)+t}{n}.
    \end{align*}
    The right-hand side is independent of $b$, and the left-hand side is non-decreasing in $b$.
    Taking the limit $b\to\infty$ and employing the monotonicity of the events in $b\in\N$, we obtain that with probability $1-e^{-t}$, for all $\bu\in\bSigma^{1/2}\bbS^{p-1}$ simultaneously,
    \begin{align*}
        &\frac{\lambda}{n}\E_{\rho_{\bu}}\left[\sum_{i=1}^{n}\btheta_{\bu}^{\top}\bSigma^{-1/2}\bX_{i}\bX_{i}^{\top}\bSigma^{-1/2}\btheta_{\bu}\right]\\
        &\le \frac{\lambda}{n}\sum_{i=1}^{n}\bu^{\top}\bSigma^{-1/2}\left(\bSigma\right)\bSigma^{-1/2}\bu+\frac{\lambda}{n}\sum_{i=1}^{n}\tr(\bfB_{\beta,r}\bSigma^{-1/2}\left(\bSigma\right)\bSigma^{-1/2})\}
        + 64K\lambda^{2}\|\bSigma\|^{2}+\frac{\log2+\er(\bSigma)+t}{n}.
    \end{align*}
    Using the following identity derived by Eq.~\eqref{eq:posteriormean:1} such that
    \begin{equation}
        \sum_{i=1}^{n}\E_{\rho_{\bu}}\left[\btheta_{\bu}^{\top}\bSigma^{-1/2}\bX_{i}\bX_{i}^{\top}\bSigma^{-1/2}\btheta_{\bu}\right]
        = \sum_{i=1}^{n}\bu^{\top}\bSigma^{-1/2}\bX_{i}\bX_{i}^{\top}\bSigma^{-1/2}\bu+\sum_{i=1}^{n}\tr(\bfB_{\beta,r}\bSigma^{-1/2}\bX_{i}\bX_{i}^{\top}\bSigma^{-1/2}),\label{eq:posteriormean:2}
    \end{equation}
    we obtain that with probability $1-e^{-t}$, for all $\bSigma^{1/2}\bu\in\bbS^{p-1}$ simultaneously,
    \begin{align*}
        &\frac{\lambda}{n}\sum_{i=1}^{n}\bu^{\top}\bSigma^{-1/2}\left(\bX_{i}\bX_{i}^{\top}-\bSigma\right)\bSigma^{-1/2}\bu+\frac{\lambda}{n}\sum_{i=1}^{n}\tr(\bfB_{\beta,r}\bSigma^{-1/2}\left(\bX_{i}\bX_{i}^{\top}-\bSigma\right)\bSigma^{-1/2})\\
        &\le 64K\lambda^{2}\|\bSigma\|^{2}+\frac{\log2+\er(\bSigma)+t}{n}.
    \end{align*}
    
    In the same manner, using the following bound derived by Lemma \ref{lemma:expmoment} such that for any $\lambda\in[0,1/(4\|\btheta\|^{2}K)]$ (by letting $\tau=-\lambda$),
    \begin{align*}
        \E_{\bX}\left[\exp\left(\lambda\btheta^{\top} \left(-\sum_{i=1}^{n}\bSigma^{-1/2}\bX_{i}\bX_{i}^{\top}\bSigma^{-1/2}\right)\btheta\right)\right]\le \exp\left(-\lambda n\btheta^{\top}\btheta\right)\exp\left(4K\lambda^{2}\|\btheta\|^{2}n\right),
    \end{align*}
    and letting $h(\bX,\btheta)=-\btheta^{\top} (\sum_{i=1}^{n}\bSigma^{-1/2}\bX_{i}\bX_{i}^{\top}\bSigma^{-1/2})\btheta$ ($\sup h\le 0$),
    we have that for any $\lambda\in[0,1/(16\|\bSigma\|K)]$, with probability $1-e^{-t}$, for all $\bSigma^{1/2}\bu\in\bbS^{p-1}$ simultaneously,
    \begin{align*}
        &\frac{\lambda}{n}\sum_{i=1}^{n}\bu^{\top}\bSigma^{-1/2}\left(\bSigma-\bX_{i}\bX_{i}^{\top}\right)\bSigma^{-1/2}\bu+\frac{\lambda}{n}\sum_{i=1}^{n}\tr(\bfB_{\beta,r}\bSigma^{-1/2}\left(\bSigma-\bX_{i}\bX_{i}^{\top}\right)\bSigma^{-1/2})\\
        &\le 64K\lambda^{2}\|\bSigma\|^{2}+\frac{\log2+\er(\bSigma)+t}{n}.
    \end{align*}

    (Step 3: concentration of the residuals) It remains to bound the sum of the traces.
    Note that
    \begin{equation*}
        \bfB_{\beta,r}=C_{\beta,r}^{-1}\E[\bzeta\bzeta^{\top}\mathbf{1}\{\|\bzeta\|\le r\}],
    \end{equation*}
    and $\bzeta\sim\Gauss(\zero,\beta^{-1}\bSigma)=\Gauss(\zero,(2\er(\bSigma))^{-1}\bSigma)$.
    Since $r$ was chosen to ensure $C_{\beta,r} \ge 1/2$,
    we obtain
    \begin{align*}
        &\left\|C_{\beta,r}^{-1}\E[\bzeta\bzeta^{\top}\mathbf{1}\{\|\bzeta\|\le r\}]\right\|\le 2\sup_{\bu\in\bbS^{p-1}}\E[(\bzeta^{\top}\bu)^{2}\mathbf{1}\{\|\bzeta\|\le r\}]\le 2\sup_{\bu\in\bbS^{p-1}}\E[(\bzeta^{\top}\bu)^{2}]\le \frac{\|\bSigma\|}{\er(\bSigma)},\\
        &\tr(C_{\beta,r}^{-1}\E[\bzeta\bzeta^{\top}\mathbf{1}\{\|\bzeta\|\le r\}])\le 2\E[\|\bzeta\|^{2}\mathbf{1}\{\|\bzeta\|\le r\}]\le 2\E[\|\bzeta\|^{2}]=2\frac{\tr(\bSigma)}{2\er(\bSigma)}=\|\bSigma\|.
    \end{align*}
    We again use Lemma \ref{lemma:expmoment} with $\bfA=\diag\{\bfB_{\beta,r},\ldots,\bfB_{\beta,r}\}$ and $ \bZ=(\bSigma^{-1/2}\bX_{1},\ldots,\bSigma^{-1/2}\bX_{n})$ ($\|\bfA\|=C_{\beta,r}^{-1}\beta^{-1}\|\bSigma\|\le \|\bSigma\|/\er(\bSigma)\le \|\bSigma\|$ by Lemma \ref{lem:norm}, and $\E[\bZ^{\top}\bfA\bZ]\le n\tr(\bfB_{\beta,r})\le n\|\bSigma\|$) and derive that for all $\lambda\in[0,1/(4\|\bSigma\|K)]$,
    \begin{align*}
        \E_{\bX}\left[\exp\left(\lambda\sum_{i=1}^{n}\tr(\bfB_{\beta,r}\bSigma^{-1/2}\left(\bX_{i}\bX_{i}^{\top}-\bSigma\right)\bSigma^{-1/2})\right)\right]&\le \exp\left(4K\|\bSigma\|^{2}\lambda^{2}n\right),\\
        \E_{\bX}\left[\exp\left(-\lambda\sum_{i=1}^{n}\tr(\bfB_{\beta,r}\bSigma^{-1/2}\left(\bX_{i}\bX_{i}^{\top}-\bSigma\right)\bSigma^{-1/2})\right)\right]&\le \exp\left(4K\|\bSigma\|^{2}\lambda^{2}n\right).
    \end{align*}
    Applying Chernoff bounds, we obtain that with probability $1-e^{-t}$,
    \begin{equation*}
        \frac{\lambda}{n}\sum_{i=1}^{n}\tr(\bfB_{\beta,r}\bSigma^{-1/2}\left(\bSigma-\bX_{i}\bX_{i}^{\top}\right)\bSigma^{-1/2})\le 4K\|\bSigma\|^{2}\lambda^{2}+\frac{t}{n},
    \end{equation*}
    and with probability $1-e^{-t}$,
    \begin{equation*}
        \frac{\lambda}{n}\sum_{i=1}^{n}\tr(\bfB_{\beta,r}\bSigma^{-1/2}\left(\bX_{i}\bX_{i}^{\top}-\bSigma\right)\bSigma^{-1/2})\le 4K\|\bSigma\|^{2}\lambda^{2}+\frac{t}{n}.
    \end{equation*}

    (Step 4: summarization) Therefore, with probability $1-4e^{-t}$,
    \begin{align*}
        \frac{1}{n}\left|\sum_{i=1}^{n}\bu^{\top}\bSigma^{-1/2}\left(\bX_{i}\bX_{i}^{\top}-\bSigma\right)\bSigma^{-1/2}\bu
      \right|\le (68K\|\bSigma\|^{2})\lambda +\frac{\log2+\er(\bSigma)+2t}{\lambda n}.
    \end{align*}
    By the translation $t\to t+\log4$, we have that with probability $1-e^{-t}$,
    \begin{align*}
        \frac{1}{n}\left|\sum_{i=1}^{n}\bu^{\top}\bSigma^{-1/2}\left(\bX_{i}\bX_{i}^{\top}-\bSigma\right)\bSigma^{-1/2}\bu
      \right|\le (68K\|\bSigma\|^{2})\lambda +\frac{(9/2)\er(\bSigma)+2t}{\lambda n},
    \end{align*}
    where we used $\log2+\er(\bSigma)+2(t+2\log2)=5\log2+\er(\bSigma)+2t\le (7/2)+\er(\bSigma)+2t\le (9/2)\er(\bSigma)+2t$.
    We choose $\lambda=\sqrt{((9/2)\er(\bSigma)+2t)/((68K\|\bSigma\|^{2})n)}$; the following computation shows that it satisfies $\lambda\le 1/(16\|\bSigma\|K)$ by the assumption on $n\ge 2K(9\er(\bSigma)+4t)$:
    \begin{align*}
        &\sqrt{((9/2)\er(\bSigma)+2t)/((68K\|\bSigma\|^{2})n)}\le 1/(16\|\bSigma\|K)\\
        \iff &((9/2)\er(\bSigma)+2t)/((68K\|\bSigma\|^{2})n)\le 1/(256\|\bSigma\|^{2}K^{2})\\
        \iff &n\ge \frac{64}{17}\|\bSigma\|^{2}K^{2}\frac{(9/2)\er(\bSigma)+2t}{K\|\bSigma\|^{2}}=\frac{32}{17}\|\bSigma\|^{2}K^{2}\frac{9\er(\bSigma)+4t}{K\|\bSigma\|^{2}}\\
        \impliedby &n\ge 2K(9\er(\bSigma)+4t),
    \end{align*}
    where we used $(32/17)\le 2$.
    We derive that with probability $1-e^{-t}$,
    \begin{align*}
        \frac{1}{n}\left|\sum_{i=1}^{n}\bu^{\top}\bSigma^{-1/2}\left(\bX_{i}\bX_{i}^{\top}-\bSigma\right)\bSigma^{-1/2}\bu
       \right|
       &\le 2\sqrt{\frac{68K\|\bSigma\|^{2}((9/2)\er(\bSigma)+2t)}{n}}\\
       &\le 12\sqrt{K}\|\bSigma\|\sqrt{\frac{9\er(\bSigma)+4t}{n}},
    \end{align*}
    where we used $\sqrt{34}\le 6$.
    Since we have set $\bu\in\bSigma^{1/2}\bbS^{p-1}$, and it holds that $\|\bfM\|=\sup_{\bv\in\bbS^{d-1}}|\bv^{\top}\bfM\bv|$ for any symmetric matrix $\bfM$, we obtain the desired conclusion.
\end{proof}

\begin{remark}\label{remark:general}
    As \citet{nakakita2024dimension} consider the concentration of the sums of random matrices under dependence, it is helpful to identify a sufficient condition for general random matrices inspired by Theorem \ref{thm:concentration}.
    Suppose that $(\bfM_{i})_{i=1}^{n}$ is a sequence of $p\times p$ positive semi-definite random matrices with $\E[\bfM_{i}]=\bSigma$. 
    By seeing the proof of Theorem \ref{thm:concentration} in detail, we notice that the following conditions are sufficient to derive the same concentration bound on $\|\frac{1}{n}\sum_{i=1}^{n}\bfM_{i}-\bSigma\|$: for some $K>0$, for all $p\times p$ positive semi-definite matrices $\bfA$ and for all $\tau\in(-\infty,1/(2\|\bfA\|K))$,
    \begin{equation}
        \E\left[\exp\left(\tau\sum_{i=1}^{n}\tr\left(\bfA\left((\bSigma^{\dagger})^{1/2}\bfM_{i}(\bSigma^{\dagger})^{1/2}-\E[(\bSigma^{\dagger})^{1/2}\bfM_{i}(\bSigma^{\dagger})^{1/2}]\right)\right)\right)\right]\le \exp\left(\frac{2\|\bfA\|K\tau^{2}}{1-2\|\bfA\|K\tau}n\tr(\bfA)\right).\label{eq:general}
    \end{equation}
    If it holds, then for any $t\ge0$ and positive integer $n$ satisfying $n\ge 2K(9\er(\bSigma)+4t)$, with probability at least $1-\exp(-t)$,
    \begin{equation*}
        \left\|\frac{1}{n}\sum_{i=1}^{n}\bfM_{i}-\bSigma\right\|\le  12\sqrt{K}\|\bSigma\|\sqrt{\frac{9\er(\bSigma)+4t}{n}}.
    \end{equation*}
\end{remark}

\section{Examples of dependent processes with the regularity condition}\label{corr:sec:example}
We display some examples of dependent processes satisfying the regularity condition of Theorem \ref{thm:concentration} above.
For any $L$-Lipschitz $\bsf$ and random vector $\bX$ whose distribution satisfies a log-Sobolev inequality with constant $K$, the distribution of $\bsf(\bX)$ satisfies a log-Sobolev inequality with constant $KL^{2}$.
Using this property, we consider causal Bernoulli shift (CBS) processes (Section \ref{corr:sec:example:cbs}), vector autoregressive (VAR) processes (Section \ref{corr:sec:example:var}), and augmented processes (Section \ref{corr:sec:example:aug}) as examples of dependent processes satisfying the regularity condition.
Since a na\"{i}ve mapping $\R^{p}\ni \bv\mapsto(\bSigma^{\dagger})^{1/2}\bv\in\R^{p}$ is $\lambda_{\min}(\bSigma)^{-1/2}$-Lipschitz, we carefully construct these examples.
The analysis on CBS and VAR validates the argument of Sections 4.1 (covariance estimation of CBS) and 4.4 (benign overfitting) of \citet{nakakita2024dimension}, and that on augmented processes recovers the contents of Sections 4.2 (lagged covariance estimation).

Furthermore, we apply our result to estimation of linear hidden Markov models (Section \ref{corr:sec:example:hmm}).
This application corresponds to Section 4.3 of \citet{nakakita2024dimension}. We also correct an error in that section.

We prepare the following lemma on the operator norms of block matrices.
\begin{lemma}\label{lem:causalblock}
    Let $(\bfJ_{i})_{i=0}^{m}$ ($m\in\Z_{\ge0}$) be a sequence of $d_{1}\times d_{2}$ matrices.
    We fix $n$ and set the following $(nd_{1})\times ((n+m)d_{2})$ matrix $\bfM$:
    \begin{equation*}
        \bfM=\left[\begin{matrix}
            \bfJ_{0} & \bfJ_{1} & \bfJ_{2} & \bfJ_{3} & \cdots & \bfO & \bfO & \bfO & \bfO\\
            \bfO & \bfJ_{0} & \bfJ_{1} & \bfJ_{2} & \cdots & \bfO & \bfO & \bfO & \bfO\\
            \bfO & \bfO & \bfJ_{0} & \bfJ_{1} & \cdots & \bfO & \bfO & \bfO & \bfO\\
            \vdots & \vdots & \vdots & \vdots & & \vdots & \vdots & \vdots & \vdots \\
            \bfO & \bfO & \bfO & \bfO & \cdots &  \bfJ_{m-2} & \bfJ_{m-1} & \bfJ_{m} & \bfO \\
            \bfO & \bfO & \bfO & \bfO & \cdots & \bfJ_{m-3} & \bfJ_{m-2} & \bfJ_{m-1} & \bfJ_{m}
        \end{matrix}\right].
    \end{equation*}
    We have $\|\bfM\|\le \sum_{i=0}^{m}\|\bfJ_{i}\|$.
\end{lemma}

\begin{proof}
    We obtain that
    \begin{align*}
        \|\bfM\|
        &=\sup_{\substack{(\bu_{i})_{i=1}^{n}\subset\R^{d_{1}}:\sum_{i=1}^{n}\|\bu_{i}\|^{2}=1,\\(\bv_{j})_{j=1}^{n+m}\subset\R^{d_{2}}:\sum_{j=1}^{n+m}\|\bv_{j}\|^{2}=1}}
        \sum_{i=1}^{n}\sum_{j=i}^{i+m}\bu_{i}^{\top}\bfJ_{j-i}\bv_{j}\\
        &\le \sup_{\substack{(\bu_{i})_{i=1}^{n}\subset\R^{d_{1}}:\sum_{i=1}^{n}\|\bu_{i}\|^{2}=1,\\(\bv_{j})_{j=1}^{n+m}\subset\R^{d_{2}}:\sum_{j=1}^{n+m}\|\bv_{j}\|^{2}=1}}
        \sqrt{\sum_{i=1}^{n}\|\bu_{i}\|^{2}}\sqrt{\sum_{i=1}^{n}\left\|\sum_{j=i}^{i+m}\bfJ_{j-i}\bv_{j}\right\|^{2}}\ \because\text{(Cauchy--Schwarz)}\\
        &= \sup_{\substack{(\bv_{j})_{j=1}^{n+m}\subset\R^{d_{2}}:\\
        \sum_{j=1}^{n+m}\|\bv_{j}\|^{2}=1}}
        \sqrt{\sum_{i=1}^{n}\left\|\sum_{j=i}^{i+m}\frac{\sum_{k=i}^{i+m}\|\bfJ_{k-i}\|}{\sum_{k=i}^{i+m}\|\bfJ_{k-i}\|}\frac{\|\bfJ_{j-i}\|}{\|\bfJ_{j-i}\|}\bfJ_{j-i}\bv_{j}\right\|^{2}}\\
        &= \left(\sum_{k=0}^{m}\|\bfJ_{k}\|\right)\sup_{\substack{(\bv_{j})_{j=1}^{n+m}\subset\R^{d_{2}}:\\
        \sum_{j=1}^{n+m}\|\bv_{j}\|^{2}=1}}
        \sqrt{\sum_{i=1}^{n}\left\|\sum_{j=i}^{i+m}\frac{\|\bfJ_{j-i}\|}{\sum_{k=i}^{i+m}\|\bfJ_{k-i}\|}\frac{1}{\|\bfJ_{j-i}\|}\bfJ_{j-i}\bv_{j}\right\|^{2}}\\
        &\le \left(\sum_{k=0}^{m}\|\bfJ_{k}\|\right)\sup_{\substack{(\bv_{j})_{j=1}^{n+m}\subset\R^{d_{2}}:\\
        \sum_{j=1}^{n+m}\|\bv_{j}\|^{2}=1}}
        \sqrt{\sum_{i=1}^{n}\sum_{j=i}^{i+m}\frac{\|\bfJ_{j-i}\|}{\sum_{k=i}^{i+m}\|\bfJ_{k-i}\|}\left\|\frac{1}{\|\bfJ_{j-i}\|}\bfJ_{j-i}\bv_{j}\right\|^{2}}\ \because\text{(Jensen)}\\
        &=\sqrt{\sum_{k=0}^{m}\|\bfJ_{k}\|}\sup_{\substack{(\bv_{j})_{j=1}^{n+m}\subset\R^{d_{2}}:\\
        \sum_{j=1}^{n+m}\|\bv_{j}\|^{2}=1}}
        \sqrt{\sum_{i=1}^{n}\sum_{j=i}^{i+m}\frac{1}{\|\bfJ_{j-i}\|}\left\|\bfJ_{j-i}\bv_{j}\right\|^{2}}\\
        &=\sqrt{\sum_{k=0}^{m}\|\bfJ_{k}\|}\sup_{\substack{(\bv_{j})_{j=1}^{n+m}\subset\R^{d_{2}}:\\
        \sum_{j=1}^{n+m}\|\bv_{j}\|^{2}=1}}
        \sqrt{\sum_{i=1}^{n}\sum_{j=i}^{i+m}\left\|\bfJ_{j-i}\right\|\left\|\bv_{j}\right\|^{2}}\\
        &=\sqrt{\sum_{k=0}^{m}\|\bfJ_{k}\|}\sup_{\substack{(\bv_{j})_{j=1}^{n+m}\subset\R^{d_{2}}:\\
        \sum_{j=1}^{n+m}\|\bv_{j}\|^{2}=1}}
        \sqrt{\sum_{j=1}^{n+m}\sum_{i=1\vee (j-m)}^{j\wedge n}\left\|\bfJ_{j-i}\right\|\left\|\bv_{j}\right\|^{2}}\\
        &\le \left(\sum_{j=0}^{m}\|\bfJ_{j}\|\right)\sup_{\substack{(\bv_{j})_{j=1}^{n+m}\subset\R^{d_{2}}:\\
        \sum_{j=1}^{n+m}\|\bv_{j}\|^{2}=1}}
        \sqrt{\sum_{j=1}^{n+m}\left\|\bv_{j}\right\|^{2}}=\sum_{j=0}^{m}\|\bfJ_{j}\|.
    \end{align*}
    This concludes the proof.
\end{proof}

\subsection{Causal Bernoulli shifts}\label{corr:sec:example:cbs}
Let us fix a non-negative integer $m$ (a lag parameter), which determines the length of a sequence of driving noises; we can let it diverge later.
For a sequence of $d$-dimensional independent and identically distributed random vectors $(\bzeta_{i})_{i=1-m}^{n}$ such that the distribution of each $\bzeta_{i}$ satisfies a log-Sobolev inequality with constant $K_{\bzeta}$, and $\bsf:(\R^{d})^{m}\to\R^{p}$ such that $\E[\bsf((\bzeta_{1-i})_{i=1}^{m})\bsf((\bzeta_{1-i})_{i=0}^{m})^{\top}]=\bfI_{p}$ whose Jacobian $\bfJ_{\bsf}$ satisfies $\sum_{i=0}^{m}\left\|\nabla_{i}\bsf\right\|\le L$ for some $L>0$ (here, $\nabla_{i}\bsf$ means the Jacobian of $\bsf$ with respect to the $i$-th input vector of $\bsf$).
A trivial example is that $m=0$, $d=p$, $\bzeta\sim\Gauss(\zero,\bfI_{p})$, and $\bsf=\bfI_{d}$.
We set $\bfZ_{\ell}=\bsf((\bzeta_{\ell-i})_{i=1}^{m})$.
The distribution of $(\bfZ_{1},\ldots,\bfZ_{n})$ satisfies a log-Sobolev inequality with constant $K_{\bzeta}L^{2}$.
Hence, we can set $\bX_{i}=\bSigma^{1/2}\bZ_{i}$ for positive semi-definite $\bSigma$.

Taking the limit $m\to\infty$ can be justified as follows. We assume that the sequence of the $np$-dimensional random vectors $(\bZ_{i}^{m})_{i=1}^{n}:=(\bsf((\bzeta_{i-j})_{j=1}^{m}))_{i=1}^{n}$ converges to $(\bZ_{i}^{\infty})_{i=1}^{n}$ in law (here, $\bsf$ itself is also dependent on $m$) when $m\to\infty$, and $L=\sup_{m}\sum_{i=0}^{m}\|\nabla_{i}\bsf\|<\infty$.
Then, the continuous mapping theorem implies that 
\begin{equation*}
    \left\|\frac{1}{n}\sum_{i=1}^{n}\bSigma^{1/2}(\bZ_{i}^{m})(\bZ_{i}^{m})^{\top}\bSigma^{1/2}-\bSigma\right\|\overset{\calL}{\to}\left\|\frac{1}{n}\sum_{i=1}^{n}\bSigma^{1/2}(\bZ_{i}^{\infty})(\bZ_{i}^{\infty})^{\top}\bSigma^{1/2}-\bSigma\right\|.
\end{equation*}
Therefore, we complete the justification by the portmanteau lemma; in fact, the weak inequality $\le$ in Theorem \ref{thm:concentration} can be replaced with the strict one $<$ since our proof is based on Markov's inequality, and thus the portmanteau lemma for open sets yields the justification.

As \citet{nakakita2024dimension} consider independent additive noises to $\bX_{i}$, we can also treat additive noises, as it is straightforward to obtain estimates for log-Sobolev constants under log-Sobolev inequalities for the distribution of additive noises.
If $(\bZ_{i})_{i=1}^{n}$ and $(\bepsilon_{i})_{i=1}^{n}$ are $np$-dimensional random vectors independent of each other and whose distributions satisfy log-Sobolev inequalities with constant $C_{\bZ}$ and $C_{\bepsilon}$ respectively, then $(\bY_{i})_{i=1}^{n}:=(\bZ_{i}+\bepsilon_{i})_{i=1}^{n}$ satisfies a log-Sobolev constant with $2(C_{\bZ}\vee C_{\bepsilon})$ by tensorization \citep{bakry2014analysis} and the $\sqrt{2}$-Lipschitz continuity of summation.
Hence, if $\bepsilon_{i}$ is an i.i.d.~noise with isotropy $\E[\bepsilon_{i}\bepsilon_{i}^{\top}]=\bfI_{d}$, then $\bX_{i}=\bSigma^{1/2}\bY_{i}$ leads to the desired property.

\subsection{Vector autoregressive processes}\label{corr:sec:example:var}
Let us consider a stationary vector autoregressive process as follows:
\begin{equation*}
    \bX_{i}=\bfA\bX_{i-1}+\bfB\bzeta_{i},
\end{equation*}
where $(\bzeta_{i})_{i}$ is a sequence of independent and identically distributed centered isotropic random vectors, and both $\bfA$ and $\bfB$ are symmetric matrices whose spectral decompositions are given with respect to the same orthonormal basis $\{\be_{j}\}\subset\R^{p}$;
\begin{align*}
    \bfA=\sum_{j=1}^{p}a_{j}\be_{j}\be_{j}^{\top},\ \bfB=\sum_{j=1}^{p}b_{j}\be_{j}\be_{j}^{\top},
\end{align*}
and assume that $|a_{j}|<1$ and $b_{j}>0$ for all $j$. 
Without loss of generality, suppose that $|a_{1}|\ge |a_{2}|\ge\cdots\ge |a_{p}|$.
Assume that the distribution of $\bzeta_{i}$ satisfies a log-Sobolev constant $K_{\bzeta}$.
Then, by the argument of Chapter 3 of \citet{bosq2000linear}, we have that the covariance of $\bX_{i}$ is given as
\begin{equation*}
    \bSigma=\sum_{j=1}^{p}b_{j}^{2}\left(1+\sum_{k=1}^{\infty}a_{j}^{2k}\right)\be_{j}\be_{j}^{\top}=\sum_{j=1}^{p}b_{j}^{2}\left(1+\frac{a_{j}^{2}}{1-a_{j}^{2}}\right)\be_{j}\be_{j}^{\top}=\sum_{j=1}^{p}\frac{b_{j}^{2}}{1-a_{j}^{2}}\be_{j}\be_{j}^{\top}.
\end{equation*}
We use the moving-average representation of autoregressive processes and obtain
\begin{align*}
    \left[\begin{matrix}
        \bSigma^{-1/2}\bX_{n}\\
        \vdots\\
        \bSigma^{-1/2}\bX_{1}
    \end{matrix}\right]=\left[\begin{matrix}
        \sum_{j=1}^{p}\sqrt{1-a_{j}^{2}}\be_{j}\be_{j}^{\top} & \sum_{j=1}^{p}a_{j}\sqrt{1-a_{j}^{2}}\be_{j}\be_{j}^{\top} & \sum_{j=1}^{p}a_{j}^{2}\sqrt{1-a_{j}^{2}}\be_{j}\be_{j}^{\top}  & \cdots \\
        \bfO & \sum_{j=1}^{p}\sqrt{1-a_{j}^{2}}\be_{j}\be_{j}^{\top} & \sum_{j=1}^{p}a_{j}\sqrt{1-a_{j}^{2}}\be_{j}\be_{j}^{\top}  & \cdots\\
        \vdots & \vdots & \vdots & \\
        \bfO & \bfO& \bfO & \cdots
    \end{matrix}\right]
    \left[\begin{matrix}
        \bzeta_{n}\\
        \bzeta_{n-1}\\
        \bzeta_{n-2}\\
        \vdots
    \end{matrix}\right].
\end{align*}
Let $\bfC=\sum_{j=1}^{p}\sqrt{1-a_{j}^{2}}\be_{j}\be_{j}^{\top}$.
It is sufficient to examine the operator norm of the following upper triangular block matrix:
\begin{equation*}
    \bfM=\left[\begin{matrix}
        \bfC & \bfC\bfA & \bfC\bfA^{2} & \bfC\bfA^{3} &\cdots \\
        \bfO & \bfC & \bfC\bfA & \bfC\bfA^{2} & \cdots \\
        \bfO & \bfO & \bfC   & \bfC\bfA &\cdots\\
        \vdots & \vdots & \vdots & \vdots& \\
        \bfO & \bfO& \bfO & \bfO & \cdots
    \end{matrix}\right].
\end{equation*}
Lemma \ref{lem:causalblock} (with an abuse $m=\infty$; we can derive the same conclusion by truncation with finite $m$ and taking the limit after that) derives that $\|\bfM\|\le \|C\|\sum_{i=0}^{\infty}\|\bfA\|^{i}=\sqrt{1-a_{p}^{2}}/(1-|a_{1}|).$
Therefore, we obtain that the distribution of $(\bSigma^{-1/2}\bX_{1},\cdots,\bSigma^{-1/2}\bX_{n})$ satisfies a log-Sobolev inequality with constant $K_{\bzeta}(1-a_{p}^{2})/(1-|a_{1}|)^{2}$.

\subsection{Augmented processes for lagged covariance matrix estimation}\label{corr:sec:example:aug}
In this section, we argue the concentration of the sample covariance matrix of an augmented processes $((\bX_{i},\bX_{i+1}))_{i}$.
Section 4.2 of \citet{nakakita2024dimension} analyses augmented processes for lagged covariance matrix estimation, and Section 4.3 employs this discussion to investigate estimation of linear hidden Markov models.
We can easily extend Theorem \ref{thm:concentration} under a log-Sobolev inequality to augmented processes and recover those results.
First, let us define the following augmented sample covariance matrix:
\begin{equation*}
    \hat{\bSigma}_{0:1}:=\frac{1}{n-1}\sum_{i=1}^{n}\left[\begin{matrix}
        \bX_{i}\bX_{i}^{\top} & \bX_{i}\bX_{i+1}^{\top}\\
        \bX_{i+1}\bX_{i}^{\top} & \bX_{i+1}\bX_{i+1}^{\top}
    \end{matrix}\right]=:\left[\begin{matrix}
        \hat{\bSigma} & \hat{\bSigma}_{1}\\
        \hat{\bSigma}_{1}^{\top} & \hat{\bSigma}_{\textnormal{shift}}
    \end{matrix}\right].
\end{equation*}
We now state a general concentration result.
\begin{proposition}[correction to Proposition 8 of \citealp{nakakita2024dimension}]\label{prp:aug}
Assume that $(\bX_{i})_{i=1}^{n}$ be a sequence of $p$-dimensional random vectors satisfying the following conditions: (i) $\E[\bX_{i}\bX_{i}^{\top}]=\bSigma$ and $\E[\bX_{i}\bX_{i+1}^{\top}]=\bSigma_{1}$ for all $i=1,\ldots,n$,
and (ii) the distribution of the $\R^{2p(n-1)}$-valued random vector
\begin{equation*}
    \left(\bSigma_{0:1}^{\dagger}\left[\begin{matrix}\bX_{1}\\
    \bX_{2}
    \end{matrix}\right],\ldots,\bSigma_{0:1}^{\dagger}\left[\begin{matrix}\bX_{n-1}\\
    \bX_{n}\end{matrix}\right]\right)
\end{equation*}
satisfies a log-Sobolev inequality with constant $K_{0:1}$, where $\bSigma_{0:1}$ is defined as
\begin{equation*}
    \bSigma_{0:1}=\left[\begin{matrix}
        \bSigma & \bSigma_{1}\\
        \bSigma_{1}^{\top} & \bSigma
    \end{matrix}\right].
\end{equation*}
Then, for any $t\ge0$ and positive integer $n$ satisfying $n-1\ge 2K_{0:1}(9\er(\bSigma)+4t)$, with probability at least $1-\exp(-t)$,
    \begin{equation*}
        \left\|\hat{\bSigma}_{0:1}-\bSigma_{0:1}\right\|\le  12\sqrt{K_{0:1}}\|\bSigma_{0:1}\|\sqrt{\frac{9\er(\bSigma_{0:1})+4t}{n-1}}.
    \end{equation*}
    Moreover, on the same event, we have
    \begin{equation*}
        \max\{\|\hat{\bSigma}-\bSigma\|,\|\hat{\bSigma}_{1}-\bSigma_{1}\|\}\le 12\sqrt{2K_{0:1}}(\|\bSigma\|+\|\bSigma_{1}\|)\sqrt{\frac{9\er(\bSigma)+2t}{n-1}}.
    \end{equation*}
\end{proposition}

\begin{proof}
The first statement is a repetition of the statement of Theorem \ref{thm:concentration}.
The second statement used the facts that (i) for a given matrix, the operator norm of a submatrix is bounded above by the operator norm of the full matrix,
(ii) $\|\bSigma_{0:1}\|\le \|\bSigma\|+\|\bSigma_{1}\|$ derived as for any $\bu_{1},\bu_{2}\in\R^{p}$ with $\|\bu_{1}\|^{2}+\|\bu_{2}\|^{2}=1$,
\begin{equation*}
    \left[\begin{matrix}
        \bu_{1}^{\top} & \bu_{2}^{\top}
    \end{matrix}\right]\left[\begin{matrix}
        \bSigma & \bSigma_{1}\\
        \bSigma_{1}^{\top} & \bSigma
    \end{matrix}\right]
    \left[\begin{matrix}
        \bu_{1}\\
        \bu_{2}
    \end{matrix}\right]\le (\|\bu_{1}\|^{2}+\|\bu_{2}\|^{2})\|\bSigma\|+(2|\bu_{1}^{\top}\bu_{2}|)\|\bSigma_{1}\|\le \|\bSigma\|+\|\bSigma_{1}\|,
\end{equation*}
and (iii) $\er(\bSigma_{0:1})=2\tr(\bSigma)/\|\bSigma_{0:1}\|\ge 2\er(\bSigma)$ by $\|\bSigma_{0:1}\|\ge \|\bSigma\|$.
\end{proof}

Derivation of concrete bounds on the log-Sobolev constant $K_{0:1}$, however, is nontrivial.
Hence, we give a concrete estimate for $K_{0:1}$ under a mild condition.
\begin{lemma}\label{lem:lsi:aug}
    Assume that $(\bX_{i})_{i=1}^{n}$ be a sequence of $p$-dimensional random vectors such that (i) $\E[\bX_{i}\bX_{i}^{\top}]=\bSigma$ and $\E[\bX_{i}\bX_{i+1}^{\top}]=\bSigma_{1}$ for all $i=1,\ldots,n$.
    Furthermore, suppose that the distribution of $((\bSigma^{\dagger})^{1/2}\bX_{1},\ldots,(\bSigma^{\dagger})^{1/2}\bX_{n})$ satisfies a log-Sobolev inequality with constant $K$.
    Then, we have that $K_{0:1}\le 2K\varrho$, where $\varrho$ is a positive constant defined as
    \begin{align*}
     \varrho=\sup_{\bu\in\bbS^{2p-1}}\bu^{\top}\sqrt{\left[\begin{matrix}
        \bSigma & \bfO\\
        \bfO & \bSigma
    \end{matrix}\right]}\left[\begin{matrix}
        \bSigma & \bSigma_{1}\\
        \bSigma_{1}^{\top} & \bSigma
    \end{matrix}\right]^{\dagger}
    \sqrt{\left[\begin{matrix}
        \bSigma & \bfO\\
        \bfO & \bSigma
    \end{matrix}\right]}\bu.
\end{align*}
\end{lemma}

\begin{proof}
    Define the following map $\bsf:\R^{pn}\to\R^{2p(n-1)}$: for all $(\bx_{i})_{i=1}^{n}\subset\R^{p}$,
\begin{equation*}
    \bsf(\bx_{1},\ldots,\bx_{n})=[
        \underbrace{\bx_{1}^{\top} \ \bx_{2}^{\top}}_{=:\by_{1}^{\top}} \ \underbrace{\bx_{2}^{\top} \ \bx_{3}^{\top}}_{=:\by_{2}^{\top}} \ \underbrace{\bx_{3}^{\top} \ \bx_{4}^{\top}}_{=:\by_{3}^{\top}} \ \cdots \ 
        \underbrace{\bx_{n-1}^{\top} \ \bx_{n}^{\top}}_{=:\by_{n-1}^{\top}}
    ]^{\top}=[\by_{1}^{\top}\ \by_{2}^{\top}\ \cdots\ \by_{n-1}^{\top}]^{\top}.
\end{equation*}
The Lipschitz constant of $\bsf$ is given as
\begin{align*}
    \left\|\bfJ_{\bsf}(\bx_{1},\ldots,\bx_{n})\right\|&=\left\|\left[\begin{matrix}
        \bfI_{p} & \bfO & \bfO & \cdots & \bfO\\
        \bfO & \bfI_{p} & \bfO & \cdots & \bfO\\
        \bfO & \bfI_{p} & \bfO & \cdots & \bfO\\
        \bfO & \bfO & \bfI_{p} & \cdots & \bfO\\
        \bfO & \bfO & \bfI_{p} & \cdots & \bfO\\
        \vdots & \vdots & \vdots & \ddots & \vdots \\
        \bfO & \bfO & \bfO & \cdots & \bfI_{p}
    \end{matrix}\right]\right\|\\
    &=\sup_{\substack{(\bu_{i})_{i}\subset\R^{p}:\sum_{i}\|\bu_{i}\|^{2}=1,\\(\bv_{j})_{j}\subset\R^{p}:\sum_{j}\|\bv_{j}\|^{2}=1}}\left[
        \bu_{1}^{\top}\ \bu_{2}^{\top}\ \cdots\ \bu_{2(n-1)}
    \right]\left[\begin{matrix}
        \bfI_{p} & \bfO & \bfO & \cdots & \bfO\\
        \bfO & \bfI_{p} & \bfO & \cdots & \bfO\\
        \bfO & \bfI_{p} & \bfO & \cdots & \bfO\\
        \bfO & \bfO & \bfI_{p} & \cdots & \bfO\\
        \bfO & \bfO & \bfI_{p} & \cdots & \bfO\\
        \vdots & \vdots & \vdots & \ddots & \vdots \\
        \bfO & \bfO & \bfO & \cdots & \bfI_{p}
    \end{matrix}\right]\left[\begin{matrix}
        \bv_{1}\\
        \bv_{2}\\
        \bv_{3}\\
        \bv_{4}\\
        \bv_{5}\\
        \vdots \\
        \bv_{n}
    \end{matrix}\right]\\
    &=\sup_{\substack{(\bu_{i})_{i}\subset\R^{p}:\sum_{i}\|\bu_{i}\|^{2}=1,\\(\bv_{j})_{j}\subset\R^{p}:\sum_{j}\|\bv_{j}\|^{2}=1}}\left(\bu_{1}^{\top}\bv_{1}+\sum_{i=1}^{n-1}\left((\bu_{2i}+\bu_{2i+1})^{\top}\bv_{i}\right)+\bu_{2(n-1)}^{\top}\bv_{n}\right)\\
    &\le \sup_{\substack{(\bu_{i})_{i}\subset\R^{p}:\sum_{i}\|\bu_{i}\|^{2}=1,\\(\bv_{j})_{j}\subset\R^{p}:\sum_{j}\|\bv_{j}\|^{2}=1}}\sqrt{\|\bu_{1}\|^{2}+\sum_{i=1}^{n-1}\|\bu_{2i}+\bu_{2i+1}\|^{2}+\|\bu_{2(n-1)}\|^{2}}\sqrt{\sum_{i=1}^{n}\|\bv_{i}\|^{2}}\\
    &=\sup_{(\bu_{i})_{i}\subset\R^{p}:\sum_{i}\|\bu_{i}\|^{2}=1}\sqrt{\|\bu_{1}\|^{2}+\sum_{i=1}^{n-1}\|\bu_{2i}+\bu_{2i+1}\|^{2}+\|\bu_{2(n-1)}\|^{2}}\\
    &\le \sup_{(\bu_{i})_{i}\subset\R^{p}:\sum_{i}\|\bu_{i}\|^{2}=1}\sqrt{\|\bu_{1}\|^{2}+2\sum_{i=1}^{n-1}(\|\bu_{2i}\|^{2}+\|\bu_{2i+1}\|^{2})+\|\bu_{2(n-1)}\|^{2}}\\
    &\le \sup_{(\bu_{i})_{i}\subset\R^{p}:\sum_{i}\|\bu_{i}\|^{2}=1}\sqrt{2\sum_{i=1}^{2(n-1)}\|\bu_{i}\|^{2}}\\
    &\le \sqrt{2}.
\end{align*}
Consequently, the distribution of the augmented process $((\bSigma^{\dagger})^{1/2}\bX_{i},(\bSigma^{\dagger})^{1/2}\bX_{i+1})_{i=1}^{n-1}$ satisfies a log-Sobolev inequality with constant $2K$.
Then, we obtain that the augmented process
\begin{equation*}
    \left(\sqrt{\left[\begin{matrix}
        \bSigma & \bSigma_{1}\\
        \bSigma_{1}^{\top} & \bSigma
    \end{matrix}\right]^{\dagger}}\left[\begin{matrix}
        \bX_{i}\\
        \bX_{i+1}
    \end{matrix}\right]\right)_{i=1}^{n-1}=\left(\sqrt{\left[\begin{matrix}
        \bSigma & \bSigma_{1}\\
        \bSigma_{1}^{\top} & \bSigma
    \end{matrix}\right]^{\dagger}}\sqrt{\left[\begin{matrix}
        \bSigma & \bfO\\
        \bfO & \bSigma
    \end{matrix}\right]}\left[\begin{matrix}
        (\bSigma^{\dagger})^{1/2}\bX_{i}\\
        (\bSigma^{\dagger})^{1/2}\bX_{i+1}
    \end{matrix}\right]\right)_{i=1}^{n-1}
\end{equation*}
satisfies a log-Sobolev inequality with constant $2K\varrho$; the identity holds since the probability that $\bX_{i}$ takes its value out of the span of the eigenvectors of $\bSigma$ is zero.
\end{proof}

$\varrho$ measures the degree of dependence between $\bX_{i}$ and $\bX_{i+1}$.
For example, if $\bSigma_{1}=\bfO$ (independent case), then $\varrho=1$.
See Section \ref{corr:sec:example:hmm} for an example of a dimension-free $\varrho$ under dependence.

\subsection{Linear hidden Markov processes}\label{corr:sec:example:hmm}
Consider the following linear hidden Markov model:
\begin{align*}
  \bY_{i}&=\bX_{i}+\bvarepsilon_{i}\\
  \bX_{i}&=\bfA\bX_{i-1}+\bxi_{i},
\end{align*}
where $(\bY_{i})_{i\in\Z}$ is a stationary observed process, $(\bX_{i})_{i\in\Z}$ is a stationary latent process, $\bfA\in\R^{p\times p}$ is an unknown parameter with $\|\bfA\|<1$, and $(\bvarepsilon_{i})_{i\in\Z}$ and $(\bxi_{i})_{i\in\Z}$ are sequences of $p$-dimensional i.i.d.~random vectors satisfying $\E[\bvarepsilon_{i}]=\E[\bxi_{i}]=\zero$ and $\Cov(\bvarepsilon_{i})=\Cov(\bxi_{i})=\bfI_{p}$.
Furthermore, $(\bvarepsilon_{i})_{i\in\Z}$ and $(\bxi_{i})_{i\in\Z}$ are independent of each other, and the distributions of $\bvarepsilon_{i}$ and $\bxi_{i}$ satisfy log-Sobolev inequalities with constants $K_{\bvarepsilon}$ and $K_{\bxi}$ respectively.
We use the notation $\bSigma_{1}^{\top}=\E[\bY_{i}\bY_{i-1}^{\top}]$ and $\bSigma=\E[\bY_{i}\bY_{i}^{\top}]$.

We consider the problem of estimating $\bfA$ by $(\bY_{i})_{i=1}^{n}$.
Let us first derive a representation of $\bfA$ in terms of $\bSigma_{1}$ and $\bSigma$.
Since
\begin{align*}
    \bY_{i}=\bfA\bX_{i-1}+\bxi_{i}+\bvarepsilon_{i}=\bfA\bY_{i-1}+\bxi_{i}+\bvarepsilon_{i}-\bfA\bvarepsilon_{i-1},
\end{align*}
we yield
\begin{align*}
    \bSigma_{1}^{\top}&=\E\left[\bY_{i}\bY_{i-1}^{\top}\right]\\
    &=\bfA\E\left[\bY_{i-1}\bY_{i-1}^{\top}\right]+\E\left[\bxi_{i}\bY_{i-1}^{\top}\right]+\E\left[\bvarepsilon_{i}\bY_{i-1}^{\top}\right]-\bfA\E\left[\bvarepsilon_{i-1}\bY_{i-1}^{\top}\right]\\
    &=\bfA\E\left[\bY_{i-1}\bY_{i-1}^{\top}\right]-\bfA\E\left[\bvarepsilon_{i-1}\bepsilon_{i-1}^{\top}\right]\\
    &=\bfA\left(\bSigma-\bfI_{p}\right).
\end{align*}
Noting that $\bSigma-\bfI_{p}=\Cov(\bX_{i})$ and $\Cov(\bX_{i})=\bfA\Cov(\bX_{i-1})\bfA^{\top}+\bfI_{p}\succeq\bfI_{p}$, we notice that $\bSigma-\bfI_{p}$ is invertible, and 
\begin{align*}
    \bfA&=\bSigma_{1}^{\top}\left(\bSigma-\bfI_{p}\right)^{-1}.
\end{align*}
\begin{remark}
\citet{nakakita2024dimension} write that $\bfA=\bSigma_{1}^{\top}(\bSigma+\bfI_{p})^{-1}$, which is another error.
We correct this error and argue the estimation of $\bfA$.
\end{remark}

Motivated by this representation and using $\hat{\bSigma}_{1}^{\top}=(1/(n-1))\sum_{i=2}^{n}(\bY_{i}\bY_{i-1}^{\top})$ and $\hat{\bSigma}=(1/(n-1))\sum_{i=1}^{n-1}(\bY_{i}\bY_{i}^{\top})$, we set the following estimator of $\bfA$:
\begin{equation*}
    \hat{\bfA}=\begin{cases}
        \hat{\bSigma}_{1}^{\top}\left(\hat{\bSigma}-\bfI_{p}\right)^{-1} &\text{if }\hat{\bSigma}-\bfI_{p}\text{ is invertible},\\
        \bfI_{p}&\text{otherwise}.
    \end{cases}
\end{equation*}
Note that we can choose an arbitrary value as $\hat{\bfA}$ when $\hat{\bSigma}-\bfI_{p}$ is not invertible.

We obtain the following bound on the estimation error of $\hat{\bfA}$.
\begin{proposition}
Assume that the distribution of $(\bSigma^{-1/2}\bY_{1},\ldots,\bSigma^{-1/2}\bY_{n})$ satisfies a log-Sobolev inequality with constant $K$.
Furthermore, we set $\varrho$ as a constant defined as
\begin{equation*}
    \varrho:=\sup_{\bu\in\bbS^{2p-1}}\bu^{\top}\left[\begin{matrix}
        \bSigma & \bfO\\
        \bfO & \bSigma
    \end{matrix}\right]^{1/2}\left[\begin{matrix}
        \bSigma & \bSigma_{1}\\
        \bSigma_{1}^{\top} & \bSigma
    \end{matrix}\right]^{\dagger}
    \left[\begin{matrix}
        \bSigma & \bfO\\
        \bfO & \bSigma
    \end{matrix}\right]^{1/2}\bu.
\end{equation*}
Then, for any $t\ge0$ and positive integer $n$ satisfying 
\begin{align*}
  n\ge 1+48^{2}K\varrho(\|\bSigma\|+\|\bSigma_{1}\|)^{2}(9\er(\bSigma)+2t),
\end{align*}
we have
\begin{align*}
    \|\hat{\bfA}-\bfA\|\le 96\sqrt{K\varrho}\left(1+\|\bSigma_{1}\|\right)(\|\bSigma\|+\|\bSigma_{1}\|)\sqrt{\frac{9\er(\bSigma)+2t}{n-1}}.
\end{align*}
\end{proposition}
\begin{proof}
First, since $n$ satisfies
\begin{align*}
  n\ge 1+48^{2}K\varrho(\|\bSigma\|+\|\bSigma_{1}\|)^{2}(9\er(\bSigma)+2t),
\end{align*}
Proposition \ref{prp:aug} and Lemma \ref{lem:lsi:aug} derive that with probability at least $1-e^{-t}$,
\begin{equation}\label{eq:hmm:event}
  \max\{\|\hat{\bSigma}-\bSigma\|,\|\hat{\bSigma}_{1}-\bSigma_{1}\|\}\le 24\sqrt{K\varrho}(\|\bSigma\|+\|\bSigma_{1}\|)\sqrt{\frac{9\er(\bSigma)+2t}{n-1}},
\end{equation}
and this upper bound is further bounded above by $1/2$.
Let us suppose that the event \eqref{eq:hmm:event} holds hereafter.

We check the non-degeneracy of $\hat{\bSigma}-\bfI_{p}$ on the event.
Since $\bSigma-\bfI_{p}=\Cov(\bX_{i})=\bfA\Cov(\bX_{i-1})\bfA^{\top}+\bfI_{p}\succeq\bfI_{p}$, we have
\begin{align*}
  \lambda_{\min}(\hat{\bSigma}-\bfI_{p})\ge \lambda_{\min}(\bSigma-\bfI_{p})- \|(\hat{\bSigma}-\bfI_{p})-(\bSigma-\bfI_{p})\|\ge 1-\|\hat{\bSigma}-\bSigma\|\ge 1/2.
\end{align*}
Therefore, on this event, $(\hat{\bSigma}-\bfI_{p})^{-1}$ is invertible, and  $\hat{\bfA}=\hat{\bSigma}_{1}^{\top}(\hat{\bSigma}-\bfI_{p})^{-1}$.

Now let us bound $\|\hat{\bfA}-\bfA\|$ from above.
On the event \eqref{eq:hmm:event}, we obtain that
\begin{align*}
  \|\hat{\bfA}-\bfA\|&= \left\|\hat{\bSigma}_{1}^{\top}(\hat{\bSigma}-\bfI_{p})^{-1}-\bSigma_{1}^{\top}(\bSigma-\bfI_{p})^{-1}\right\|\\
  &=\left\|\left(\hat{\bSigma}_{1}^{\top}-\bSigma_{1}^{\top}\right)(\hat{\bSigma}-\bfI_{p})^{-1}+\bSigma_{1}^{\top}\left((\hat{\bSigma}-\bfI_{p})^{-1}-(\bSigma-\bfI_{p})^{-1}\right)\right\|\\
  &=\left\|\left(\hat{\bSigma}_{1}^{\top}-\bSigma_{1}^{\top}\right)(\hat{\bSigma}-\bfI_{p})^{-1}+\bSigma_{1}^{\top}\left((\hat{\bSigma}-\bfI_{p})^{-1}(\bSigma-\hat{\bSigma})(\bSigma-\bfI_{p})^{-1}\right)\right\|\\
  &\le \left\|\left(\hat{\bSigma}_{1}^{\top}-\bSigma_{1}^{\top}\right)(\hat{\bSigma}-\bfI_{p})^{-1}\right\|+\left\|\bSigma_{1}^{\top}\left((\hat{\bSigma}-\bfI_{p})^{-1}(\bSigma-\hat{\bSigma})(\bSigma-\bfI_{p})^{-1}\right)\right\|\\
  &\le 2\left\|\hat{\bSigma}_{1}-\bSigma_{1}\right\|+2\left\|\bSigma_{1}\right\|\left\|\hat{\bSigma}-\bSigma\right\|\\
  &\le 4(1+\|\bSigma_{1}\|)\max\{\|\hat{\bSigma}-\bSigma\|,\|\hat{\bSigma}_{1}-\bSigma\|\}\\
  &\le 96\sqrt{K\varrho}\left(1+\|\bSigma_{1}\|\right)(\|\bSigma\|+\|\bSigma_{1}\|)\sqrt{\frac{9\er(\bSigma)+2t}{n-1}},
\end{align*}
where the second inequality is owing to $\lambda_{\min}(\hat{\bSigma}-\bfI_{p})\ge 1/2$ and $\lambda_{\min}(\bSigma-\bfI_{p})\ge 1$.
\end{proof}

We study an example of the log-Sobolev constant $K$ and $\varrho$.
Suppose that $\bfA=\sum_{i=1}^{p}a_{i}\be_{i}\be_{i}^{\top}$ for $|a_{1}|\ge |a_{2}|\ge \cdots\ge |a_{p}|$ and an orthonormal basis $\{\be_{i}\}$, and the distributions of $\bvarepsilon_{i}$ and $\bxi_{i}$ satisfy log-Sobolev inequality with constant $K_{\bvarepsilon}$ and $K_{\bxi}$ respectively.
In a similar manener to Section \ref{corr:sec:example:var}, we obtain that the distribution of $(\bX_{1},\ldots,\bX_{n})$ (that is, the latent process without the standardization by $\Cov(\bX_{i})$) satisfies a log-Sobolev inequality with constant $K_{\bxi}/(1-|a_{1}|)^{2}$ since
\begin{equation*}
    \left[\begin{matrix}
        \bX_{n}\\
        \vdots\\
        \bX_{1}
    \end{matrix}\right]
    =
    \left[\begin{matrix}
        \sum_{j=1}^{p}\be_{j}\be_{j}^{\top} & \sum_{j=1}^{p}a_{j}\be_{j}\be_{j}^{\top} & \sum_{j=1}^{p}a_{j}^{2}\be_{j}\be_{j}^{\top} & \cdots \\
        \bfO & \sum_{j=1}^{p}\be_{j}\be_{j}^{\top} & \sum_{j=1}^{p}a_{j}\be_{j}\be_{j}^{\top} &  \cdots \\
        \vdots & \vdots & \vdots & \\
        \bfO & \bfO & \bfO & \cdots
    \end{matrix}\right]
    \left[\begin{matrix}
        \bxi_{n}\\
        \bxi_{n-1}\\
        \bxi_{n-2}\\
        \vdots
    \end{matrix}\right],
\end{equation*}
and Lemma \ref{lem:causalblock} (again with an abuse $m=\infty$ for brevity) yields that the operator norm of the matrix above is bounded above by $1/(1-|a_{1}|)$.
The independence of $(\bX_{1},\ldots,\bX_{n})$ and $(\bvarepsilon_{1},\ldots,\bvarepsilon_{n})$ and the tensorization argument \citep{bakry2014analysis} lead to that the joint distribution of $(\bX_{1},\ldots,\bX_{n})$ and $(\bvarepsilon_{1},\ldots,\bvarepsilon_{n})$ satisfies a log-Sobolev inequality with constant $\max\{K_{\bxi}/(1-|a_{1}|)^{2},K_{\bvarepsilon}\}$.
For some permutation matrix $\bfP\in\R^{2pn\times2pn}$,
\begin{align*}
    \left[\begin{matrix}
        \bY_{1}\\
        \vdots\\
        \bY_{n}
    \end{matrix}\right]
    &=\left[\begin{matrix}
        \bX_{1}\\
        \vdots\\
        \bX_{n}
    \end{matrix}\right]+
    \left[\begin{matrix}
        \bvarepsilon_{1}\\
        \vdots\\
        \bvarepsilon_{n}
    \end{matrix}\right]=\left[\begin{matrix}
        \bfI_{p} & \bfI_{p} & \bfO & \bfO & \cdots& \bfO & \bfO\\
        \bfO & \bfO & \bfI_{p} & \bfI_{p} & \cdots & \bfO & \bfO \\
        \vdots & \vdots & \vdots &\vdots &  &\vdots & \vdots\\
        \bfO & \bfO &  \bfO & \bfO & \cdots &  \bfI_{p} & \bfI_{p} 
    \end{matrix}
    \right]
    \left[\begin{matrix}
        \bX_{1}\\
        \bvarepsilon_{1}\\
        \vdots\\
        \bX_{n}\\
        \bvarepsilon_{n}
    \end{matrix}\right]\\
    &=\left[\begin{matrix}
        \bfI_{p} & \bfI_{p} & \bfO & \bfO & \cdots& \bfO & \bfO\\
        \bfO & \bfO & \bfI_{p} & \bfI_{p} & \cdots & \bfO & \bfO \\
        \vdots & \vdots & \vdots &\vdots &  &\vdots & \vdots\\
        \bfO & \bfO &  \bfO & \bfO & \cdots &  \bfI_{p} & \bfI_{p} 
    \end{matrix}
    \right]\bfP
    \left[\begin{matrix}
        \bX_{1}\\
        \vdots\\
        \bX_{n}\\
        \bvarepsilon_{1}\\
        \vdots\\
        \bvarepsilon_{n}
    \end{matrix}\right]=:\bfJ\bfP\left[\begin{matrix}
        \bX_{1}\\
        \vdots\\
        \bX_{n}\\
        \bvarepsilon_{1}\\
        \vdots\\
        \bvarepsilon_{n}
    \end{matrix}\right].
\end{align*}
The matrix $[\bfI_{p}\ \bfI_{p}]\in\R^{p\times 2p}$ has the operator norm bounded above by $\sqrt{2}$ since $\|\bv_{1}+\bv_{2}\|^{2}=\sum_{i=1}^{p}(\bv_{1}^{(i)}+\bv_{2}^{(i)})^{2}\le 2\sum_{i=1}^{p}((\bv_{1}^{(i)})^{2}+(\bv_{2}^{(i)})^{2})=2$ for any $\bv_{1},\bv_{2}\in\R^{p}$ with $\|\bv_{1}\|^{2}+\|\bv_{2}\|^{2}=1$,
and thus the operator norm of $\bfJ$ is at most $\sqrt{2}$ again by Lemma \ref{lem:causalblock}.
Since $\|\bfP\|\le 1$, the log-Sobolev constant of $(\bY_{1},\ldots,\bY_{n})$ is bounded above by $2\max\{K_{\bxi}/(1-|a_{1}|)^{2},K_{\bvarepsilon}\}$.
Using $\|\bfI\otimes\bSigma^{-1/2}\|\le 1$,
we obtain that the log-Sobolev constant $K$ of $(\bSigma^{-1/2}\bY_{1},\ldots, \bSigma^{-1/2}\bY_{n})$ satisfies
\begin{equation*}
    K\le 2\max\left\{\frac{K_{\bxi}}{\left(1-|a_{1}|\right)^{2}},K_{\bvarepsilon}\right\}.
\end{equation*}
What remains is an estimate of $\varrho$.
We have
\begin{align*}
    \bSigma_{1}&=\bfA\left(\bSigma-\bfI_{p}\right)=\sum_{j=1}^{p}a_{j}\be_{j}\be_{j}^{\top}\left(\sum_{j=1}^{p}\frac{1}{1-a_{j}^{2}}\be_{j}\be_{j}^{\top}-\sum_{j=1}^{p}\be_{j}\be_{j}^{\top}\right)\\
    &=\sum_{j=1}^{p}a_{j}\be_{j}\be_{j}^{\top}\left(\sum_{j=1}^{p}\frac{a_{j}^{2}}{1-a_{j}^{2}}\be_{j}\be_{j}^{\top}\right)=\sum_{j=1}^{p}\frac{a_{j}^{3}}{1-a_{j}^{2}}\be_{j}\be_{j}^{\top},
\end{align*}
and thus we obtain the spectral representation
\begin{equation*}
    \left[\begin{matrix}
        \bSigma & \bSigma_{1}\\
        \bSigma_{1}^{\top} &\bSigma
    \end{matrix}\right]=\left[\begin{matrix}
        \bSigma & \bfO\\
        \bfO &\bSigma
    \end{matrix}\right]^{1/2}\left[\begin{matrix}
        \sum_{j=1}^{p}\be_{j}\be_{j}^{\top} & \sum_{j=1}^{p}a_{j}^{3}\be_{j}\be_{j}^{\top}\\
        \sum_{j=1}^{p}a_{j}^{3}\be_{j}\be_{j}^{\top} &\sum_{j=1}^{p}\be_{j}\be_{j}^{\top}
    \end{matrix}\right]\left[\begin{matrix}
        \bSigma & \bfO\\
        \bfO &\bSigma
    \end{matrix}\right]^{1/2}.
\end{equation*}
Without loss of generality, we set the Cartesian basis as $\{\be_{i}\}$.
Then, for any unit vector $\bu=(\bu_{1},\ldots,\bu_{2p})\in\bbS^{2p-1}$,
\begin{align*}
    \bu^{\top}\left[\begin{matrix}
        \sum_{j=1}^{p}\be_{j}\be_{j}^{\top} & \sum_{j=1}^{p}a_{j}^{3}\be_{j}\be_{j}^{\top}\\
        \sum_{j=1}^{p}a_{j}^{3}\be_{j}\be_{j}^{\top} &\sum_{j=1}^{p}\be_{j}\be_{j}^{\top}
    \end{matrix}\right]\bu
    &=\sum_{j=1}^{p}\bu_{j}^{2}+2\sum_{j=1}^{p}a_{j}^{3}\bu_{j}\bu_{p+j}+\sum_{j=1}^{p}\bu_{p+j}^{2}\\
    &\ge\sum_{j=1}^{p}\bu_{j}^{2}-2\sum_{j=1}^{p}|a_{j}|^{3}|\bu_{j}||\bu_{p+j}|+\sum_{j=1}^{p}\bu_{p+j}^{2}\\
    &\ge\sum_{j=1}^{p}\bu_{j}^{2}-2\sum_{j=1}^{p}\|\bfA\|^{3}|\bu_{j}||\bu_{p+j}|+\sum_{j=1}^{p}\bu_{p+j}^{2}\\
    &=\sum_{j=1}^{p}(1+\|\bfA\|^{3})\bu_{j}^{2}+\sum_{j=1}^{p}(1+\|\bfA\|^{3})\bu_{p+j}^{2}\\
    &\quad-\sum_{j=1}^{p}\|\bfA\|^{3}(|\bu_{j}|^{2}+2|\bu_{j}||\bu_{p+j}|+|\bu_{p+j}|^{2})\\
    &\ge (1+\|\bfA\|^{3})-2\|\bfA\|^{3}\sum_{j=1}^{p}(\bu_{j}^{2}+\bu_{p+j}^{2})\\
    &=1-\|\bfA\|^{3}.
\end{align*}
Therefore, we obtain
\begin{equation*}
    \varrho=\lambda_{\max}\left(\left[\begin{matrix}
        \sum_{j=1}^{p}\be_{j}\be_{j}^{\top} & \sum_{j=1}^{p}a_{j}^{3}\be_{j}\be_{j}^{\top}\\
        \sum_{j=1}^{p}a_{j}^{3}\be_{j}\be_{j}^{\top} &\sum_{j=1}^{p}\be_{j}\be_{j}^{\top}
    \end{matrix}\right]^{-1}\right)\le \frac{1}{1-|a_{1}|^{3}}.
\end{equation*}
As a result, we have that with probability $1-e^{-t}$,
\begin{align*}
    \|\hat{\bfA}-\bfA\|\le 96\sqrt{\frac{2}{1-|a_{1}|^{3}}\max\left\{\frac{K_{\bxi}}{\left(1-|a_{1}|\right)^{2}},K_{\bvarepsilon}\right\}}\left(1+\|\bSigma_{1}\|\right)(\|\bSigma\|+\|\bSigma_{1}\|)\sqrt{\frac{9\er(\bSigma)+2t}{n-1}}.
\end{align*}

    \putbib[cor_bibliography]
\end{bibunit}
  


\title{Dimension-free Bounds for Sums of Dependent Matrices and Operators with Heavy-Tailed Distributions}
\author{Shogo Nakakita\textsuperscript{1}, Pierre Alquier\textsuperscript{2}, and Masaaki Imaizumi\textsuperscript{1,3}\vspace{2ex}\\
{\it \textsuperscript{1}The University of Tokyo, \textsuperscript{2}ESSEC Business School,}\\
{\it \textsuperscript{3}RIKEN Center for Advanced Intelligence Project}}
\maketitle

\begin{bibunit}[apalike]

\setcounter{section}{0}
\setcounter{theorem}{0}

\begin{abstract}

We prove deviation inequalities for sums of high-dimensional random matrices and operators with dependence and {\rc heavy tails}. Estimation of high-dimensional matrices is a concern for numerous modern applications. However, most results are stated for independent observations. Therefore, it is critical to derive results for dependent and heavy-tailed matrices. In this paper, we derive a dimension-free upper bound on the deviation of the sums. Thus, the bound does not depend explicitly on the dimension of the matrices but rather on their effective rank. Our result generalizes several existing studies on the deviation of sums of matrices. It relies on two techniques: (i) a variational approximation of the dual of moment generating functions, and (ii) robustification through the truncation of the eigenvalues of the matrices. We reveal that our results are applicable to several problems, such as covariance matrix estimation, hidden Markov models, and overparameterized linear regression.
\end{abstract}

\section{Introduction}

We study non-asymptotic upper bounds on the deviations of the sums of multiple random matrices (or operators) from its expectation.
Assume that we observe a sequence of $n$ random, symmetric matrices $M_1,\ldots, M_n$ that are potentially high-dimensional, dependent, and heavy-tailed, but have a common expectation $\Sigma:= \Ep[M_\ell]$.
We are interested in evaluating the deviation of their empirical mean from the expectation $\Sigma$, measured in terms of the operator norm $\|\cdot\|$ for matrices. Specifically, we want to derive an upper bound of the following value for each integer $n$:
\begin{align*}
    \left\| \frac{1}{n} \sum_{\ell = 1}^n M_\ell - \Sigma \right\|.
\end{align*}

This problem is foundational and important; moreover, it has a variety of applications, the most typical example being the estimation of covariance matrices. 
Let $Y_1,\ldots ,Y_n$ be a sequence of random vectors; then, we can estimate its covariance matrix $\Sigma = \Ep[Y_1 Y_1^\top]$ using the empirical mean $n^{-1} \sum_{\ell = 1}^n M_\ell $ by defining $M_\ell = Y_\ell Y_\ell^\top$.
This setup can easily be applied for estimating Fisher information matrices, for example. 
Other applications include estimation of adjacency matrices of random graphs \citep{oliveira2009concentration}, signal recovery in compressed sensing \citep{donoho2006compressed}, and linear regression under overparameterization \citep{bartlett2020benign}.
Considering the increasing variety of data in modern data science, it is expedient to study the upper bound in various settings, including dependent or heavy-tailed observations $M_\ell$.

{\rc 
This problem has been actively investigated in various directions. 
The first study \citep{rudelson1999random} derived upper bounds on the operator norm of the deviation.
In the high-dimensional case, several studies
\citep{bunea2015sample,mendelson2014singular,srivastava2013covariance,koltchinskii2017concentration} derived upper bounds that do not depend on the dimensionality of the matrices $M_\ell$, referred to as \textit{dimension-free} bounds, using instead the effective rank of the matrices.
Particularly, another study \citep{giulini2018robust} investigated an infinite-dimensional version of the problem.
These results enable us to estimate high-dimensional matrices without assuming sparsity \citep{cai2010optimal} or specifying the distribution of the matrix \citep{adamczak2010quantitative,guedon2007lp,zhivotovskiy2021dimension}. 
The exact asymptotic risk is also studied by \citep{han2022exact}.
A bootstrap method and dimension-free bound are developed by \citep{lopes2023bootstrapping} for high-dimensional operators in this setting.
In the case of heavy-tailed matrices $M_\ell$, a study \citep{liaw2017simple} derived a dimension-free upper bound that clarifies how the tail property affects the bound.
The tightness of the bound is further improved by \citep{Vershynin2018High,jeong2022sub} in the heavy-tailed setting.
In the case of dependent matrices, a study \citep{han2020moment} derived a bound in expectation, following the approach of
\citep{handel2017structured}.
}

\subsection{Focus and Result}

We aim to derive a dimension-free upper bound on the deviations of the empirical mean of random matrices that are dependent and heavy-tailed. 
This setting is a generalization of the aforementioned studies. 
To handle the setup, we first introduce {\rc notation} and assumptions.
Let $\mathbb{H}$ be a Hilbert space. The first time, we only consider $\mathbb{H} = \mathbb{R}^p$. We then extend the results to an infinite-dimensional $\mathbb{H}$.

Let $M$ be a symmetric linear operator from $\mathbb{H}$ to itself.  
In dimension-free bounds, the dimension of $\mathbb{H}$ is replaced by the \textit{effective rank}, as shown by \citep{koltchinskii2017concentration}. It is defined as follows:
\begin{definition}[Effective Rank]
For a symmetric positive semi-definite {\rc trace-class} operator $M: \mathbb{H} \to \mathbb{H}$, the effective rank is defined as
\begin{align*}
    \mathbf{r}\left(M\right) := \frac{\mathrm{Tr}(M)}{\|M\|},
\end{align*}
where $\mathrm{Tr}(M)$ denotes the trace of $M$ and $\|M\|$ is its operator norm.
\end{definition}
\noindent
It can be interpreted as measuring the effective dimension of the image of $M$, which can be smaller than the actual dimension of $ \mathbb{H}$.

To measure the dependence of a sequence of matrices $M_1,\ldots, M_n$, we consider a coefficient $\Gamma_{\ell,n}$ for $\ell = 1,\ldots, n$ that bounds the martingale increment
\begin{align*}
\abs{ \mathbb{E} [ g(M_{\ell +1},\ldots ,M_n) \mid \mathcal{F}_\ell] - \mathbb{E} [ g(M_{\ell +1}, \ldots ,M_n)] } \leq \Gamma_{\ell,n},
\end{align*}
for any Lipschitz-continuous function $g(.)$, where $\mF_\ell = \sigma(M_1,\ldots,M_\ell)$ is the $\sigma$-algebra generated by  $M_1,\ldots,M_\ell$. 
The formal definition of the coefficient from \citep{rio2000inegalites, dedecker2007weak}, which has been used in many papers on dependent variables, is provided below. This coefficient leads to a general notion of dependence that includes many dependent processes, such as causal Bernoulli shifts and chains with infinite memory.
We discuss this point below.

We introduce functions $\Surv_\ell: \mathbb{R}_+ \to \mathbb{R}_+$ for $\ell = 1,\ldots,n$ to measure the tail of the distribution of $\|M_\ell\|$ as $\Surv_\ell(t) = \mathbb{P}(\|M_\ell\|\geq t)$. We also define 
\begin{align}
    G(t) := \max_{1\leq \ell\leq n} \int_{t}^{\infty} \Surv_\ell(u) {\rm d} u .
\end{align}
Our general bounds are given in terms of $\Surv_\ell(\cdot)$ and $G(\cdot)$. We can then study how the tail probability of $\|M_\ell\|$ affects the order of magnitude of the upper bound. 
Particularly, when $\|M_\ell\|$ is bounded, we obtain the rates proven in \citep{koltchinskii2017concentration,zhivotovskiy2021dimension} for independent observations, but in a broader dependent framework. 
In the case where $\Surv_\ell(t)$ and $G(t)$ decay exponentially fast in $t$, we recover the same rates up to a $\log n$ factor. Our results also provide a rate of convergence in the case where $G(t)$ decays polynomially in $t$ (in this case, the rate is slower).

Our main result takes the form of an upper bound in probability on the deviation of the empirical mean of the matrices.
Therefore, for any $t, \tau > 0$, the following inequality holds with probability at least $1-\exp(-t) - \sum_{\ell=1}^n \Surv_\ell(\tau)$:
\begin{align*}
    \left\|\frac{1}{n}\sum_{\ell =1}^{n} M_{\ell}-\Sigma\right\|\le2\sqrt{2}\left\|\Sigma\right\|\left(2\tau+ \max_{\ell = 1,\ldots,n}\Gamma_{\ell, n}\right)\sqrt{\frac{4\mathbf{r}\left(\Sigma\right)+t}{n}} + G(\tau).
\end{align*}
The results suggest that (i) we can obtain a dimension-free upper bound under quite general conditions of heavy-tail and dependence; (ii) the dependence property affects the bound through a factor $\Gamma_{\ell, n}$; and (iii) the heavy-tail property appears to affect the bound via a factor $2\tau$ and an additional term $G(\tau)$, where $\tau$ is a free parameter that can be adjusted to balance both terms. Particularly, even under a slow, polynomial decay of $G(\tau)$, we can still select $\tau=\tau_n$ such that $\sum_{\ell=1}^n \Surv_\ell(\tau) = o(1)$ and $G(\tau)=o(1)$ and then obtain an upper bound that converges to $0$, but at a rate slower than $1/\sqrt{n}$.

From a technical perspective, this study makes two contributions. 
The first is the evaluation of the moment-generating function using a variational inequality, following \citep{catoni2017dimension}. We {\rc provide an} upper bound the deviation of the sum of matrices using its moment-generating function. 
This approach was employed by \citep{zhivotovskiy2021dimension} and others. We extend it to our setting with dependent random matrices using an inequality {\rc due to}~\citep{rio2000inegalites}.
The second is the truncation technique that addresses heavy tails. This technique is classical in addressing unbounded losses in machine learning and has been used in the context of time series by~\cite{alquier2012model}.
It is also related to the influence function used in works on robust statistics~\citep{catoni2012challenging,catoni2017dimension,zhivotovskiy2021dimension,abdalla2022covariance}.
We apply this technique in our setting by truncating the eigenvalues of the dependent random matrices.
Specifically, we control the effect of the heavy tails by decomposing the deviation of the empirical mean into two parts: the deviations of the truncated mean and the deviations between the truncated and standard means.

\subsection{Organization}

Section \ref{sec:setting} introduces the setting of the problem and also provides assumptions and examples of situations where they are satisfied.
Section \ref{sec:result} presents the main results. We begin with the case of dependent but bounded matrices in Theorem~\ref{thm:main}. We then extend this result to the unbounded case in Corollary~\ref{cor:standard}.
Section \ref{sec:application} describes several applications in which we apply our bound.
Section \ref{sec:proof} contains the proofs of the main results.
Section \ref{sec:conclusion} concludes the paper.
The appendix provides the rest of the proofs.

\subsection{Notation}

Let $\mathbb{H}$ be a Hilbert space equipped with the scalar product $\left<\cdot,\cdot \right>$ and $\|\cdot\|$ be the corresponding norm. Let $\mathcal{S}$ be the set of symmetric linear continuous operators $\mathbb{H}\rightarrow\mathbb{H}$, that is, for any $(u,v)\in\mathbb{H}^2$, and for any $M\in\mathcal{S}$, $\left<Mu,v\right>=\left<u,Mv\right>< \infty$. In the special case $\mathbb{H}=\mathbb{R}^p$, $\mathcal{S}$ is simply the set of symmetric matrices. For any $M\in\mathcal{S}$ we let $\|M\|$ denote its operator norm $\|M\|=\sup_{u\in\mathbb{H},\|u\|=1}\|Mu\|$.
Throughout this paper, $\mathbf{M}=(M_\ell)_{\ell=1,\dots,n}$ is a finite random sequence of elements of $\mathcal{S}$, whose expectation is constant with $\Sigma = \mathbb{E}[M_\ell]$. 
{\rc 
Note that an expectation of a random operator $M_\ell$ is defined as a linear operator $\Sigma: \mathbb{H} \to \mathbb{H}$ satisfying $\langle u, \Sigma v \rangle = \mathbb{E}[\langle u, M_\ell v \rangle]$ for any $u,v \in \mathbb{H}$.
}
This paper aims to examine the estimation of $\Sigma$. For any $\ell\in\{0,\dots,n\}$, let $\mathcal{F}_{\ell}=\sigma(M_1,\dots,M_\ell)$.
For probability measures $P,P'$, $P \ll P'$ means that $P$ is absolutely continuous with respect to $P'$, and $\mathrm{KL}(P \| P') = \int \log (dP / dP') dP$ denotes the Kullback--Leibler divergence.
{\rc For a sequence of sets $A_1,A_2,...$, we define $\bigtimes_{i=1}^\infty A_i := A_1 \times A_2 \times \cdots$. }

\section{Dependent Matrices with Heavy-Tailed Distributions}
\label{sec:setting}

\subsection{Setup}

First, we introduce some assumptions on $\mathbf{M}=(M_\ell)_{\ell=1,\dots,n}$. The first is the dependence between the  $M_\ell$'s, which is quantified through a weak dependence coefficient. The second is the tail probability of the distribution of $\|M_\ell\|$.

\subsubsection{Dependence}
We introduce a coefficient to measure the dependence of the process of operators/matrices, which is essentially from~\citep{rio2000inegalites}. It is also discussed in \citep{dedecker2007weak}.
First, we define the set of Lipschitz functions on $\ell$ operators/matrices for $\ell \in\{ 1,\ldots,n\}$.

\begin{definition}[Lipschitz function on $\ell$ elements]
Let $E$ be a space equipped with the norm $\|\cdot\|_E$. For any $\ell\in\mathbb{N}$ and $L > 0$, we let ${\rm Lip}_{\ell}(E,L)$ denote the set of all functions $h:E^{\ell}\rightarrow \mathbb{R}$ such that for any $(a_1,\dots,a_{\ell},b_1,\dots,b_\ell)\in E^{2\ell}$,
$$ \abs{ h(a_1,\dots,a_{\ell}) - h(b_1,\dots,b_{\ell}) } \leq L \sum_{i=1}^{\ell} \|a_i - b_i\|_{E} . $$
\end{definition}
Owing to this definition, we can introduce our weak dependence condition:
\begin{assumption}[Weak dependence]
\label{asm:weakdep}
{\rc 
For any $\ell \in \{1,...,n-1\}$ and function $g\in {\rm Lip}_{n-\ell}(\mathcal{S},1)$,  $\mathbb{E}[g(M_{\ell+1},\dots,M_n) \mid \mathcal{F}_\ell]$ and $\mathbb{E}[g(M_{\ell+1},\dots,M_n)]$ exist. }
Further, there exist real numbers $(\Gamma_{\ell,n})_{1\leq \ell \leq n-1}$ such that for any $\ell\in\{1,\dots,n-1\}$ and for any function $g\in {\rm Lip}_{n-\ell}(\mathcal{S},1)$, we have
\begin{equation}
\label{eq:asm:weakdep}
\abs{ \mathbb{E}[g(M_{\ell+1},\dots,M_n) \mid \mathcal{F}_\ell] - \mathbb{E}[g(M_{\ell+1},\dots,M_n)] } \leq \Gamma_{\ell,n},
\end{equation}
almost surely. We set $\Gamma_n = \max_{1\leq \ell \leq n-1} \Gamma_{\ell,n}$.
\end{assumption}

This assumption has several noteworthy points:
(i) The coefficient used in the assumption is a generalization of the uniform mixing coefficient for bounded processes; (ii) the coefficient quantifies the dependence of the $M_\ell$'s: the larger it is, the more the matrices are dependent, while $\Gamma_n=0$ for independent matrices; and (iii) it is not comparable with the $\alpha/\beta$-mixing property: examples of non-mixing processes with small $\Gamma_n$ are known. {\rc A classical real-valued example from~\citep{andrews1984non} is given by $
M_{\ell+1} = (M_{\ell}+\varepsilon_{\ell})/2 $ where the $(\varepsilon_\ell)$ are i.i.d. from a Bernoulli distribution with parameter $1/2$. It is proven in Section 1.5 page 8 of~\citep{dedecker2007weak} that this process is not strongly mixing. On the other hand, it is quite direct to check that $\Gamma_{\ell,n} \leq 1 $.}
Essentially, when $\Gamma_n$ remains bounded for large $n$, we recover the same rates of estimation as for independent matrices.
This includes linear auto-regressive moving-average (ARMA) processes and a causal Bernoulli shifts (CBS), that are described below.
For more details, we refer the reader to \citep{rio2000inegalites,dedecker2005new,dedecker2007weak,alquier2012model}. 

In previous studies such as \citep{rio2000inegalites}, Assumption \ref{asm:weakdep} is used for bounded processes; that is, for any $\ell$, $\|M_\ell\|$ is bounded almost surely. As aforementioned, when working with unbounded processes, we begin by studying a truncated and bounded version of the process. A fact that we often use in this paper is that when $\mathbf{M}=(M_1,\dots,M_{n})$ satisfies Assumption \ref{asm:weakdep}, then so does $(f(M_1),\dots,f(M_n))$ where $f:E\rightarrow E$ is an adequate truncation function (that is, $f$ is $1$-Lipschitz).

\begin{proposition}
\label{proposition:lipschitz:dependence}
Assume that  $\mathbf{M}=(M_1,\dots,M_{n})$ satisfies Assumption \ref{asm:weakdep} and that $f:\mathcal{S} \rightarrow \mathcal{S}$ is $1$-Lipschitz. Then $(f(M_1),\dots,f(M_n))$ also satisfies Assumption \ref{asm:weakdep}.
\end{proposition}

\subsubsection{Tail Probability of the Random Matrices}

We introduce an assumption on the tail probability of $\|M_\ell\|$ that includes heavy-tailed matrices/operators.
\begin{assumption}[Tail probability]
\label{asm:tailbound}
We define, for any $t\geq 0$, $\Surv_{\ell}(t) = \mathbb{P}(\|M_\ell\|>t)$ the tail function of $\|M_{\ell}\|$. We assume that {\rc $\int_{0}^{\infty}\Surv_{\ell}(t) {\rm d} t<\infty$}, and we define
\begin{align*}
    G(t) = \max_{1\leq \ell\leq n } \int_{t}^{\infty} \Surv_{\ell}(u) {\rm d} u.
\end{align*}
\end{assumption}

\subsection{Examples} \label{sec:example}

We provide examples in which Assumptions~\ref{asm:weakdep} and~\ref{asm:tailbound} are satisfied. A recurring case of interest is the one of a stationary $\mathbb{H}$-valued stochastic process $(Y_\ell)_{\ell\in\mathbb{Z}}$, with $M_\ell = Y_{\ell} Y_{\ell}^\top $. The estimation of $\Sigma=\mathbb{E}[M_{\ell}]$ corresponds to the estimation of the covariance matrix of $(Y_\ell)_{\ell\in\mathbb{Z}}$.

\subsubsection{Independent Matrices}
Before diving into time dependence, we study the simple case where the matrices $M_{\ell}$ are independent and identically distributed. Therefore, Assumption~\ref{asm:weakdep} is trivially satisfied with $\Gamma_n=0$. Moreover, in Assumption~\ref{asm:tailbound}, we have $\Surv_\ell=\Surv_1$ for any $\ell$ and thus $G(t) = \int_{t}^{\infty} \Surv_{1}(u) {\rm d} u $.

We then consider the special case where $M_\ell = Y_{\ell} Y_{\ell}^\top $ and the $Y_\ell$ are i.i.d. Then, $\Surv_{\ell }(t)  = \mathbb{P}( \|M_\ell\| > t ) = \mathbb{P}( \|Y_\ell\|^2 > t ) $ and thus Assumption~\ref{asm:tailbound} can be checked by the study of the tails of $\|Y_\ell\|^2$. We detail three cases of interest.

\textbf{Bounded case:} if $\|Y_\ell\| \leq C$ almost surely, then $\Surv_{\ell }(t) =0 $ for any $t\geq  C^2$, and thus Assumption~\ref{asm:tailbound} is satisfied with $G(t) = 0$ for $t\geq C^2$.

\textbf{Exponential tails:} let us start with a specific example: $Y_\ell\sim \mathcal{N}(0,\Sigma)$ in $\mathbb{R}^p$. Then, by (the proof of) Lemma 1 of~\cite{laurent2000adaptive}, for all $s\in(0,1/{\rc (2||\Sigma||)})$, we have $\log \mathbb{E} \exp( s \|Y_\ell\|^2 ) \leq \frac{s^2 {\rm Tr}(\Sigma^2) }{1-2 s \|\Sigma\| }$, 
and thus it holds that
\begin{align}    
\mathbb{P}(\|Y_\ell\|^2 > t )  \leq {\rc\frac{\Ep[\exp( s \|Y_\ell\|^2 )]}{\exp(st)}} \leq \exp\left( \frac{s^2 {\rm Tr}(\Sigma^2) }{1-2 s\|\Sigma\| } - st \right).
\end{align}
We then put $s = \frac{t}{2 {\rm Tr}(\Sigma^2) + 2 \|\Sigma\| t }$ and obtain:
$$
\Surv_\ell(t) = \mathbb{P}(\|Y_\ell\|^2 > t )  \leq \exp\left( -\frac{ t^2  }{4( {\rm Tr}(\Sigma^2) +  \|\Sigma\| t) } \right). $$
Particularly, we set $t\geq \frac{{\rm Tr}(\Sigma^2)}{\|\Sigma\|} $ and obtain
$$ \Surv_\ell(t) \leq \exp\left( -\frac{ t^2  }{4( {\rm Tr}(\Sigma^2) +  \|\Sigma\| t) } \right) \leq \exp\left( -\frac{ t  }{2 \|\Sigma\|  } \right) $$
and thus $G(t) \leq 2 \|\Sigma\| \exp( -\frac{t}{2\|\Sigma\|} )$ holds.
More generally, we consider examples where $ \Surv_\ell(t) \leq \exp(-at) $ and thus $G(t) \leq \exp(-at)/a$ for some $a>0$ for large enough $t > 0$.

\textbf{Polynomial tails:} we consider a more general situation where $Y_\ell = \sqrt{R_\ell} V_\ell$ where $V_\ell$ is distributed on the unit sphere in $\mathbb{R}^p$, and $R_\ell$ is a non-negative random variable. In this case,
$$ \Surv_\ell(t) =  \mathbb{P}(\|Y_\ell\|^2 > t )  = \mathbb{P}( R_\ell > t )  $$
and thus Assumption~\ref{asm:tailbound} is satisfied if $\mathbb{P}( R_\ell > t ) =o(1/t)$ when $t\rightarrow\infty$, and we have:
$$ G(t) = \int_{t}^{\infty} \mathbb{P}( R_\ell > u )  {\rm d} u . $$
This includes exponential tails as above, where $ \mathbb{P}( R_\ell > u ) \leq \exp(-at)  $ for some $a>0$. This also includes heavier tail probabilities.
For example, if $R_\ell$ is a (shifted) Pareto random variable, $\mathbb{P}( R_\ell > t ) = \frac{1}{(t+1)^a} $, Assumption~\ref{asm:tailbound} is satisfied if $a>1$ and we have $G(t) \leq \frac{a-1}{(t+1)^{a-1}} $.

\subsubsection{Causal Bernoulli Shift}
An important category of examples is the class of causal Bernoulli shifts (CBS), which includes a large class of stochastic processes. We consider a bounded CBS first, and then define a class of unbounded processes built on CBSs.
\begin{example}[Causal Bernoulli shifts, CBS]
\label{exm:cbs}
Let $\Xi = (\xi_\ell)_{\ell\in\mathbb{Z}}$ be a sequence of bounded i.i.d. $\mathbb{H}$-valued random variables: $\|\xi_\ell\| \leq B_\xi $ almost surely. Let $C:{\rc\bigtimes_{i=1}^{\infty}\mathbb{H}} \rightarrow \mathbb{H}$ with $C(0,0,\dots)=0$. Assume that, for any $(a_1,b_1,a_2,b_2,\dots)\in {\rc\bigtimes_{i=1}^{\infty}\mathbb{H}}$ we have
$$ \| C(a_1,a_2,\dots) - C(b_1,b_2,\dots) \| \leq \sum_{\ell=1}^{\infty} \alpha_\ell \|a_\ell-b_\ell \| \text{ and } \mathcal{A} := \sum_{\ell=1}^\infty \alpha_\ell < \infty. $$
Then, we can define the stationary process $(X_\ell)_{\ell\in\mathbb{Z}}$ given by
$$ X_\ell = C(\xi_{\ell},\xi_{\ell-1},\xi_{\ell-2},\dots). $$
This process $(X_\ell)_{\ell\in\mathbb{Z}}$ is called a CBS. Note that $\|X_{\ell}\|\leq B :=  \mathcal{A} B_\xi $ almost surely.
\end{example}

CBSs include many well-known stationary and ergodic processes, such as causal ARMA.
We have the following result:
\begin{proposition}\label{prop:cbs}
Let $(Y_\ell)_{\ell\in\mathbb{Z}}$ be a CBS and {\rc let} $M_\ell =Y_\ell Y_{\ell}^\top $ for any $\ell\in\mathbb{Z}^2$.
Then, $\mathbf{M}=(M_\ell)_{\ell=1,\dots,n}$ satisfies Assumption~\ref{asm:weakdep} with $\Gamma_{\ell,n} = 4 B B_\xi \sum_{i=\ell+1}^\infty \min(i,n) \alpha_i $ and Assumption~\ref{asm:tailbound} with $G(t) =  \mathbf{1}_{\{t\leq 4B^2\}}$.
\end{proposition}
In this result, we do not consider heavy tails, because CBSs are bounded processes.
The proof of Proposition \ref{prop:cbs} is included in that of Proposition \ref{prop:assumptions}, which is about a more general class of unbounded processes.

\subsubsection{Application to Unbounded Processes}

\begin{proposition}
\label{prop:assumptions}
We assume that $(X_\ell)_{\ell\in\mathbb{Z}}$ is a CBS. Let $\mathcal{E}=(\varepsilon_\ell)_{\ell\in\mathbb{Z}}$ be a sequence of centered i.i.d. $\mathbb{H}$-valued random variables with $\Surv_\varepsilon(t) := \Pr(\|\varepsilon_\ell\|^2 \geq t)$ {\rc such that $\int_{0}^{\infty}\Surv_{\varepsilon}(t)dt<\infty$ holds}, all independent from $(X_\ell)_{\ell\in\mathbb{Z}}$. 
We define the process $(Y_\ell)_{\ell\in\mathbb{Z}}$ as
\begin{align*}
    Y_{\ell}=X_{\ell}+\varepsilon_\ell,
\end{align*}
and $M_\ell =Y_\ell Y_{\ell}^\top $ for any $\ell\in\mathbb{Z}$.
Then, $\mathbf{M}=(M_\ell)_{\ell=1,\dots,n}$ satisfies Assumption~\ref{asm:weakdep} with $\Gamma_{\ell,n} = 4 B B_\xi \sum_{i=\ell+1}^\infty \min(i,n) \alpha_i $ and Assumption~\ref{asm:tailbound} with $\Surv_{\ell}(t) \leq  \mathbf{1}_{\{t\leq 4B^2\}} + \Surv_\varepsilon(t/4)${\rc[removed some words]}.

Furthermore, if $ \Gamma := 4 B B_\xi \sum_{i=2}^\infty i \alpha_i < \infty$, then $\max_{1 \leq \ell \leq n} \Gamma_{\ell,n} \leq \Gamma$ holds.
\end{proposition}

\subsubsection{Application to Chains with Infinite Memory}

We finally discuss chains with infinite memory, which turn out to be special cases of CBSs.
\begin{example}[Chain with Infinite Memory, CIM]
Let $\Xi = (\xi_\ell)_{\ell\in\mathbb{Z}}$ be a sequence of bounded i.i.d. $\mathbb{H}$-valued random variables: $\|\xi_\ell\| \leq B_\xi $ almost surely. Let $D:{\rc\bigtimes_{i=1}^{\infty}\mathbb{H}} \rightarrow \mathbb{H}$ with $D(0,0,\dots)=0$. Assume that, for any $(a_0,b_0,a_1,b_1,a_2,b_2,\dots)\in \mathbb{H}^{\infty}$ we have
$$ \| D(a_0,a_1,a_2,\dots) - D(b_0,b_1,b_2,\dots) \| \leq   \sum_{\ell=0}^{\infty} \beta_\ell \|a_\ell-b_\ell \| \text{ and } \mathcal{B} := \sum_{\ell=1}^{\infty} \beta_\ell < 1. $$
Then, there is a stationary solution $(X_\ell)_{\ell\in\mathbb{Z}}$ to the equation \citep{doukhan2008weakly}:
$$ X_\ell = D(\xi_{\ell},X_{\ell-1},X_{\ell-2},X_{\ell-3},\dots). $$
The process $(X_\ell)_{\ell\in\mathbb{Z}}$ is called a chain with infinite memory (CIM). 
\end{example}

There is a simple connection between CBSs and CIMs.
Using Proposition 4.1 of \citep{alquier2012model}, a CIM can be rewritten as a CBS as
$$ X_\ell = C(\xi_{\ell},\xi_{\ell-1},\xi_{\ell-2},\dots) \text{ with } \alpha_\ell = \beta_0 \mathcal{B}^{\ell-1}.$$

\begin{remark}
Let us briefly discuss vector auto-regression (VAR) in this framework: $X_{\ell} = A X_{\ell - 1} + \xi_\ell$, with $A \in \mathbb{R}^{p} \otimes \mathbb{R}^{p}$.
A VAR with bounded noise terms $\xi_\ell$ is obviously a CIM, and thus Assumptions \ref{asm:weakdep} and \ref{asm:tailbound} are satisfied by such a process. They even remain satisfied by $Y_\ell = X_\ell +\varepsilon_\ell  $ for heavy-tailed $\varepsilon_\ell$'s, by using Proposition \ref{prop:assumptions}. 
However, when the noise $\xi_\ell$ in the VAR is unbouded, we need to apply another technique to handle it.
Therefore, we have to approximate the VAR by a finite-order moving-average (MA) process and show it satisfies Assumptions \ref{asm:weakdep} and \ref{asm:tailbound}.
As the approximation error vanishes when the order $k$ of the MA grows, we can apply our result without difficulty.
\end{remark}

\section{Main Results} \label{sec:result}

We introduce our main result in stages. 
First, we consider the case where $M_\ell$ is a $p \times p$ matrix with the bounded property, and then we extend it to the unbounded and heavy-tailed cases. 
Finally, we extend the result to the case where $M_\ell$ is an operator between infinite-dimensional spaces.

\subsection{Result on $p$-Dimensional Matrix}

\subsubsection{Bounded Case}

We first consider the case $\mathbb{H} = \mathbb{R}^p$, with $p \in \mathbb{N}$, and the matrices $\|M_\ell\|$ are bounded for all $ \ell = 1,\ldots,n$.
Obviously, $M_\ell$ is not heavy-tailed in this case; thus, the main contribution here is to handle the dependence in $\mathbf{M}$.

The derivation of this result starts with the variational inequality: with a probability measure $\mu$ on a parameter space $\Theta$,
and with a random parameter $\theta$ in $\Theta$ and a random variable $X$, 
it holds that with probability at least $1-\exp(-t)$, for any probability measure $\rho\ll \mu$ and any measurable function $h$,
\begin{align*}
    \Ep_{\rho}[h\left(X,\theta\right)]\le \Ep_{\rho}\left[\log\Ep_{X}\left[\exp\left(h\left(X,\theta\right)\right)\right]\right]+ \mathrm{KL}\left(\rho \| \mu\right)+t
\end{align*}
where $ \Ep_{\rho}$ denotes the expectation with respect to $\theta$ under the distribution $\rho${\rc; note that the left-hand side is still random}.
This result is taken from \citep{catoni2017dimension} and \citep{zhivotovskiy2021dimension}.
In our setting, $X=\mathbf{M}$ and we control $\Ep_{X}\left[\exp\left(h\left(X,\theta\right)\right)\right]$ thanks to an inequality by \citep{rio2000inegalites} for dependent matrices (these steps are detailed in the proofs below). We obtain 
\begin{align*}
    \mathbb{E} \left[ \exp\left( \lambda h(\mathbf{M})- \lambda \mathbb{E}[h(\mathbf{M})] \right) \right]
    \leq \exp\left(\frac{\lambda^2 L^2 \sum_{\ell=1}^n \left( 2\kappa + \Gamma_{\ell,n} \right)^2 }{8 n^2}\right).  
\end{align*}
An adequate choice of $h$ leads to our main result: the concentration bound for the estimation of $\Sigma$ using the empirical mean of $\mathbf{M}$.
\begin{theorem} \label{thm:main}
Assume that $\mathbf{M}$ is a sequence of positive semi-definite symmetric random $p \times p$ matrices such that, for some $\kappa>0$, for all $\ell=1,\ldots,n$, $\Ep\left[M_{\ell}\right]=\Sigma$ and $\left\|M_{\ell}\right\|\le \kappa^{2}$ almost surely.
Under Assumption~\ref{asm:weakdep}, for all $t>0$, with probability at least $1-\exp(-t)$, we have
\begin{align*}
    \left\|\frac{1}{n}\sum_{\ell=1}^{n} M_{\ell}-\Sigma\right\|\le{4}\sqrt{2}\left\|\Sigma\right\|\left({ \kappa^{2}+\Gamma_{n}}\right)\sqrt{\frac{4\mathbf{r}\left(\Sigma\right)+t}{n}}.
\end{align*}
\end{theorem}
Let us comment briefly on this result.
First, this is a dimension-free bound in which $p$ does not appear. The statistical dimension is instead described by the effective rank $\mathbf{r}\left(\Sigma\right)$. 
This is identical to the statistical dimension of the independent case of \citep{koltchinskii2017concentration} and others.
Then, the effect of this dependence appears as $\Gamma_{n}$ in the factor $(2\kappa^{2}+\Gamma_{n})$ of the upper bound. 
If $\mathbf{M}$ is independent, we have $\Gamma_{n} = 0$. More generally, we described above a large class of processes where  $\Gamma_n$ is {\rc bounded from above by} a constant $\Gamma$. In both cases, our upper bound matches the one of \citep{koltchinskii2017concentration} up to constants.

\subsubsection{Heavy-Tailed Case}

We then extend Theorem~\ref{thm:main} to unbounded, possibly heavy-tailed matrices $M_\ell$.

The general idea is to apply Theorem~\ref{thm:main} to a sequence of transformed matrices $\{f(M_1),\ldots,f(M_n)\}$ where $f:\mathcal{S}\rightarrow\mathcal{S}$ is a bounded function, such that  $\sup_{ M \in \mS}\|f(M)\|\leq \tau$.
This application yields a bound on $\| \frac{1}{n}\sum_{\ell=1}^n f(M_\ell) - \mathbb{E}[f(M_\ell)] \|$. Then, we handle the effect of $f$, that is, $\| \frac{1}{n}\sum_{\ell=1}^n f(M_\ell) - \frac{1}{n}\sum_{\ell=1}^n M_\ell\|$ and $\| \mathbb{E}[f(M_\ell)] - \Sigma \|$, to obtain an upper bound on $\|\frac{1}{n}\sum_{\ell=1}^{n} M_{\ell}-\Sigma\|$.
This results in the introduction of an additional term depending on $\tau$ in the upper bound. 
This technique leads to the following results.
\begin{corollary} \label{cor:standard}
Assume that $\mathbf{M}$ is a sequence of $p \times p$ {\rc positive semi-definite}, symmetric, random matrices with $\Ep\left[M_{\ell}\right]=\Sigma$, which satisfies Assumptions~\ref{asm:weakdep} and \ref{asm:tailbound}. 
For any $\tau>0$ and for all $t>0$, with probability at least $1-\exp(-t)-\sum_{\ell=1}^n \Surv_\ell(\tau)$ it holds that
\begin{align*}
    \left\|\frac{1}{n}\sum_{\ell=1}^{n} M_{\ell}-\Sigma\right\|\le{4}\sqrt{2}\left\|\Sigma\right\|\left({\tau+\Gamma_{n}}\right)\sqrt{\frac{4\mathbf{r}\left(\Sigma\right)+t}{n}} + G(\tau).
\end{align*}
\end{corollary}

First, note that the bound holds with probability $1-\exp(-t)-\sum_{\ell=1}^n \Surv_\ell(\tau)$. If $\tau$ is constant and $\Surv_\ell(\tau)>0$, then $\sum_{\ell=1}^n \Surv_\ell(\tau)$ can grow to $\infty$ when $n\rightarrow\infty$, and the statement becomes vacuous for large $n$. However, by letting $\tau=\tau_n \rightarrow\infty$, and if the $\Surv_\ell$'s decrease fast enough, we are able to keep $\sum_{\ell=1}^n \Surv_\ell(\tau) $ small enough (for example, smaller than $1/n$).

The effect of the {\rc heavy-tailed} $G(\tau)$ also appears additively in the second term of the derived upper bound. Here again, by letting $\tau=\tau_n \rightarrow\infty$, we can make the term $G(\tau)$ small enough.
The tightness of the bound, of course, depends on how far we are from the boundedness assumption, that is, on how fast the function $G$ decreases.
We provide the following examples:

\textbf{Bounded Case}: Assume that $\|M_i\|\leq \kappa$ almost surely for some $\kappa > 0$, then Assumption~\ref{asm:tailbound} is satisfied with $G(\tau) =0 $ for $\tau\geq \kappa$. 
Thus, we can take $\tau=\kappa$ and recover exactly Theorem~\ref{thm:main}.

\textbf{Exponential-Tail Case}: Assume that $\Surv_\ell(\cdot)$ has an exponential decay, that is, there is an $a>0$ such that for any $\ell$, $\Surv_\ell(t) {\rc\le} \exp(- a t) $. Notably, $G(t) \leq  \exp(- a t)/a $.
Thus, Corollary \ref{cor:standard} states that with probability at least $1-\exp(-t)-n  \exp(-a \tau) $,
    \begin{align*}
    \left\|\frac{1}{n}\sum_{\ell=1}^{n} M_{\ell}-\Sigma\right\|\le  4\sqrt{2}\left\|\Sigma\right\|\left( \tau+\Gamma_{n}\right)\sqrt{\frac{4\mathbf{r}\left(\Sigma\right)+t}{n}} + \frac{\exp\left( - a \tau \right)}{a}.
\end{align*}
For some $\alpha>1$, we put $\tau = \tau_n = \frac{\alpha \log n}{a}$ which implies that $n \sum_{\ell=1}^n \Surv_\ell(\tau_n) \leq \frac{1}{n^{\alpha-1}} $, and we set $t = \log ( \delta^{-1})$.
Subsequently, for every $\delta \in (0,1)$, with probability at least $1-\delta-\frac{1}{n^{\alpha-1}}$, we have
    \begin{align*}
        \left\|\frac{1}{n}\sum_{\ell=1}^{n} M_{\ell}-\Sigma\right\|
    &\le 4\sqrt{2}\left\|\Sigma\right\|\left( \frac{\alpha \log n }{a}+\Gamma_{n}\right)\sqrt{\frac{4\mathbf{r}\left(\Sigma\right)+\log (\delta^{-1})}{n}}  + \frac{1}{a n^{\alpha}} .
\end{align*}
In this upper bound, the effect of the heavy-tail appears in the second term in $1/(an^\alpha)$, which is negligible {\rc with respect to} the first term (because $\alpha$ is chosen $>1$). The main difference with the bounded case
 is the factor $( \frac{\alpha \log(n)}{a} + \Gamma_{n})$ in the first term, which increases in $\log(n)$. Thus, the dependence in $n$ and in the statistical dimension are similar to the ones in \citep{koltchinskii2017concentration,zhivotovskiy2021dimension} up to an additional $\log(n)$ factor.


\textbf{Polynomial-Tail Case}: 
Assume that {\rc$\Surv_\ell$} has a polynomial decay $\Surv_\ell(t) {\rc\le}  a t^{-b} $ with $a > 0$ and $b>2$. Then, $G(t) \leq \frac{a}{b-1} t^{1-b} $. Thus, with $\tau = \tau_n$, the bound is, with probability at least $1-\exp(-t)-n a \tau_n^{-b}$,
\begin{align*}
    \left\|\frac{1}{n}\sum_{\ell=1}^{n} M_{\ell}-\Sigma\right\|\le 4\sqrt{2}\left\|\Sigma\right\|\left( \tau_n+\Gamma_{n}\right)\sqrt{\frac{4\mathbf{r}\left(\Sigma\right)+t}{n}} + \frac{a}{b-1} \tau_n^{1-b}.
\end{align*}
Here, for some $\alpha>1$, we take $ \tau_n = a^{1/b} (n)^{\alpha /b} $, and also set $t = \log \delta^{-1}$.
Then, for any $\delta \in (0,1)$, we obtain that with probability at least $1-\delta-1/n^{\alpha-1} $,
\begin{align*}
    &\left\|\frac{1}{n}\sum_{\ell=1}^{n} M_{\ell}-\Sigma\right\|\\
    &\le 4\sqrt{2}\left\|\Sigma\right\|\left(a^{1/b}n^{\alpha/b} +\Gamma_{n}\right)\sqrt{\frac{4\mathbf{r}\left(\Sigma\right)+ \log(2\delta^{-1})}{n}} + \frac{1}{(b-1) n^{\alpha-1}}  .
\end{align*}
The rate in the first term is more seriously deteriorated. However, we still have convergence as soon as $1<\alpha<b/2$, which is possible only in the case $b>2$. Our rate is not as sharp as the one in~\cite{srivastava2013covariance} for heavy-tailed matrices in the independent case. We are not aware of how to extend the work of \cite{srivastava2013covariance} to dependent matrices, and claim that our result is {\rc the first rate} obtained on matrices that are simultaneously heavy-tailed and dependent.
\begin{remark}[Comparison]
    We discuss the comparison between Corollary \ref{cor:standard} and the analysis of the case with independent matrices by \cite{zhivotovskiy2021dimension,abdalla2022covariance}.
    Corollary 5 does not always recover the rates that are known in the i.i.d setting; however, the techniques used \cite{zhivotovskiy2021dimension,abdalla2022covariance} strongly rely on the independence assumption.
    Under specific assumptions, we make the following findings: 
    (i) In the bounded case, we recover the same rates as the previous studies, extending them from i.i.d to the non-i.i.d setting for free.
    (ii) In the exponential tail case, we recover these rates up to a $\log(n)$ factor, which we interpret as a cost of extending them to the dependent setting.
    (iii) In the polynomial tail case, we admit that we have a slower rate {\rc than the one above}. However, we are not aware of any work that tackles simultaneously heavy tails and time dependence. The fact that we obtain a rate of convergence here, even if it is slow, is already a contribution.
\end{remark}

\subsection{Result on Infinite-Dimensional Operator}
Here, we consider the case of an infinite-dimensional separable Hilbert space $\mathbb{H}$, which has also been studied in \citep{koltchinskii2017concentration,giulini2018robust}.
Following \citep{giulini2018robust}, we extend our result for the $p$-dimensional setting to the infinite-dimensional case.

The idea is to find a finite-dimensional approximation of the spectral norm of operators using an orthonormal basis.
Let $(e_{j})_{j \in \mathbb{N}}$ be an orthonormal basis of $\mathbb{H}$ and $\mathbb{H}_{k}:=\mathrm{span}\left\{e_{1},\ldots,e_{k}\right\}$.
For an $\mathbb{H}\otimes\mathbb{H}$-valued random operator $M_{\ell}$, let $( M_{\ell}^{\left(j_{1},j_{2}\right)})_{j_1,j_2=1}^k$ be a sequence of real-valued random variables such that $M_{\ell}^{\left(j_{1},j_{2}\right)}:=\langle M_{\ell}e_{j_{1}},e_{j_{2}}\rangle$.
We see that
\begin{align*}
    &\sup_{u_{k}\in \mathbb{H}_{k} : \|u_{k}\|=1} \bigabs{ \left\langle \left(\frac{1}{n}\sum_{\ell=1}^{n}M_{\ell}-\Sigma\right)u_{k},u_{k}\right\rangle }\\
    &=\sup_{\substack{u_{k}^{(j)}\in\mathbb{R},j=1,\ldots,k\\\sum_{j=1}^{k}(u_{k}^{(j)})^{2}=1}}\bigabs{\frac{1}{n}\sum_{i=1}^{n}\sum_{j_{1}=1}^{k}\sum_{j_{2}=1}^{k}u_{k}^{\left(j_{1}\right)}u_{k}^{\left(j_{2}\right)}\left(M_{\ell}^{\left(j_{1},j_{2}\right)} - \Ep\left[M_{\ell}^{\left(j_{1},j_{2}\right)}\right] \right) }. 
\end{align*}
Then, the right-hand side is a spectral norm of the difference between the sampled \textit{matrix} and the population one, to which Theorem \ref{thm:main} and Corollary \ref{cor:standard} are applicable.
Based on this approach, and considering the limit $k \to \infty$, we obtain the following result:
\begin{theorem} \label{thm:operator}
Assume that $\mathbf{M}$ is a sequence of positive, semi-definite, symmetric, $\mathbb{H}\otimes\mathbb{H}$-valued random operators with $\Ep\left[M_{\ell}\right]=\Sigma$, and also satisfies Assumptions~\ref{asm:weakdep} and \ref{asm:tailbound}. 
For any $\tau>0$ and for all $t>0$, with probability at least $1-\exp(-t)- \sum_{\ell=1}^n \Surv_\ell(\tau)$, it holds that
\begin{align*}
    \left\|\frac{1}{n}\sum_{\ell=1}^{n} M_{\ell}-\Sigma\right\|\le{ 4}\sqrt{2}\left\|\Sigma\right\|\left({ \tau+\Gamma_{n}}\right)\sqrt{\frac{4\mathbf{r}\left(\Sigma\right)+t}{n}} + G(\tau).
\end{align*}
\end{theorem}
The obtained upper bound remains the same, even for infinite dimensions. Our approach in the finite-dimensional case cannot be applied directly in the infinite-dimensional case. This is because our proof using variational equalities depends on a density function of $p$-dimensional Gaussian vector. {\rc Thus, we cannot avoid first considering $\mathbb{H}_k$ and subsequently letting $k\rightarrow\infty$}.

\section{Applications} \label{sec:application}

\subsection{Covariance Operator Estimation} \label{sec:ex_covariance}

We consider the problem of covariance operator estimation using dependent samples with heavy tails {\rc under the setting and assumptions of Proposition \ref{prop:assumptions}}.
Let $(X_\ell)_{\ell \in \mathbb{N}}$ be a CBS in $\mathbb{H}$ and consider the strongly stationary process  $(Y_\ell)_{\ell \in \mathbb{Z}}$, given by
\begin{align*}
    Y_\ell = X_\ell + \varepsilon_\ell,
\end{align*}
as in Proposition \ref{prop:assumptions}.
Additionally, assume that $\Ep[X_1] = 0$ and its covariance operator is $\Sigma \in \mathcal{S}$; that is, $\Sigma$ is defined as $\Sigma u = \Ep[ \langle Y_1, u \rangle Y_1]$ for any $u \in \mathbb{H}$.

{\rc Assume that }we have $n$ observations $\mathbf{Y} = (Y_\ell)_{\ell = 1,\ldots,n}$ from the process $(Y_\ell)_{\ell \in \mathbb{Z}}$. Then, we define the empirical covariance operator:
\begin{align*}
    M_\ell u :=\langle Y_\ell, u \rangle Y_\ell,
\end{align*}
for any $u \in \mathbb{H}$.
Using this notion, we obtain $n$ operators $\mathbf{M}$ from $\mathbf{Y}$ and then obtain the empirical covariance operator as 
\begin{align}
    \hat{\Sigma} := \frac{1}{n} \sum_{\ell=1}^n M_\ell. \label{def:cov_operator}
\end{align}
By a direct application of Corollary \ref{cor:standard}, we obtain the following result, stated without proof.
\begin{proposition} \label{prop:covariance}
{\rc Assume that the sequence $\mathbf{M}$ satisfies the setting of Proposition \ref{prop:assumptions}.}
Consider the empirical covariance operator defined in \eqref{def:cov_operator}.
Then, for any $\tau>0$ and $t>0$, the following inequality holds with probability at least $1-\exp(-t)-\sum_{\ell=1}^n \Surv_\ell(\tau)$:
\begin{align*}
    \|\hat{\Sigma}-\Sigma\|\le { 4}\sqrt{2}\left\|\Sigma\right\|\left({\tau+\Gamma_{n}}\right)\sqrt{\frac{4\mathbf{r}\left(\Sigma\right)+t}{n}} + G(\tau),
\end{align*}
where $\Surv_\ell(\tau) =  \mathbf{1}_{\{\tau \leq 4B^2\}} +  \Surv_\varepsilon(\tau/4)$, and $G(\tau) = \int_{\tau}^{\infty} \Surv_\ell(t) {\rm d}t$.
\end{proposition}

\subsection{Lagged Covariance Matrix Estimation} \label{sec:lagged_covariance}
We consider the estimation of a lagged covariance matrix, which is also called a cross-covariance matrix.
Consider the same process $(Y_\ell)_{\ell \in \mathbb{Z}}$ as in Section \ref{sec:ex_covariance}. Here, we aim to estimate 
\begin{align*}
    \Sigma_1 := \mathbb{E}[Y_{\ell} Y_{\ell+1}^\top],
\end{align*}
from $n$ observations $\mathbf{Y} = (Y_\ell)_{\ell=1,\ldots,n}$.
This problem and the solution discussed below {\rc can obviously be extended to} $\Sigma_h := \mathbb{E}[Y_{\ell} Y_{\ell+h}^\top]$ for $h \geq 2$. 
Note that $\Sigma_1$ is not symmetric; hence, our main results cannot be directly applied to a naive estimator, $\hat{\Sigma}_1 := (n-1)^{-1} \sum_{\ell=1}^{n-1} Y_\ell Y_{\ell + 1}^\top$.
We still denote $\Sigma =\mathbb{E}[Y_{\ell} Y_{\ell}^\top]$, which is shown as $\Sigma_h$ for $h=0$, and its empirical estimator $\hat{\Sigma} := (n-1)^{-1} \sum_{\ell=1}^{n-1} Y_\ell Y_\ell^\top$.

To estimate $\Sigma_1$, we define an augmented process and estimator for the covariance matrix of the process.
We define the Hilbert space $\mathbb{H}^2$ equipped with the scalar product $\langle(y_1,y_2),(y_1',y_2')\rangle = \langle y_1,y_1'\rangle  + \langle y_1,y_2'\rangle $ for $(y_1,y_2), (y'_1,y'_2) \in \mathbb{H}^2$. 
Let $\tilde{Y}_{\ell} = (Y_\ell,Y_{\ell+1})^\top$ be the $\mathbb{H}^2$-valued augmented process, whose covariance is
\begin{align}
    \Sigma_{0:1} := \mathbb{E}\left[\tilde{Y}_{\ell} \tilde{Y}_{\ell}^\top \right] = \left(
\begin{array}{c c}
\mathbb{E}[Y_{\ell} Y_{\ell}^\top] & \mathbb{E}[Y_{\ell} Y_{\ell+1}^\top] \\
\mathbb{E}[Y_{\ell + 1} Y_{\ell}^\top] & \mathbb{E}[Y_{\ell + 1} Y_{\ell + 1}^\top]
\end{array}
\right) = \left(
\begin{array}{c c}
\Sigma_0 & \Sigma_1 \\
\Sigma_1^\top &\Sigma_0
\end{array}
\right). \label{def:lagged_cov}
\end{align}
The main idea is to estimate $\Sigma_{0:1}$, which directly leads to an estimator of $\Sigma_1$.

Using observations $\mathbf{Y}$, we build $\tilde{Y}_{1},\ldots,\tilde{Y}_{n-1}$ and their sample-wise product matrices $M_1,\ldots,M_{n-1}$ as 
$$ M_\ell :=\tilde{Y}_{\ell} \tilde{Y}_{\ell}^\top = \left(
\begin{array}{c c}
Y_{\ell} Y_\ell^\top & Y_{\ell} Y_{\ell + 1}^\top \\
Y_{\ell + 1} Y_{\ell}^\top & Y_{\ell + 1} Y_{\ell + 1}^\top
\end{array}
\right).
$$
We then construct {\rc an estimator}
\begin{align}
    \hat{\Sigma}_{0:1} := \frac{1}{n-1} \sum_{\ell=1}^{n-1} M_\ell =   \left(
\begin{array}{c c}
\hat{\Sigma} & \hat{\Sigma}_1 \\
\hat{\Sigma}_1^\top & \hat{\Sigma}
\end{array}
\right). \label{def:lagged_estimator}
\end{align}
We show a concentration inequality for $\hat{\Sigma}_{0:1}$ and additionally show the convergence of $\hat{\Sigma}$ and $\hat{\Sigma}_1$.
\begin{proposition} \label{prop:lag}
    Assume that $\mathbf{Y}$ is as in Proposition~\ref{prop:assumptions}.
    Consider the matrices in \eqref{def:lagged_cov} and the estimator in \eqref{def:lagged_estimator}.
    Then, for any $\tau>0$ and $t>0$, the following inequality holds with probability at least $1-\exp(-t)-\sum_{\ell=1}^{n-1} \Surv_\ell(\tau)$:
    \begin{align*}
    \|\hat{\Sigma}_{0:1} - \Sigma_{0:1} \|
    \le 4\sqrt{2}\left(\|\Sigma_{0:1}\|\right) \left({\tau}+\Gamma_{n}\right)\sqrt{\frac{4 \mathbf{r}\left(\Sigma_{0:1}\right) +t}{n - 1}} + G(\tau),
    \end{align*}
where $\Surv_\ell(\tau) =  \mathbf{1}_{\{\tau \leq 4B^2\}} +  \Surv_\varepsilon(\tau/4)$, and $G(\tau) = \int_{\tau}^{\infty} \Surv_\ell(t) {\rm d}t$.
    Furthermore, with the same probability, we obtain
    \begin{align*}
    &\max\{\|\hat{\Sigma}-\Sigma\|, \| \hat{\Sigma}_1 - \Sigma_1\|\}
    \\
    &\le 4\sqrt{2}\left(\|\Sigma_1\| + \|\Sigma\|\right) \left({\tau}+\Gamma_{n}\right)\sqrt{\frac{4 \mathbf{r}\left(\Sigma_{0:1}\right)  +t}{n - 1}} + G(\tau) .
\end{align*}
\end{proposition}
The first statement is simply an application of Theorem \ref{thm:main} to the estimator $\hat{\Sigma}_{0:1}$.
The second simply follows from the first statement together with the facts $\left\|\Sigma_{0:1}\right\| \leq \|\Sigma_1\| + \|\Sigma\| $. 
By using the relation $\mathrm{\rc Tr}(\Sigma_{0:1}) =2\mathrm{\rc Tr}(\Sigma) = 2 \|\Sigma\| \mathbf{r}\left(\Sigma\right)$, we obtain the result.

\subsection{Linear Hidden Markov Model}
We consider a linear hidden Markov model (HMM) and study estimation in this model.
Specifically, we consider a HMM model with a lag order $1$ and set $\mathbb{H} = \mathbb{R}^{p}$.
Assume that we observe a sequence of $p$-dimensional random vectors $ \mathbf{Y} = (Y_\ell)_{\ell=0}^n$ which follows the following equations for $\ell \in \mathbb{Z}$:
\begin{align}
    Y_{\ell} &= X_{\ell} + \varepsilon_\ell,  \label{def:hmm1}\\
    X_{\ell} &= A X_{\ell - 1} + \xi_\ell \label{def:hmm2}
\end{align}
where $A \in \mathbb{R}^{p \times p}$ is an unknown parameter matrix such that $\|A\| \in (0,1)$, and $(X_\ell)_{\ell \in \mathbb{Z}}$ is a latent process such that $\|X_\ell\| \leq B$ almost surely.
Here, $(\varepsilon_\ell)_{\ell \in \mathbb{Z}}$ is a sequence of i.i.d. $p$-dimensional noise variable with zero mean and finite variance, and $\xi_\ell$ is a sequence of  i.i.d. $p$-dimensional bounded noise variable with zero mean such that $\|\xi_\ell\| \leq B_\xi$ almost surely.
Under the condition $\|A\| \in (0,1)$ and the boundedness of the $\xi_\ell$'s, the upper bound $B$ is guaranteed to be finite.
For brevity, we assume that $\Ep[\varepsilon_\ell \varepsilon_\ell^\top] = \Ep[\xi_\ell \xi_\ell^\top] = I$.
We also define the covariance matrix $\Sigma:= \Ep[Y_\ell Y_\ell^\top]$ and the lagged covariance matrix $\Sigma_1:= \Ep[Y_{\ell + 1} Y_\ell^\top]$.
Here, we aim to estimate the unknown parameter matrix $A$.

We study a convenient form of the HMM model.
We define noise matrices $\mathsf{E} = (\varepsilon_0,\ldots,\varepsilon_{n-1}) \in \mathbb{R}^{p \times n}$, $\mathsf{E}_+ = (\varepsilon_1,\ldots,\varepsilon_n) \in \mathbb{R}^{p \times n}$, and $\mathsf{Z} = (\xi_1,\ldots,\xi_n) \in \mathbb{R}^{p \times n}$ and also define matrices $\mathsf{Y} = (Y_0,\ldots,Y_{n-1}) \in \mathbb{R}^{p \times n}$ and $\mathsf{Y}_+ = (Y_1,\ldots,Y_{n}) \in \mathbb{R}^{p \times n}$.
Then, we rewrite \eqref{def:hmm1} and \eqref{def:hmm2} as
\begin{align*}
    (\mathsf{Y}_+ - \mathsf{E}_+) = A (\mathsf{Y} - \mathsf{E}) + \mathsf{Z}.
\end{align*}
Multiplying $(\mathsf{Y} - \mathsf{\rc E})^\top$ on both sides from the right and taking an expectation yields
\begin{align*}
    A &= (\Ep[(\mathsf{Y}_+ - \mathsf{E}_+) (\mathsf{Y} - \mathsf{E})^\top] - \Ep[\mathsf{Z} (\mathsf{Y} - \mathsf{E})^\top]) \Ep[(\mathsf{Y} - \mathsf{E}) (\mathsf{Y} - \mathsf{E})^\top]^{-1}\\
    &= \Sigma_{1} (\Sigma + I)^{-1}.
\end{align*}
Here, we utilize the independent properties of the noise, and $\Ep[\varepsilon_\ell \varepsilon_\ell^\top]=I$.

We then define an estimator of $A$.
Using the estimators $\hat{\Sigma} := \mathsf{Y} \mathsf{Y}^\top / n = n^{-1} \sum_{\ell=0}^{n-1} Y_\ell Y_\ell^\top$ and $\hat{\Sigma}_{1} := \mathsf{Y}_+ \mathsf{Y}^\top / n = n^{-1} \sum_{\ell=0}^{n-1} Y_{\ell+1} Y_\ell^\top$,  we define the following estimator:
\begin{align}
    \hat{A} :=  \hat{\Sigma}_{1} (\hat{\Sigma} +  I)^{-1}. \label{def:hmm_estimator}
\end{align}
Then, we obtain the following result:
\begin{proposition} \label{prop:hmm}
    Consider the HMM model \eqref{def:hmm1}-\eqref{def:hmm2} and the estimator in \eqref{def:hmm_estimator} for the parameters in the model.
    Then, for any $t > 0$ and $\tau > 0$, with probability at least $1-\exp(-t)-\sum_{\ell=1}^{n-1} \Surv_\ell(\tau)$, the following inequality holds:
\begin{align*}
    &\|\hat{A} - A\| \\
    &\leq 4\sqrt{2}\left(\|\Sigma_1\| + \|\Sigma\|\right) (1 + \|\Sigma_1\|)\left(\tau+\Gamma_{n}\right)\sqrt{\frac{4 \mathbf{r}\left(\Sigma_{0:1}\right)  +t}{n - 1}} + (1 + \|\Sigma_1\|) G(\tau),
\end{align*}
where $\Sigma_{0:1}$ is defined in \eqref{def:lagged_cov}, $\Surv_\ell(\tau) =  \mathbf{1}_{\{\tau \leq 4B^2\}} +  \Surv_\varepsilon(\tau/4)$, and $G(\tau) = \int_{\tau}^{\infty} \Surv_\ell(t) {\rm d}t$.
\end{proposition}
It is obtained by bounding the estimation error $\|\hat{A} - A\|$ with the estimation errors of the covariance matrix $\Sigma$ and the lagged covariance matrix $\Sigma_1$, as described in Proposition \ref{sec:lagged_covariance}. 
Note that it is possible to extend the number of lags in this HMM model to more than $1$.

\subsection{Overparameterized Linear Regression}

Here, we study a linear regression problem with dependent and heavy-tail covariates in the overparameterization framework developed by \citep{bartlett2020benign}.

Let $(X_\ell)_{\ell \in \mathbb{Z}}$ be a CBS as a $\mathbb{H}$-valued latent process and $(Y_\ell)_{\ell \in \mathbb{Z}}$ be a generated process as  a $\mathbb{H}$-valued covariate such that
\begin{align}
    Y_\ell &= X_\ell + \varepsilon_\ell \label{def:lin_reg_latent},
\end{align}
where $\varepsilon_\ell$ is an i.i.d. $\mathbb{H}$-valued noise variable with a mean value of zero.
Additionally, we define $\theta^* \in \mathbb{H}$ as a true unknown parameter and a covariance operator $\Sigma = \Ep[Y_\ell Y_\ell^\top]$.
For $\ell \in \mathbb{Z}$, we consider an $\mathbb{R}$-valued random variable $Z_\ell$ called the response variable, given by:
\begin{align}
    Z_\ell &= \langle \theta^*, Y_\ell \rangle + U_\ell, \label{def:lin_reg_model}
\end{align}
where $U_\ell$ is an $\mathbb{R}$-valued independent random variable with mean zero and a variance  $\sigma^2 > 0$.

The goal of the regression problem is to estimate $\theta^*$ from the observations {\rc $\{(Z_i, Y_i):i=1,\ldots,n\}$}.
We introduce a design matrix and operator as $\mathsf{Z} = (Z_1,\ldots,Z_n)^\top \in \mathbb{R}^{n}$ and $\mathsf{Y}: \mathbb{H} \to \mathbb{R}^n$ such that {\rc $\mathsf{Y} \theta = ( \langle Y_1, \theta \rangle ,\ldots, \langle Y_n, \theta \rangle )^\top \in \mathbb{R}^n$ holds} for $\theta \in \mathbb{H}$.
{\rc Similarly, with $E = (e_1,...,e_n)^\top \in \mathbb{R}^n$, we define an operator $\mathsf{Y}^\top : \mathbb{R}^n \to \mathbb{H}$ such that $\mathsf{Y}^\top E = \sum_{i=1}^n e_i Y_i$.}
{\rc Further, we define an empirical covariance operator $\hat{\Sigma} : \mathbb{H} \to \mathbb{H}$ as $\hat{\Sigma} = \mathsf{Y}^\top \mathsf{Y} / n$ and a projection operator $\Pi_\mathsf{Y}: \mathbb{H} \to \mathbb{H}$ as $\Pi_\mathsf{Y} := \mathsf{Y}^\top (\mathsf{Y} \mathsf{Y}^\top)^{-1} \mathsf{Y}$.}

To estimate $\theta^*$, we define the minimum norm estimator as:
\begin{align}
    \hat{\theta} = \argmin_{\theta \in \mathbb{H}} \{\|\theta\|^2: \mathsf{Y}^\top \mathsf{Y} \theta = \mathsf{Y}^\top \mathsf{Z}\} = \mathsf{Y}^\top (\mathsf{Y} \mathsf{Y}^\top)^\dagger \mathsf{Z}, \label{def:estimator_regression}
\end{align}
where $^\dagger$ denotes the pseudo-inverse of operators.
The excess risk of  $\hat{\theta}$ is measured using 
\begin{align}
    R(\hat{\theta}) := \Ep_{(Z_*, Y_*)}[(Z_* - \langle Y_*, \hat{\theta} \rangle)^2 - (Z_* - \langle Y_*, {\theta}^* \rangle)^2], \label{def:risk_regression}
\end{align}
where $(Z_*, Y_*)$ is an i.i.d. copy of $(Z_1,Y_1)$ from the regression model \eqref{def:lin_reg_model} and $\Ep_{(Z_*, Y_*)}[\cdot]$ is the expectation with respect to $(Z_*, Y_*)$.

We present a technical assumption that specializes in the overparameterization setting.
Let $\Pi_\Sigma^{\perp}$ be a projection operator onto a linear space spanned by vectors orthogonal to any eigenvector of $\Sigma$.
\begin{assumption} \label{asmp:design_matrix_z}
$\mathrm{dim}(\Pi_\Sigma^\perp (\mathsf{Y}) ) > n $ holds almost surely.
\end{assumption}
This assumption is identical to Assumption 1 in \citep{bartlett2020benign} and is intended to address cases where no degeneracies exist, such as the perfect {\rc collinearity} between the variables.

With this setting, we bound the risk of the estimator for the overparameterized linear regression model.

\begin{proposition} \label{prop:over_param}
Consider the linear regression model \eqref{def:lin_reg_model} with the process $(X_\ell)_{\ell \in \mathbb{Z}}$ in \eqref{def:lin_reg_latent} being a CBS.
Assume that Assumption  \ref{asmp:design_matrix_z} holds.
Consider the estimator \eqref{def:estimator_regression} and its excess risk \eqref{def:risk_regression}.
Assume, for any $t,\tau > 0$, with probability at least $1-\exp(-t)- \sum_{\ell=1}^n \Surv_\ell(\tau)$, we have
\begin{align*}
    R(\hat{\theta}) &\leq  { 4}\sqrt{2} c \|\theta^*\|^2\left\|\Sigma\right\| \left({\tau+\Gamma_{n}}\right)\sqrt{\frac{4\mathbf{r}\left(\Sigma\right)+t}{n}} + G(\tau) + ct \sigma^2 \mathrm{\rc Tr}(C),
\end{align*}
where $C = (\mathsf{Y} \mathsf{Y}^\top)^{-1} \mathsf{Y} \Sigma \mathsf{Y}^\top (\mathsf{Y} \mathsf{Y}^\top)^{-1}$, $\Surv_\ell(\tau) =  \mathbf{1}_{\{\tau \leq 4B^2\}} +  \Surv_\varepsilon(\tau/4)$, and $G(\tau) = \int_{\tau}^{\infty} \Surv_\ell(t) {\rm d}t$.
\end{proposition}
This result indicates that we can bound the bias term of the risk of the overparameterized linear regression estimator, even in the dependent and heavy-tailed setting. 
The last term $ct \sigma^2 \mathrm{\rc Tr}(C)$ represents the variance of the risk, which converges to zero by removing correlations and controlling for them using different techniques. 
This goes beyond the scope of this paper, see Lemma 11 in \citep{bartlett2020benign} for further details.

\section{Proofs for Main Results in Section~\ref{sec:result}}\label{sec:proof}

\subsection{Outline} \label{sec:proof_outline}

We first state two lemmas at the core of our proofs in Section~\ref{subsub:prelim}. Lemma~\ref{lemma:catoni} appears in many forms in the proofs of the PAC-Bayes bounds~\citep{cat2007,alquier2021user}. For convenience, we use the version stated in~\citep{catoni2017dimension,zhivotovskiy2021dimension}. Lemma~\ref{lemma:rio} is Rio's version of Hoeffding's inequality~\citep{rio2000inegalites} for weakly dependent random variables, that we applied to matrices.

Then, we prove Theorem~\ref{thm:main} in Section~\ref{subsub:proof:main}. We essentially follow the techniques developed in~\citep{catoni2017dimension,zhivotovskiy2021dimension}. However, both these studies rely on exponential inequalities for independent random variables. Therefore, we use Rio's inequality, which requires the boundedness assumption.

In Section~\ref{subsub:truncation}, we introduce a truncation function that transforms unbounded matrices into bounded ones. We thus apply Theorem~\ref{thm:main} to the truncated matrices. We then control the effect of the truncation function to prove Corollary~\ref{cor:standard}.

We mention that we have developed the proof to deal with dependent matrices. 
For the case with independent matrices, the tools developed in \citep{catoni2017dimension,zhivotovskiy2021dimension} can be used.
However, their proofs rely strongly on the independence property, which is why we needed to introduce new arguments for dependent matrices.

\subsection{Preliminary results}
\label{subsub:prelim}

\begin{lemma}[\citep{catoni2017dimension}]
\label{lemma:catoni}
Assume that $X$ is a random variable defined in a measurable space $\left(\mathcal{X},\mathcal{A}\right)$, and $\left(\Theta,\mathcal{F}\right)$ is a measurable parameter space.
Let $\mu$ be a probability measure on $\left(\Theta,\mathcal{F}\right)$ and $h:\mathcal{X}\times\Theta\to\mathbb{R}$ be a real-valued $\mathcal{A}\otimes\mathcal{F}/\mathcal{B}\left(\mathbb{R}\right)$-measurable function such that $\Ep_{X}[\exp\left(h\left(X,\theta\right)\right)]<\infty$ for $\mu$-almost all $\theta$.
It holds that with probability at least $1-\exp(-t)$, for all probability measures $\rho\ll \mu$ simultaneously,
\begin{align*}
    \Ep_{\rho}[h\left(X,\theta\right)]\le \Ep_{\rho}\left[\log\Ep_{X}\left[\exp\left(h\left(X,\theta\right)\right)\right]\right]+ \mathrm{KL}\left(\rho \| \mu\right)+t.
\end{align*}
\end{lemma}

\begin{proof}
The proof is merely a consequence of the duality relationship:
\begin{align*}
    &\Ep_{X}\left[\exp \left\{\sup_{\rho\ll\mu}\left(\Ep_{\rho}\left[h\left(X,\theta\right)-\log\Ep_{X}[\exp (h\left(X,\theta\right))]\right]-\mathrm{KL}\left(\rho \| \mu\right)\right) \right\} \right]\\
    &=\Ep_{X}\Ep_{\mu} \left[\exp\left(h\left(X,\theta\right)-\log\Ep_{X}[\exp (h\left(X,\theta\right))]\right) \right]\\
    &=\Ep_{X}\Ep_{\mu}\left[\frac{\exp ( h\left(X,\theta\right))}{\Ep_{X}[\exp (h\left(X,\theta\right))]}\right]\\
    &=\Ep_{\mu}\Ep_{X}\left[\frac{\exp ( h\left(X,\theta\right))}{\Ep_{X}[\exp (h\left(X,\theta\right))]}\right]\\
    &=1.
\end{align*}
We use Tonelli's theorem to exchange the order of expectations.
Then Markov's inequality leads to the inequality that holds with probability at least $1-\exp(-t)$;
\begin{align*}
    \sup_{\rho\ll\mu}\left(\Ep_{\rho}\left[h\left(X,\theta\right)-\log\Ep_{X}[\exp (h\left(X,\theta\right))]\right]-\mathrm{KL}\left(\rho \|\mu\right)\right)<t.
\end{align*}
This completes the proof.
\end{proof}

\begin{lemma}[Rio's version of Hoeffding's inequality~\citep{rio2000inegalites}, applied to matrices]
\label{lemma:rio}
Let $\left\{M_{1},\ldots,M_{n}\right\}$ be a sequence of positive semi-definite symmetric random matrices with $ \max_{\ell = 1,..,n}\left\|M_{\ell}\right\|\le \kappa^{2}$ almost surely for some $\kappa>0$. 
Let us assume that Assumption~\ref{asm:weakdep} is satisfied. Then, for any function $h\in {\rm Lip}_{n}(E,L) $ and for any $\lambda>0$ we have
\begin{align*}
&\mathbb{E} \left[ \exp\left( \lambda h(M_1,\dots,M_n)- \lambda \mathbb{E}[h(M_1,\dots,M_n)] \right) \right]
\\
&\leq \exp\left(\frac{\lambda^2 L^2 \sum_{\ell=1}^n \left( 2\kappa + \Gamma_{\ell,n} \right)^2 }{8}\right).    
\end{align*}
\end{lemma}

\subsection{Bounded Case (Theorem~\ref{thm:main})}
\label{subsub:proof:main}

\begin{proof}[Proof of Theorem~\ref{thm:main}]
The proof consists of truncation of $\rho=\rho_{u,v}$ given by \citep{zhivotovskiy2021dimension} and the lemma above obtained using duality.

(Step 1) 
Let us assume that $\Sigma$ is invertible. Otherwise, we only need to consider a lower-dimensional subspace, and the proof is similar to the case with invertible $\Sigma$.
Let $\mu$ denote a $2p$-dimensional product measure of two $p$-dimensional Gaussian measures with a zero mean and covariance $\left(2\mathbf{r}\left(\Sigma\right)\right)^{-1}\Sigma$.
We define $\mathbb{S}^{p-1}$ as the unit ball in $\mathbb{R}^{p}$.
Let us set $u,v\in\Sigma^{1/2}\mathbb{S}^{p-1}$ and define $f_{u},f_{v}$ as probability density functions with respect to the Lebesgue measure such that
\begin{align*}
    f_{u}\left(x\right)&=\frac{\exp\left(-\mathbf{r}\left(\Sigma\right)\left(x-u\right)^{\top}\Sigma^{-1}\left(x-u\right)\right)\mathbf{1}\{\left\|x-u\right\|\le \sqrt{\left\|\Sigma\right\|}\}}{ \int \exp\left(-\mathbf{r}\left(\Sigma\right)\left(x'-u\right)^{\top}\Sigma^{-1}\left(x'-u\right)\right)\mathbf{1}\{\left\|x'-u\right\|\le \sqrt{\left\|\Sigma\right\|}\} \mathrm{d}x'},\\
    f_{v}\left(x\right)&=\frac{\exp\left(-\mathbf{r}\left(\Sigma\right)\left(x-u\right)^{\top}\Sigma^{-1}\left(x-v\right)\right)\mathbf{1}\{\left\|x-v\right\|\le \sqrt{\left\|\Sigma\right\|}\}}{ \int \exp\left(-\mathbf{r}\left(\Sigma\right)\left(x'-v\right)^{\top}\Sigma^{-1}\left(x'-v\right)\right)\mathbf{1}\{\left\|x'-v\right\|\le \sqrt{\left\|\Sigma\right\|}\} \mathrm{d}x'}.
\end{align*}
Here, $\mathbf{1}\{\mE\}$ is an indicator function which is $1$ if an event $\mE$ is true and $0$ otherwise.
Assume that the independent random vectors $\theta,\eta$ have densities $f_{u}$ and $f_{v}$.
{\rc Note that} $\Ep[\left(\theta,\eta\right)]=\left(u,v\right)$ by the symmetricity of $f_{u},f_{v}$, and $\max\left\{\left\|\theta\right\|,\left\|\eta\right\|\right\}\le 2\sqrt{\left\|\Sigma\right\|}$ almost surely.
Let $\rho_{u,v}$ be a probability measure of $\theta,\eta$ given by $\rho_{u,v}(\mathrm{d}x,\mathrm{d}y)=f_{u}\left(x\right)f_{v}\left(y\right)\mathrm{d}x\mathrm{d}y,x,y\in\mathbb{R}^{d}$. 
In the proof of Theorem 1 of \citep{zhivotovskiy2021dimension},
it is shown that
\begin{align*}
    \mathrm{KL}\left(\rho_{u,v} \| \mu\right)\le 2\log2+2\mathbf{r}\left(\Sigma\right).
\end{align*}

(Step 2)
Let $f\left(A,\theta,\eta\right):=\theta^{\top}A\eta$ for any $A\in\mathbb{R}^{p}\otimes\mathbb{R}^{p}$ and $\theta,\eta\in\mathbb{R}^{p}$.
Lemma~\ref{lemma:rio} with $h(M_1,\dots,M_n) = \sum_{\ell=1}^n f(M_\ell,\theta,\eta)$ gives that for any $\lambda>0$,
\begin{align*}
    &\Ep_{\mathbf{M}}\left[\exp\left(\lambda\sum_{\ell=1}^{n}f\left(M_{\ell},\theta,\eta\right)\right)\right]\\
    &=\Ep_{\mathbf{M}}\left[\exp\left(\lambda\sum_{\ell=1}^{n}\theta^{\top}M_{\ell}\eta\right)\right]\\
    &\le \exp\left(n\lambda \theta^{\top}\Sigma \eta+\frac{\lambda^{2}\left\|\theta\right\|^{2}\left\|\eta\right\|^{2}{ n}\left(2\kappa^{2}+{ 2}\Gamma_{n}\right)^{2}}{{8}}\right),
    \end{align*}
because for any $A_{1},\ldots,A_{n},B_{1},\ldots,B_{n}\in\mathbb{R}^{p}\otimes\mathbb{R}^{p}$,
\begin{align*}
    \abs{h\left(A_{1},\dots,A_{n}\right)-h\left(B_{1},\dots,B_{n}\right)}&=\bigabs{\theta^{\top}\left(\sum_{\ell=1}^{n}\left(A_{\ell}-B_{\ell}\right)\right)\eta}\\
    &\le \left\|\theta\right\|\left\|\eta\right\|\sum_{\ell=1}^{n}\left\|A_{\ell}-B_{\ell}\right\|.
\end{align*}
It holds that
\begin{align*}
    &\frac{1}{n}\Ep_{\rho_{u,v}}\left[\log\Ep_{\mathbf{M}}\left[\exp\left(\lambda\sum_{\ell=1}^{n}f\left(M_{\ell},\theta,\eta\right)\right)\right]\right]\\
    &\le \frac{1}{n}\Ep_{\rho_{u,v}}\left[n\lambda \theta^{\top}\Sigma \eta+\frac{\lambda^{2}\left\|\theta\right\|^{2}\left\|\eta\right\|^{2}n\left(\kappa^{2}+\Gamma_{n}\right)^{2}}{2}\right]\\
    &=\Ep_{\rho_{u,v}}\left[\lambda \theta^{\top}\Sigma \eta+\frac{\lambda^{2}\left\|\theta\right\|^{2}\left\|\eta\right\|^{2}\left(\kappa^{2}+\Gamma_{n}\right)^{2}}{2}\right]\\
    &\le \lambda u^{\top}\Sigma v+\frac{\lambda^{2}\left(2\sqrt{\left\|\Sigma\right\|}\right)^{4}\left(\kappa^{2}+\Gamma_{n}\right)^{2}}{2}\\
    &=\lambda u^{\top}\Sigma v+8\lambda^{2}\left\|\Sigma\right\|^{2}\left(\kappa^{2}+\Gamma_{n}\right)^{2}.
\end{align*}
The last inequality comes from the fact that $\max\left\{\left\|\theta\right\|,\left\|\eta\right\|\right\}\le 2\sqrt{\left\|\Sigma\right\|}$.
Therefore, from Lemma~\ref{lemma:catoni} with $h\left(M_{1},\ldots,M_{n},\theta,\eta\right)=\lambda\sum_{\ell=1}^{n}f\left(M_{\ell},\theta,\eta\right)$ and the fact that $\log2\le\mathbf{r}\left(\Sigma\right)$ for any $\Sigma$, we obtain
\begin{align*}
    \frac{1}{n}\sum_{\ell=1}^{n}\lambda u^{\top}M_{\ell}v\le \lambda u^{\top}\Sigma v+8\lambda^{2}\left\|\Sigma\right\|^{2}\left(\kappa^{2}+\Gamma_{n}\right)^{2}+\frac{4\mathbf{r}\left(\Sigma\right)+t}{n},
\end{align*}
simultaneously for all $u,v$ with probability at least $1-\exp(-t)$.
By choosing 
\begin{align*}
    \lambda=\sqrt{\frac{4\mathbf{r}\left(\Sigma\right)+t}{8n\left\|\Sigma\right\|^{2}\left(\kappa^{2}+\Gamma_{n}\right)^{2}}},
\end{align*}
we obtain
\begin{align*}
    \left\|\frac{1}{n}\sum_{\ell=1}^{n} M_{\ell}-\Sigma\right\|\le4\sqrt{2}\left\|\Sigma\right\|\left(\kappa^{2}+\Gamma_{n}\right)\sqrt{\frac{4\mathbf{r}\left(\Sigma\right)+t}{n}}.
\end{align*}
This is our claim.
\end{proof}

\subsection{Heavy-Tailed Case (Corollary \ref{cor:standard})}
\label{subsub:truncation}

We first present a truncation function, which is necessary to our robustification strategy for heavy-tailed random matrices.

\begin{definition}
\label{dfn:psi:tau}
For any $\tau>0$, we define the truncation function $\psi_\tau:\mathbb{R}\rightarrow\mathbb{R}$ as follows:
$$ \psi_\tau( x) = \left\{ 
\begin{array}{l}
 -\tau \text{ if } x<\tau, \\
 x \text{ if } {\abs{x}}\leq \tau, \\
 \tau \text{ if } x>\tau. 
\end{array}
\right.
$$
\end{definition}
There is a standard method for extending a real function $\mathbb{R}\rightarrow\mathbb{R}$ to a function of symmetric matrices $\mathcal{S}\rightarrow\mathcal{S}$, by applying the function to the eigenvalues of the matrix. Specifically, given $A\in\mathcal{S}$, $A$ can be written as
$$ A = Q \left(
\begin{array}{c c c}
 \lambda_1 & \dots & 0 \\
 \vdots & \ddots & \vdots \\
 0 & \dots & \lambda_p
\end{array}
\right)  Q^T,
$$
for some matrix $Q$ such that $Q Q^T = I$, where $(\lambda_1,\dots,\lambda_p)$ are the eigenvalues of $A$. We then define $\psi_\tau(A)$ by
$$
\psi_{\tau}(A)
=Q \left(
\begin{array}{c c c}
 \psi_\tau(\lambda_1) & \dots & 0 \\
 \vdots & \ddots & \vdots \\
 0 & \dots & \psi_\tau(\lambda_p)
\end{array}
\right)  Q^T.
$$
We can now state the first corollary of Theorem~\ref{thm:main}.
\begin{corollary}
\label{cor:truncated:1}
Assume that $\left\{M_{1},\ldots,M_{n}\right\}$ satisfies Assumption~\ref{asm:weakdep}. Fix $\tau>0$. Then for all $t>0$, with probability at least $1-\exp(-t)$ it holds that
\begin{align*}
    \left\|\frac{1}{n}\sum_{\ell=1}^{n} ( \psi_\tau(M_{\ell})- \mathbb{E}[\psi_\tau(M_{\ell})] ) \right\|\le{4}\sqrt{2}\left\|\Sigma\right\|\left({\tau+\Gamma_{n}}\right)\sqrt{\frac{4\mathbf{r}\left(\Sigma\right)+t}{n}}  .
\end{align*}
\end{corollary}
\begin{proof}[Proof of Corollary~\ref{cor:truncated:1}]
Because $x\mapsto \psi_\tau(x)$ is $1$-Lipschitz, the sequence of matrices $\{\psi_\tau(M_1)-\mathbb{E}[\psi_\tau(M_1)],\dots,\psi_\tau(M_n) -\mathbb{E}[\psi_\tau(M_n)] \}$ satisfies Assumption~\ref{asm:weakdep}. As they are bounded by $\tau$ and all have the same expectation (zero), therefore, we can apply Theorem~\ref{thm:main} to yield the result.
\end{proof}
As stated in the outline of the proof, we now have to understand the difference between the expectation of the truncated matrices and the expectations of the (non-truncated) matrices themselves.
\begin{proposition}
\label{prop:trunc:exp}
Fix $\tau > 0$.
Under Assumption~\ref{asm:tailbound}, we have
$$ \max_{1 \leq \ell \leq n}\left\| \mathbb{E}[\psi_\tau(M_\ell)] - \mathbb{E}[M_\ell] \right\| \leq G(\tau). $$
\end{proposition}
\begin{proof}[Proof of Proposition~\ref{prop:trunc:exp}]
For any $\ell = 1,\ldots,n$, we have
\begin{align*}
    \left\| \mathbb{E}[\psi_\tau(M_\ell)] - \mathbb{E}[M_\ell] \right\|
    & \leq \mathbb{E} [\left\| \psi_\tau(M_\ell)-M_\ell \right\|]
    \\
    & = \mathbb{E} [(\|M_\ell\|-\tau)\mathbf{1}_{\{\|M_\ell\|-\tau>0\}}]
    \\
    & \leq \int_{0}^{\infty} (u-\tau) \mathbf{1}_{\{u-\tau>0\}} {\rm d} \mathbb{P}(\|M_\ell\|-\tau\leq u)
    \\
    & = - \int_{\tau}^{\infty} (u-\tau) {\rm d} \mathbb{P}(\|M_\ell\|-\tau > u)
    \\
    & = \left[ - (u-\tau) \mathbb{P}(\|M_\ell\|-\tau > u)  \right]_{\tau}^{\infty} 
    \\
    & \quad \quad + \int_{\tau}^{\infty}  \mathbb{P}(\|M_\ell\|-\tau> u) du.
\end{align*}
The first term is null as {\rc $x \Surv_{\ell}(x)=\int_{0}^{\infty}\mathbf{1}_{\{u\le x\}}\Surv_{\ell}(x) du\le \int_{0}^{\infty}\Surv_{\ell}(u)du$ and the dominated convergence theorem gives $x \Surv_{\ell}(x)\to0$ as $x\to\infty$}, and thus
\begin{align*}
    \left\| \mathbb{E}[\psi_\tau(M_\ell)] - \mathbb{E}[M_\ell] \right\|
    & \leq \int_{\tau}^{\infty}  \mathbb{P}(\|M_\ell\|-\tau> u) du
    \\
    & \leq \int_{\tau}^{\infty}  \mathbb{P}(\|M_\ell\|> u) du
    \\
    & = \int_{\tau}^{\infty} \Surv_{\ell}(u) du .
\end{align*}
which ends the proof.
\end{proof}
\begin{corollary}
\label{cor:trunc:2}
Assume that $\{M_1,\dots,M_n\}$ satisfies Assumptions~\ref{asm:weakdep} and~\ref{asm:tailbound}, and $\mathbb{E}[M_\ell ]=\Sigma$. Fix $\tau>0$. For all $t>0$, with probability at least $1-\exp(-t)$ it holds that
\begin{align*}
    \left\|\frac{1}{n}\sum_{\ell=1}^{n} \psi_\tau(M_{\ell})-\Sigma\right\|\le{4}\sqrt{2}\left\|\Sigma\right\|\left({\tau+\Gamma_{n}}\right)\sqrt{\frac{4\mathbf{r}\left(\Sigma\right)+t}{n}} + G(\tau).
\end{align*}
\end{corollary}
\begin{proof}[Proof of Corollary~\ref{cor:trunc:2}]
First, we decompose the norm as
\begin{align*}
 &\left\|\frac{1}{n}\sum_{\ell=1}^{n} \psi_\tau(M_{\ell})-\Sigma\right\|\\
 & \leq  \left\|\frac{1}{n}\sum_{\ell=1}^{n} (\psi_\tau(M_{\ell})- \mathbb{E}[\psi_\tau(M_{\ell})] )\right\| +  \left\|\frac{1}{n}\sum_{\ell=1}^{n} (\mathbb{E}[\psi_\tau(M_{\ell})] - \mathbb{E}[M_{\ell} ])\right\|
 \\
 & \leq  \left\|\frac{1}{n}\sum_{\ell=1}^{n} [\psi_\tau(M_{\ell})- \mathbb{E}[\psi_\tau(M_{\ell})] ]\right\| + \frac{1}{n}\sum_{\ell=1}^{n} \left\| \mathbb{E}[\psi_\tau(M_{\ell})] - \Sigma \right\|,
\end{align*}
where we use the triangle inequality in the first line, and Jensen's inequality and $\mathbb{E}[M_{\ell} ]=\Sigma$ in the second line. As Assumption~\ref{asm:weakdep} is satisfied, we can upper bound the first term with probability $1-\exp(-t)$ by Corollary~\ref{cor:truncated:1} together with Proposition \ref{proposition:lipschitz:dependence}. Because Assumption~\ref{asm:tailbound} is also satisfied, we can bound the second term using Proposition~\ref{prop:trunc:exp}. 
\end{proof}

Note that Corollary~\ref{cor:trunc:2} already provides an estimation result for $\Sigma$ when matrices $M_\ell$ are unbounded. However, in contrast to Corollary~\ref{cor:standard}, not only does the bound depend on $\tau$ but the estimator $\frac{1}{n}\sum_{\ell=1}^{n} \psi_\tau(M_{\ell})$ does as well. A mistake in the choice of $\tau$ can lead to poor estimation in practice.

To control the distance between this estimator $\frac{1}{n}\sum_{\ell=1}^{n} \psi_\tau(M_{\ell})$ and the standard estimator $\frac{1}{n}\sum_{\ell=1}^{n} M_\ell$, we prove the following proposition.
 \begin{proposition}
 \label{prop:trunc:estimator}
 Under Assumption~\ref{asm:tailbound}, we have
 $$ \mathbb{P}\left( \left\| \frac{1}{n}\sum_{\ell=1}^n \psi_\tau(M_\ell) - \frac{1}{n}\sum_{\ell=1}^n M_\ell \right\| \neq 0 \right) \leq \sum_{\ell=1}^n \Surv_\ell(t) .$$
 \end{proposition}
\begin{proof}[Proof of Proposition~\ref{prop:trunc:estimator}]
We have
\begin{align*}
    &\mathbb{P}\left( \left\| \frac{1}{n}\sum_{\ell=1}^n \psi_\tau(M_\ell) - \frac{1}{n}\sum_{\ell=1}^n M_\ell \right\| \neq 0 \right)
    \\&
    \leq
        \mathbb{P}\left(  \frac{1}{n}\sum_{\ell=1}^n\left\| \psi_\tau(M_\ell) - M_\ell \right\| > 0 \right)
    \\
& = \mathbb{P}\left( \exists \ell: \left\| \psi_\tau(M_\ell) - M_\ell \right\| > 0 \right)
\\
& \leq \sum_{\ell=1}^n \Surv_\ell(t).
\end{align*}
\end{proof}
We can now prove Corollary~\ref{cor:standard}.
\begin{proof}[Proof of Corollary~\ref{cor:standard}]
Using the triangle inequality,
$$
 \left\|\frac{1}{n}\sum_{\ell=1}^{n} M_{\ell}-\Sigma\right\|
 \leq  \left\|\frac{1}{n}\sum_{\ell=1}^{n} M_\ell - \frac{1}{n}\sum_{\ell=1}^{n} \psi_\tau(M_\ell) \right\| +  \left\|\frac{1}{n}\sum_{\ell=1}^{n} \psi_\tau(M_{\ell}) - \Sigma \right\|.
$$
The assumptions of Corollary~\ref{cor:standard} include: $\{M_1,\dots,M_n\}$ satisfy Assumptions~\ref{asm:weakdep} and~\ref{asm:tailbound}, and $\mathbb{E}[M_\ell]=\Sigma$, which enables us to use Corollary~\ref{cor:trunc:2} to upper bound the second term with probability $1-\exp(-t)$. This also allows for the use of Proposition~\ref{prop:trunc:estimator} to prove that the first term will be null with probability at least $1-\sum_{\ell=1}^n \Surv_\ell(\tau)$.
\end{proof}

\subsection{Infinite-Dimensional Case (Theorem \ref{thm:operator})}

\begin{proof}[Proof of Theorem \ref{thm:operator}]

For a sequence of $\mathbb{H}\otimes\mathbb{H}$-valued positive semi-definite symmetric random operators $\left\{M_{1},\ldots,M_{n}\right\}$ with $\Ep\left[M_{\ell}\right]=\Sigma$ and $\max_{1 \leq \ell \leq n}\left\|M_{\ell}\right\|\le \kappa^{2}$ almost surely for some $\kappa>0$ satisfying Assumption~\ref{asm:weakdep},
\begin{align*}
    &P\left(\sup_{\substack{u_{k}\in \mathbb{H}_{k}\\\left\|u_{k}\right\|=1}}\bigabs{\left\langle \left(\frac{1}{n}\sum_{\ell=1}^{n}M_{\ell}-\Sigma\right)u_{k},u_{k}\right\rangle }\ge {4}\sqrt{2}\left\|\Sigma\right\|\left({\kappa^{2}+\Gamma_{n}}\right)\sqrt{\frac{4\mathbf{r}\left(\Sigma\right)+t}{n}}\right)\\
    &\le \exp(-t),
\end{align*}
because for $\Sigma_{k}$ such that $\Sigma_{k}^{\left(j_{1},j_{2}\right)}:=\Ep[M_{\ell}^{\left(j_{1},j_{2}\right)}]$ and $\Sigma:=\Ep[M_{\ell}]$, $\left\|\Sigma_{k}\right\|\le \left\|\Sigma\right\|$ and $\mathrm{\rc Tr}\left(\Sigma_{k}\right)\le \mathrm{\rc Tr}\left(\Sigma\right)$, and $\Gamma_{n}$ is also uniform for each, as evident from the proof.
Note that for any $c\ge 0$ and $k\in\mathbb{N}$,
\begin{align*}
    &\left\{\sup_{\substack{u_{k}\in \mathbb{H}_{k}\\\left\|u_{k}\right\|=1}}\bigabs{\left\langle \left(\frac{1}{n}\sum_{\ell=1}^{n}M_{\ell}-\Sigma\right)u_{k},u_{k}\right\rangle }\ge c\right\} \\
    &\subset \left\{\sup_{\substack{u_{k+1}\in \mathbb{H}_{k+1}\\\left\|u_{k+1}\right\|=1}}\bigabs{\left\langle \left(\frac{1}{n}\sum_{\ell=1}^{n}M_{\ell}-\Sigma\right)u_{k+1},u_{k+1}\right\rangle } \geq c\right\},
\end{align*}
and 
\begin{align*}
    \lim_{k\to\infty}\left\{\sup_{\substack{u_{k}\in \mathbb{H}_{k}\\\left\|u_{k}\right\|=1}}\bigabs{\left\langle \left(\frac{1}{n}\sum_{\ell=1}^{n}M_{\ell}-\Sigma\right)u_{k},u_{k}\right\rangle }\ge c\right\}=\left\{\left\|\frac{1}{n}\sum_{\ell=1}^{n}M_{\ell}-\Sigma\right\|\ge c\right\}.
\end{align*}
The continuity of $P$ leads to
\begin{align*}
    P\left(\left\|\frac{1}{n}\sum_{\ell=1}^{n}M_{\ell}-\Sigma\right\|\ge {4}\sqrt{2}\left\|\Sigma\right\|\left({ \kappa^{2}+\Gamma_{n}}\right)\sqrt{\frac{4\mathbf{r}\left(\Sigma\right)+t}{n}}\right)\le \exp(-t).
\end{align*}
Then, using the same approach to extend Theorem \ref{thm:main} to Corollary \ref{cor:standard}, we obtain the statement.
\end{proof}

\section{Conclusion} \label{sec:conclusion}

We studied the deviations of the empirical mean of random matrices from its expected value in the dependent, heavy-tailed case. 
The upper bound derived here is independent of the dimension of the matrices but depends on the trace of the expectation and the tail of the distribution.
Additionally, the upper bound increases with the strength of the dependence between the matrices.
The proof here is based on a variational inequality and robustification by truncation. 
Our result is applied to the estimation problem of covariance operators/matrices, parameter estimation in linear hidden Markov models, and linear regression under overparameterization.

A limitation of our result is the tightness of the obtained upper bound. 
It is difficult to achieve lower bounds when random matrices are dependent and heavy-tailed, while some lower bounds are known when they are independent and each element is Gaussian.
Therefore, deriving lower bounds in this case is an interesting subject for future research.









\appendix

\section{Proof for Examples}
\label{subsec:proofs-for-example}

\begin{proof}[Proof of Proposition~\ref{proposition:lipschitz:dependence}]

Let $f:E\rightarrow E$ be a $1$-Lipschitz function and define $\mathcal{G}_\ell=\sigma(f(M_1),\dots,f(M_\ell))$. We aim to prove that, for any $g\in {\rm Lip}_{n-\ell}(\mathcal{S},1)$, we have
\begin{equation}
\abs{ \mathbb{E}[g(f(M_{\ell+1}),\dots,f(M_n)) \mid \mathcal{G}_\ell] - \mathbb{E}[g(f(M_{\ell+1},\dots,f(M_n))] } \leq \Gamma_{\ell,n}.
\label{equa:intermediaire:preuve:1}
\end{equation}
Let $h$ be defined by $h(a_1,\dots,a_\ell) = g(f(a_1),\dots,f(a_\ell))$. Then $h\in {\rm Lip}_{n-\ell}(\mathcal{S},1)$. Indeed,
\begin{align*}
 & \abs{    h(a_1,\dots,a_\ell) -  h(b_1,\dots,b_\ell)  } \\
 & = \abs{  g(f(a_1),\dots,f(a_\ell)) -  g(f(b_1),\dots,f(b_\ell)) }\\
 & \leq L \sum_{i=1}^\ell \|f(a_i) - f(b_i)\|_E \\
 & \leq L \sum_{i=1}^\ell \|a_i - b_i\|_E,
\end{align*}
where we used respectively the definition of $h$, the fact that $g\in {\rm Lip}_{n-\ell}(\mathcal{S},1)$ and the fact that $f$ is $1$-Lipschitz. Thus, because $(M_1,\dots,M_n)$ satisfies Assumption~\ref{asm:weakdep} and $h\in {\rm Lip}_{n-\ell}(\mathcal{S},1)$, then
    \begin{equation*}
\abs{ \mathbb{E}[h(M_{\ell+1},\dots,M_n) \mid \mathcal{F}_\ell] - \mathbb{E}[h(M_{\ell+1},\dots,M_n)] } \leq \Gamma_{\ell,n}
\end{equation*}
that we can rewrite as
    \begin{equation}
\abs{ \mathbb{E}[g(f(M_{\ell+1}),\dots,f(M_n)) \mid \mathcal{F}_\ell] - \mathbb{E}[g(f(M_{\ell+1},\dots,f(M_n))] } \leq \Gamma_{\ell,n}.
\label{equa:intermediaire:preuve:2}
\end{equation}
This is almost~\eqref{equa:intermediaire:preuve:1}; however, the conditional expectation does not hold with respect to the correct $\sigma$-algebra. This is easily fixed because $\mathcal{G}_\ell \subseteq \mathcal{F}_\ell$. Thus,
\begin{align*}
& \abs{ \mathbb{E}[g(f(M_{\ell+1}),\dots,f(M_n)) \mid \mathcal{G}_\ell] - \mathbb{E}[g(f(M_{\ell+1},\dots,f(M_n))] } \\
& =
\abs{ \mathbb{E}[ \mathbb{E}[g(f(M_{\ell+1}),\dots,f(M_n)) \mid \mathcal{F}_\ell] \mid \mathcal{G}_\ell] - \mathbb{E}[g(f(M_{\ell+1},\dots,f(M_n))] }
\\
&
\leq
 \mathbb{E}\left[  \abs{ \mathbb{E}[g(f(M_{\ell+1}),\dots,f(M_n))\mid \mathcal{F}_\ell]- \mathbb{E}[g(f(M_{\ell+1},\dots,f(M_n))] } \mid \mathcal{G}_\ell \right]
\\
&
\leq \Gamma_{\ell,n},
\end{align*}
by using~\eqref{equa:intermediaire:preuve:2}.
\end{proof}

\begin{proof}[Proof of Proposition~\ref{prop:assumptions}]
We define $(\bar{\xi}_\ell)_{\ell\in\mathbb{Z}}$ as an independent copy $\Xi$. 
We fix $\ell\in\{1,\dots,n\}$; we verify~\eqref{eq:asm:weakdep}. To do so, we define, for $m>\ell$,
$$ \bar{X}_m = C(\xi_m,\xi_{m-1},\dots,\xi_{\ell+1},\bar{\xi}_{\ell},\bar{\xi}_{\ell-1},\bar{\xi}_{\ell-2},\dots), $$
and $\bar{Y}_m = \bar{X}_m + \varepsilon_m $.
We put $\mathcal{G}_{\ell} = \sigma(\xi_{\ell},\xi_{\ell-1},\xi_{\ell-2},\dots; \varepsilon_\ell,\varepsilon_{\ell-1},\dots) $. Then, for $g\in{\rm Lip}_{n-\ell}(\mathcal{S},1)$, we have
\begin{align*}
 & \mathbb{E}[g(M_{\ell+1},\dots,M_n)\mid \mathcal{F}_\ell] - \mathbb{E}[g(M_{\ell+1},\dots,M_n)] 
 \\
 & = \mathbb{E}[\mathbb{E}[g(M_{\ell+1},\dots,M_n)\mid \mathcal{G}_\ell] - \mathbb{E}[g(M_{\ell+1},\dots,M_n)] \mid \mathcal{F}_\ell],
\end{align*}
and we prove an upper bound on $\mathbb{E}[g(M_{\ell+1},\dots,M_n) \mid \mathcal{G}_\ell] - \mathbb{E}[g(M_{\ell+1},\dots,M_n)]$. 
Hence, we have
\begin{align*}
 & \mathbb{E}[g(M_{\ell+1},\dots,M_n)\mid \mathcal{G}_\ell] - \mathbb{E}[g(M_{\ell+1},\dots,M_n)]
 \\
 & = \mathbb{E}[g(\bar{Y}_{\ell+1} \bar{Y}_{\ell+1}^\top ,\dots,\bar{Y}_n \bar{Y}_{n}^\top) - g(Y_{\ell+1} Y_{\ell+1}^\top ,\dots,Y_n Y_{n}^\top)\mid \mathcal{G}_\ell]
 \\
 & \leq \sum_{m=\ell+1}^n \left\| \mathbb{E}\left[ \bar{Y}_m \bar{Y}_{m}^\top - Y_m Y_{m}^\top   \mid \mathcal{G}_{\ell}\right]   \right\|
 \\
 & = \sum_{m=\ell+1}^n \left\| \mathbb{E}\left[ (\bar{X}_m + \varepsilon_m) (\bar{X}_{m} + \varepsilon_{m})^\top - (X_m + \varepsilon_m) ( X_{m} + \varepsilon_{m})^\top  \mid \mathcal{G}_{\ell}\right]   \right\|
 \\
 & =  \sum_{m=\ell+1}^n \left\| \mathbb{E}\left[ \bar{X}_m \bar{X}_{m}^\top - X_m X_{m}^\top  \mid \mathcal{G}_{\ell}\right]   \right\|
 \\
 & = \sum_{m=\ell+1}^n \left\| \mathbb{E}\left[ \bar{X}_m \bar{X}_{m}^\top -  \bar{X}_m X_{m}^\top   + \bar{X}_m X_{m}^\top   - X_m X_{m}^\top  \mid \mathcal{G}_{\ell}\right]   \right\|
 \\
 & \leq \sum_{m=\ell+1}^n  \biggl( \left\| \mathbb{E}\left[ \bar{X}_m \bar{X}_{m}^\top -  \bar{X}_m X_{m}^\top  \mid \mathcal{G}_{\ell}\right]   \right\| + \left\| \mathbb{E}\left[  \bar{X}_m X_{m}^\top   - X_m X_{m}^\top  \mid \mathcal{G}_{\ell}\right]   \right\| \biggr)
 \\
 & \leq \sum_{m=\ell+1}^n  B \biggl( \left\| \mathbb{E}\left[  \bar{X}_{m}^\top - X_{m}^\top  \mid \mathcal{G}_{\ell}\right]   \right\| + \left\| \mathbb{E}\left[  \bar{X}_m   - X_m   \mid \mathcal{G}_{\ell}\right]   \right\| \biggr).
\end{align*}
Then, we obtain
\begin{align*}
& \left\| \mathbb{E}\left[  \bar{X}_m   - X_m   \mid \mathcal{G}_{\ell}\right]  \right\|
\\
& = {\rc\left\| \mathbb{E}\left[ C(\xi_m,\dots,\xi_{\ell+1},\bar{\xi_\ell},\bar{\xi}_{\ell-1},\dots)- C(\xi_m,\dots,\xi_{\ell+1},\xi_\ell,\xi_{\ell-1},\dots)  \mid \mathcal{G}_{\ell}\right]  \right\|}
\\
& \leq \sum_{i=m-\ell}^\infty \alpha_i \mathbb{E}\left[  \|\bar{\xi}_{m-i} - \xi_{m-i} \|  \mid \mathcal{G}_{\ell}\right] 
 \leq 2 \sum_{i=m-\ell}^\infty \alpha_i B_\xi,
\end{align*}
and thus,
\begin{align*}
  \mathbb{E}[g(M_{\ell+1},\dots,M_n)\mid \mathcal{G}_\ell] - \mathbb{E}[g(M_{\ell+1},\dots,M_n)]
  & \leq \sum_{m=\ell+1}^n \left[  4 B \sum_{i=m-\ell}^\infty \alpha_i B_\xi \right]
  \\
  & \leq 4 B \sum_{i=\ell+1}^\infty \min(i,n) \alpha_i B_\xi.
\end{align*}
Thus,~\eqref{eq:asm:weakdep} is satisfied with $\Gamma_{\ell,n} =  4 B B_\xi \sum_{i=\ell+1}^\infty \min(i,n) \alpha_i $.
Let us now verify Assumption~\ref{asm:tailbound}. We have
\begin{align*}
\mathbb{P}(\|M_\ell\| \geq t )
&
= \mathbb{P}(\|(X_\ell+\varepsilon_{\ell})(X_{\ell}+\varepsilon_{\ell})^\top \| \geq t )
\\
& = \mathbb{P}(\|X_\ell+\varepsilon_{\ell} \|^2   \geq t )
\\
& = \mathbb{P}(\|X_\ell+\varepsilon_{\ell} \|  \geq \sqrt{t} ) 
\\
&
\leq  \mathbb{P}(\|X_\ell\| \geq \sqrt{t}/2 ) +  \mathbb{P}(\|\varepsilon_\ell\| \geq \sqrt{t}/2 )
\\
& \leq  \mathbf{1}_{\{t\leq 4B^2\}} +  \mathbb{P}(\|\varepsilon_\ell\| \geq \sqrt{t}/2 ),
\end{align*}
which ends the proof.
\end{proof}

\begin{proof}[Proof of Proposition \ref{prop:lag}]
Since $(X_\ell)_{\ell \in \mathbb{N}}$ is a CBS, {\rc $X_\ell = C(\xi_{\ell},\xi_{\ell-1},\xi_{\ell-2},\dots) $} with
$$ \| C(a_1,a_2,\dots) - C(b_1,b_2,\dots) \| \leq \sum_{\ell=1}^{\infty} \alpha_\ell \|a_\ell-b_\ell \| \text{ and } \mathcal{A} := \sum_{\ell=1}^\infty \alpha_\ell < \infty. $$
Using the form, we show that $ (\tilde{X}_\ell)_{\ell \in \mathbb{N}} := ((X_\ell, X_{\ell + 1})^\top )_{\ell \in \mathbb{N}}$ is also a CBS, since we have
$$ \tilde{X}_\ell = (C(\xi_{t},\xi_{t-1},\xi_{t-2},\dots),C(\xi_{t-1},\xi_{t-2},\xi_{t-3},\dots)) = D(\xi_t,\xi_{t-1},\xi_{t-2},\dots) $$
with some function $D$ which satisfies
$$ \| D(a_1,a_2,\dots) - D(b_1,b_2,\dots) \| \leq \sum_{\ell=1}^{\infty} (\alpha_\ell + \alpha_{\ell+1}) \|a_\ell-b_\ell \| .$$
Since $(\tilde{X}_\ell)_{\ell \in \mathbb{N}} $ is a CBS, $(M_1,\ldots,M_{\ell - 1})$ satisfies Assumption~\ref{asm:weakdep} with $\Gamma_{\ell,n} = 8 B B_\xi \sum_{i=\ell+1}^n \min(i,n) \alpha_i  $ and $\Gamma_n :=  8 B B_\xi \sum_{i=2}^n \min(i,n) \alpha_i  $ by Proposition~\ref{prop:assumptions}.

For Assumption~\ref{asm:tailbound}, we utilize the fact that the largest eigenvalue of a matrix is no more than a sum of the largest eigenvalues of its submatrices and obtain
\begin{align*}
&\mathbb{P}(\|M_\ell\| \geq t ) \\
&\leq \mathbb{P}(\|(X_\ell+\varepsilon_{\ell})(X_{\ell}+\varepsilon_{\ell})^\top \| \geq t /4) \\
& \quad + \Pr( \|(X_\ell+\varepsilon_{\ell})(X_{\ell+1}+\varepsilon_{\ell+1})^\top \| \geq t /2 ) \\
& \quad  +  \Pr (\|(X_{\ell + 1}+\varepsilon_{\ell + 1})(X_{\ell + 1}+\varepsilon_{\ell + 1})^\top \|\geq t /4 )
\\
& = 2 \mathbb{P}(\|X_\ell+\varepsilon_{\ell} \|^2   \geq t/4 ) +  \Pr( \|(X_\ell+\varepsilon_{\ell})(X_{\ell+1}+\varepsilon_{\ell+1})^\top \| \geq t /2 )
\\
& \leq  2 \mathbb{P}(\|X_\ell+\varepsilon_{\ell} \|  \geq \sqrt{t} / 2 ) + \Pr( \|X_\ell+\varepsilon_{\ell} \| \|X_{\ell+1}+\varepsilon_{\ell+1} \| \geq t /2 )
\\
& \leq  2 \mathbb{P}(\|X_\ell+\varepsilon_{\ell} \|  \geq \sqrt{t} / 2 ) + 2\Pr( \|X_\ell+\varepsilon_{\ell} \|  \geq \sqrt{t /2} )
\\
&
\leq 4 \mathbb{P}(\|X_\ell\| \geq \sqrt{t/2} ) + 4 \mathbb{P}(\|\varepsilon_\ell\| \geq \sqrt{t/2} )
\\
& \leq 4 \mathbf{1}_{\{t\leq 4B^2\}} + 4 \mathbb{P}(\|\varepsilon_\ell\| \geq \sqrt{t/2} ).
\end{align*}
Hence, Assumption \ref{asm:tailbound} holds $G(\tau)=4 \mathbf{1}_{\{t\leq 2B^2\}} + 4 \mathbb{P}(\|\varepsilon_\ell\| \geq \sqrt{t/2} )$.
Thus, Corollary \ref{cor:standard} shows the first statement.

Finally, the fact $\left\|\Sigma_{0:1}\right\| \leq \|\Sigma_1\| + \|\Sigma\| $ yields the second statement.
\end{proof}

\begin{proof}[Proof of Proposition \ref{prop:hmm}]
First, we confirm that $(X_\ell)_{\ell \in \mathbb{Z}}$ is a CBS in Example \ref{exm:cbs} by its definition.
Hence, by Proposition \ref{prop:cbs}, a sequence of matrices generated by $Y_\ell Y_\ell^\top$ satisfies Assumption \ref{asm:weakdep} and \ref{asm:tailbound}.

We show that the estimation error $\|\hat{A} - A\|$ is bounded by the estimation error of $\hat{\Sigma}$ and $\hat{\Sigma}_1$.
We bound the error as
\begin{align}
    \|\hat{A} - A\| &= \| (\hat{\Sigma}_{1} - \Sigma_{1})  (\hat{\Sigma} +  I)^{-1} + \Sigma_{Y,1} ((\hat{\Sigma} +  I)^{-1} - ({\Sigma} +  I)^{-1})\| \notag \\
    & \leq \|\hat{\Sigma}_{1} - \Sigma_{1}\| \| (\hat{\Sigma} +  I)^{-1}\| + \| \Sigma_{1}({\Sigma} +  I)^{-1} ( \Sigma -  \hat{\Sigma}) (\hat{\Sigma} +  I)^{-1} \| \notag  \\
    & \leq \|\hat{\Sigma}_{1} - \Sigma_{1}\| +  \|\Sigma -  \hat{\Sigma}\| \|\Sigma_{1} \|. \label{ineq:ahat}
\end{align}
Here, we use the facts $\|(\hat{\Sigma} +  I)^{-1}\| \leq 1$ and $\|({\Sigma} +  I)^{-1}\| \leq 1$.

We combine the above results.
Using the same discussion for Proposition \ref{prop:lag} in Section \ref{sec:lagged_covariance}, we have 
\begin{align*}
    &\max\{\|\hat{\Sigma}-\Sigma\|, \| \hat{\Sigma}_1 - \Sigma_1\|\}\\    
    &\le 4\sqrt{2}\left(\|\Sigma_1\| + \|\Sigma\|\right) \left(\tau+\Gamma_{n}\right)\sqrt{\frac{4 \mathbf{r}\left(\Sigma_{0:1}\right)  +t}{n}} +  G(\tau),
\end{align*}
where the definition of $\Sigma_{0:1}$ follows Section \ref{sec:lagged_covariance}.
We combine this inequality with the result \eqref{ineq:ahat}, and we obtain the statement.
\end{proof}

\begin{proof}[Proof of Proposition \ref{prop:over_param}]
By Lemma 7 in \citep{bartlett2020benign}, the risk $R(\hat{\theta})$ is evaluated as
\begin{align*}
    R(\hat{\theta}) &\leq 2 (\theta^*)^\top (I - \Pi_\mathsf{Y}) \Sigma (I - \Pi_\mathsf{Y})\theta^* + \sigma^2 \mathrm{\rc Tr}((\mathsf{Y} \mathsf{Y}^\top)^{-1} \mathsf{Y} \Sigma \mathsf{Y}^\top (\mathsf{Y} \mathsf{Y}^\top)^{-1}) \\
     & = 2 (\theta^*)^\top B \theta^* + c t \sigma^2 \mathrm{\rc Tr}(C),
\end{align*}
where $B= (I - \Pi_\mathsf{Y}) \Sigma (I - \Pi_\mathsf{Y})$.
We bound the first term as
\begin{align*}
     (\theta^*)^\top B \theta^* &=  (\theta^*)^\top (I - \Pi_\mathsf{Y}) \Sigma (I - \Pi_\mathsf{Y}) \theta^*  \\
     &=(\theta^*)^\top (I - \Pi_\mathsf{Y}) (\Sigma - n^{-1} \mathsf{Y}^\top \mathsf{Y}) (I - \Pi_\mathsf{Y}) \theta^*\\
     & \leq \|\theta^*\|^2 \| I - \Pi_\mathsf{Y}\| \|\Sigma - n^{-1} \mathsf{Y} ^\top \mathsf{Y}\| \\
     & \leq  \|\theta^*\|^2\|\Sigma - n^{-1} \mathsf{Y} ^\top \mathsf{Y}\|,
\end{align*}
where the second equality follows $(I - \Pi_\mathsf{Y})\mathsf{Y}^\top =  (I - \mathsf{Y}^\top (\mathsf{Y} \mathsf{Y}^\top)^{-1} \mathsf{Y}) \mathsf{Y}^\top = \mathsf{Y}^\top  - \mathsf{Y}^\top (\mathsf{Y} \mathsf{Y}^\top)^{-1} (\mathsf{Y} \mathsf{Y}^\top) = 0$, and the second inequality follows $\| I - \Pi_\mathsf{Y}\| \leq 1$ from the non-expansive property of projection operators.
Recalling that $n^{-1} \mathsf{Y}^\top \mathsf{Y} = \hat{\Sigma}$ as in \eqref{def:cov_operator}, Proposition \ref{prop:covariance} yields the statement.
\end{proof}

\putbib[main]
\end{bibunit}

\end{document}